\documentclass[twoside,12pt,reqno]{amsart}
\usepackage{a4wide}
\usepackage[utf8]{inputenc}
\usepackage[T1]{fontenc}
\usepackage[english]{babel}
\usepackage{tgpagella}
\usepackage{hyperref}
\usepackage{amscd} 
\usepackage{csquotes}
\usepackage{fancyhdr}
\usepackage{multirow}
\usepackage{cleveref}
\usepackage{comment}
\usepackage{pgf,tikz}
\usepackage{tikz-cd,rotating}
\usepackage{amsthm, amsfonts, amssymb, amsmath, latexsym, enumerate,array,color, xcolor}

\usepackage{amscd} 
\usepackage[all]{xy}
\usepackage{pstricks,graphicx}
\usepackage{stmaryrd}
\usepackage{hyperref,mdwlist,mathrsfs}
\usepackage{tikz}
\usepackage{pgf,tikz}
\usetikzlibrary{arrows}
\usepackage{paralist}

\newtheorem{theorem}{Theorem}[section]

\newtheorem{lemma}[theorem]{Lemma}
\newtheorem{proposition}[theorem]{Proposition}
\newtheorem{corollary}[theorem]{Corollary}
\newtheorem*{proposition*}{Proposition}
\theoremstyle{definition}
\newtheorem{definition}[theorem]{Definition}
\newtheorem{remark}[theorem]{Remark}

\newtheorem{example}[theorem]{Example}

\newcommand{\cB}{\ensuremath{\mathcal{B}}}

\DeclareMathOperator{\codim}{codim}

\DeclareMathOperator{\im}{im}
\DeclareMathOperator{\Jac}{Jac} 
\DeclareMathOperator{\Proj}{Proj}
\DeclareMathOperator{\Sing}{Sing}

\DeclareMathOperator{\reg}{reg}
\DeclareMathOperator{\Syz}{Syz}
\DeclareMathOperator{\Tor}{Tor}

\newcommand{\ZZ}{ \ensuremath{\mathbb{Z}}}
\newcommand{\PP}{ \ensuremath{\mathbb{P}}}
\newcommand{\RR}{ \ensuremath{\mathbb{R}}}
\newcommand{\CC}{ \ensuremath{\mathbb{C}}}
\newcommand{\fa}{\mathfrak{a}}
\newcommand{\fb}{\mathfrak{b}}
\newcommand{\fc}{\mathfrak{c}}
\newcommand{\fq}{\mathfrak{q}}
\newcommand{\fm}{\mathfrak{m}}
\newcommand{\cA}{ \ensuremath{\mathcal{A}}}
\newcommand{\cL}{ \ensuremath{\mathcal{L}}}
\newcommand{\cK}{ \ensuremath{\mathcal{K}}}

\newcommand{\ffi}{\varphi}

\numberwithin{equation}{section}

\title{Jacobian Ideals of Hyperplane Arrangements and their Graded Betti Numbers}

\author[]{Juan C. Migliore}
\address{Juan C. Migliore: Department of Mathematics, University of Notre  Dame, Notre Dame, IN 46556, USA}
\email{migliore.1@nd.edu}

\author[]{Uwe Nagel} 
\address{Uwe Nagel: Department of Mathematics, University of Kentucky, 715 Patterson Office Tower, Lexington, KY 40506-0027, USA}
\email{uwe.nagel@uky.edu}

\thanks{\hspace{-15pt} Migliore was partially supported by Simons Foundation grant \#839618,}

\thanks{\hspace{-15pt} Nagel was partially supported by Simons Foundation grant \#636513}

\theoremstyle{definition}


\begin{document}

\begin{abstract}

A  hyperplane arrangement $\cA$ is said to be {\it free} if the corresponding Jacobian ideal $J_\cA$ is Cohen-Macaulay. In particular, if $\cA$ is free then $J_\cA$ is unmixed (i.e. equidimensional). Freeness is an important property, yet its presence is not well understood.  A well-known  conjecture of Terao says that freeness of $\cA$ depends only on the intersection lattice of $\cA$.  Given an arbitrary arrangement $\cA$, we define the ideal $J_\cA^{top}$ to be the intersection of the codimension 2 primary components of $J_\cA$. This ideal is again unmixed, but not necessarily Cohen-Macaulay, though if $\cA$ is free then $J_\cA = J_\cA^{top}$. In this paper, we develop a new method for studying the ideals $J_\cA$  and $J_\cA^{top}$ and use it to establish results in the spirit of Terao's conjecture, but focusing on $J_\cA^{top}$ rather than $J_\cA$.  It is based on a novel application of liaison theory, namely the notion of the {\it general residual} of $\cA$. This residual ideal defines a scheme with surprisingly simple properties. These allow us to track back to $J_\cA^{top}$. Considerably extending earlier results with Schenck, we identify mild conditions on a hyperplane arrangement which imply that the Hilbert function of $\Jac( f_\cA)^{top}$ or even all the numerical information of the resolution of $\Jac( f_\cA)^{top}$, i.e.\ its graded Betti numbers, is determined by the intersection lattice of $\cA$. As another application, we establish new bounds on the global Tjurina number of a hyperplane arrangement. For line arrangements, we have additional results. Note that in this case $\Jac( f_\cA)^{top}$ is the same as the saturation of $\Jac( f_\cA)$.  We show, for an arbitrary line arrangement $\cA$, that the graded Betti numbers of $\Jac( f_\cA)^{sat}$ determine the graded Betti numbers of $\Jac( f_\cA)$, and of the corresponding Milnor module $J_\cA^{sat}/J_\cA$. Furthermore, we obtain a new freeness criterion for line arrangements -- it highlights the fact that free line arrangements are special by proving that a related codimension two ideal has the least possible number of  generators, namely two, if and only if $\cA$ is free. 
We illustrate our results by computing the graded Betti numbers for a number of basic arrangements that were not accessible with previous methods.

\end{abstract}

\maketitle

\tableofcontents

\section{Introduction}

Hyperplane arrangements are of interest in several areas and have been intensely investigated using tools from algebra, algebraic geometry, combinatorics, topology and other areas.
In this paper, we introduce a new method to investigate properties of Jacobian ideals of hyperplane arrangements. 

We work over a field $K$ of characteristic zero and we set $S = K[x_0,\dots,x_n]$, the polynomial ring over $K$ with standard grading. A subscheme $X$ of $\PP^n$ is said to be {\it arithmetically Cohen-Macaulay (ACM)} if $S/I_X$ is a Cohen-Macaulay ring.

Let $\cA \subset \PP^n$ be a finite set of hyperplanes of $\PP^n$. Abusing notation, we sometimes do not distinguish between the set $\cA$ and the hyperplane arrangement $\bigcup_{L \in \cA} L$. The arrangement $\cA$ is defined by a product $f_\cA$ of linear forms.  Let $J_\cA = \Jac(f_\cA)$ be the Jacobian ideal of $f_\cA$, i.e. the ideal generated by the $n+1$ partial derivatives of $f_\cA$. The ideal $J_\cA$ is not necessarily saturated. The saturation $J_\cA^{sat}$ in general is not unmixed, i.e. there may be embedded components. Both $J_\cA$ and $J_\cA^{sat}$ define the same subscheme $X_\cA$ of $\PP^n$.

Given a primary decomposition of $J_\cA$, it is a standard fact in commutative algebra that the intersection of primary components of codimension 2 is uniquely determined. The resulting ideal is unmixed (i.e. no embedded or isolated components) by construction. We denote by $J_\cA^{top}$ this ideal, and by $X_\cA^{top}$ the codimension 2 equidimensional scheme defined by $J_\cA^{top}$. This is a subscheme of $X_\cA$.

When $n=2$, $\cA$ is a line arrangement and $X_\cA$ is a zero-dimensional scheme. In this case there are no embedded points, and $J^{sat}$ is unmixed so $X_\cA = X_\cA^{top}$. The Milnor module is the graded $R$-module $J_\cA^{sat}/J_\cA$. The global Tjurina number is then the degree of $X_\cA$, an invariant studied by many authors. We extend the notion of global Tjurina number to hyperplane arrangements in any projective space:   the global Tjurina number for any hyperplane arrangement in $\PP^n$ is the degree of $X_\cA$, which coincides with the degree of $X_\cA^{top}$ since components of codimension 
$\geq 3$ do not contribute to the degree  (see the paragraph following Proposition \ref{prop:jac = union}).

In earlier work \cite{MNS} (also with H. Schenck) and \cite{MN}, the authors introduced the use of liaison addition and basic double linkage to the study of the schemes $X_\cA^{top}$, with a focus on their Hilbert functions and structure. In this paper we introduce a very strong new method from liaison theory, which we call the {\it general residual}, and we apply it together with the older methods to the study of graded Betti numbers. 

The general residual is defined as follows. Let $g$ be a general linear combination of the generators of $J_\cA = \Jac (f_\cA)$. We interpret $g$ as being the derivative of $f_\cA$ with respect to a general linear form $\ell$. Then $(f_\cA,g)$ is a regular sequence (Lemma \ref{lem:decomposition of ci}), so it links $J_\cA$ to an unmixed ideal $r_\cA$ that we call the general residual of $\cA$. 

 In liaison theory, it is often advantageous to use minimal links. In contrast, a general residual is \emph{not} a minimal link of a Jacobian ideal. A benefit of our choice of link is that it allows us to describe the linked ideal, that is, the general residual rather explicitly.  In particular, we describe the geometry of the scheme defined by a general residual: $r_\cA$ is supported on the singular locus of $\cA$, and each component is a complete intersection, one of whose generators is linear (Proposition \ref{prop:decomposition of residual}). Using the input data, this describe this linear generator explicitly.  Note that the support of the scheme defined by $r_\cA$ is completely determined by the codimension two subspaces in the intersection lattice of $\cA$ (see \Cref{cor:decomposition of residual}). In fact, the ideal itself is determined by these data and the choice of the general linear form $\ell$. Of course the precise residual depends on $\ell$, but  the generality of the form implies that properties such as graded Betti numbers and so Cohen-Macaulayness do not depend on the choice of the form.  Linking back using the same complete intersection gives us $J_\cA^{top}$. 
 Under mild assumptions, these simple properties of $r_\cA$ allow us to describe its primary decomposition, regularity and graded Betti numbers using geometric and liaison-theoretic methods. We then use liaison theory to derive the analogous information about $J_\cA^{top}$. In particular, we find the graded Betti numbers of $J_\cA^{top}$ for large classes of hyperplane arrangements.

Before giving the main results of this paper (that work in any projective space), we describe three applications for line arrangements. First,  we determine new bounds for the global Tjurina number of line arrangements, extending work of du Plessis and Wall, as well as of Beorchia and Mir\'o-Roig. Second,  in the special situation of line arrangements in $\PP^2$, through $r_\cA$ (under mild assumptions) we determine the graded Betti numbers of $J_\cA^{sat}$, which in turn allows us to determine the graded Betti numbers of both the Milnor module and of the original Jacobian ideal $J_\cA$. Third, through our study of the general residual, we are also able to give a new freeness criterion for line arrangements.  It highlights the fact that free line arrangements are special by proving that a related codimension two ideal has the least possible number of  generators, namely two, if and only if $\cA$ is free. 

We now give a more detailed overview of the paper, including the main results. In Section \ref{sec:background}  we recall some of the main tools that we use in this paper. These include a geometric description of $X^{top}$ (Proposition \ref{jac = union}),   arrangement-theoretic versions of liaison addition (Theorem \ref{LAthm1})  and basic double linkage (Corollary \ref{BDL corollary}), and   a duality result that will be crucial when we study the Milnor module later in this paper (\Cref{thm:Milnor duality}). 
Section \ref{sec:reduction to plane}  gives, via some technical lemmas, a tool connecting the  ideal $J_\cA^{top}$ of an arrangement $\cA$ in $\PP^n$ with the  ideal of the line arrangement in a general plane $\Lambda$ obtained by intersecting $\cA$ with $\Lambda$, in the case that $X^{top}$ is ACM (Corollary \ref{cor:restrict cCM top part}). 

Section \ref{sec:residuals}  introduces the general residual described above. An important ingredient is the point $\ell^\vee$ dual to the linear form $\ell$ (whose coordinates are the coefficients of $\ell$), which we view as being in the same projective space. Another is the set $S(\cA)$, the {\it top support of the singular locus}, which is the set of codimension two subspaces defined by all pairs of hyperplanes in $\cA$. Also, if $P \in S(\cA)$ then $g_P$ is the product of linear forms (up to scalar multiple) vanishing along $P$, and  $t_P = \deg g_P$. We show (Lemma \ref{lem:decomposition of ci}) that
\[
(f_{\cA}, \frac{\partial f_{\cA}}{\partial \ell}) = \bigcap_{P \in S(\cA)} (g_p,  \frac{\partial g_p}{\partial \ell}) \ \ \hbox{ and } \ \ |\cA|\cdot (|\cA| -1) = \sum_{P \in S(\cA)} t_p (t_p -1).
\]
This leads to the striking primary decomposition for $r_\cA$ (Proposition \ref{prop:decomposition of residual}  and Corollary \ref{cor:decomposition of residual}):
\[
r_\cA = \bigcap_{P \in S(\cA)} (\ell_P,  \frac{\partial g_P}{\partial \ell}) = \bigcap_{P \in S(\cA)} (\ell_P, I_P^{t_P - 1}).
\]
The main result of  \Cref{sec:residuals} is Theorem \ref{thm:reg residual} giving an upper bound for the Castelnuovo-Mumford regularity of $r_\cA$, namely $\reg r_\cA \leq |\cA| -1$. We combine some of these ideas in Corollary \ref{cor:restrict  aCM residual}, where we show that if $X^{top}$ is ACM and $\Lambda$ is a general 2-plane, and if $\cB$ is the restriction of $\cA$ to $\Lambda$, then $r_\cB$ has the same graded Betti numbers as $r_\cA$, and in fact the scheme defined by $r_\cB$ is the intersection of that defined by $r_\cA$ with $\Lambda$.

The main result of  \Cref{sec:initial degree} (Theorem \ref{thm:initial degree} ) uses the theory of E-type and N-type resolutions from liaison theory to give the surprising conclusion that for any hyperplane arrangement $\cA$ in $\PP^n$, the initial degree of $J_\cA^{top}$ is precisely $|\cA|-1$.  As a consequence (Corollary \ref{cor:reg estimate sharp}), we show that the bound on $\reg r_\cA$ is sharp: $\reg r_\cA = |\cA|-1$ for any arrangement $\cA$.
Section \ref{sec:initial degree} also introduces an auxiliary ideal, $I = r_\cA \cap I_{\ell^\vee}$. Proposition \ref{prop:add general point} shows that $r_\cA = (I,\frac{\partial f_\cA}{\partial \ell})$ and that $\reg I = |\cA|$. 

The remarkable properties of general residuals suggests that it is worthwhile to study a larger class of ideals whose primary components have the same simple structure as those of general residuals. We dub these ideals residual-lie ideals and characterize which residual-like ideals are actually general residuals. The result is one of the ingredients of the mentioned new freeness criterion. 

\Cref{sect:degree bounds Jacobian} gives our new bounds for the global Tjurina number for line arrangements mentioned above. Specifically, we prove

\vspace{.2in}

\noindent {\bf Theorem.} (Theorem  \ref{thm:deg Jacobian})
{\it For any arrangement  $\cA = \cA (f_{\cA})$ of $d \ge 3 $ lines in $\PP^2$, one has: 
\begin{itemize} 
\item[(a)] 
\[
\deg (\Jac (f_{\cA})) \ge \binom{d}{2},  
\]
and equality is true if and only if no three lines of $\cA$ are concurrent. 

\item[(b)] If the $d$ lines of $\cA$ are not concurrent then 
\[
\deg (\Jac (f_{\cA})) \le d^2 - 3d + 3.  
\]
Furthermore, equality is true if and only if $d-1$ lines, but not $d$ lines of $\cA$ are concurrent.  

\item[(c)] If $\cA$ does not have a subset of $d-1$ concurrent lines  (hence in particular $d \geq 4$) then 
\[
\deg (\Jac (f_{\cA})) \le d^2 - 4d + 7.  
\]
Moreover, if $d \ge 5$ then equality is true if and only if $\cA$ consists of $d-2$ concurrent lines and the two other lines meet in a point which is on one of the $d-2$ lines. 
\end{itemize} 
}
\vspace{.2in}

\noindent Using the hyperplane section methods mentioned above, Corollary \ref{cor:deg Jacobian} extends this to hyperplane arrangements in $\PP^n$.

It should be remarked that the methods introduced in \cite{MNS} and \cite{MN} can be applied to give graded Betti numbers of $J_\cA^{top}$ for many  hyperplane arrangements. Section \ref{sec:free res certain arr} shows that the introduction of general residuals greatly expands the classes of hyperplane arrangements for which we can produce the minimal free resolution of $J_\cA^{top}$. It also illustrates why passing to $r_\cA$ is so useful. The first step is Proposition \ref{prop:add almost generic hyperplane}. It was known that adding a general hyperplane $H$ to a hyperplane arrangement $\cA$ gives an arrangement $\cA + H$ for which $X_{\cA+H}^{top}$ is a basic double link of $X_\cA$. However, this is not true if $H$ vanishes on an element of $S(\cA)$ (i.e. we add an {\it almost generic hyperplane}). Nevertheless, Proposition \ref{prop:add almost generic hyperplane} shows that $r_{\cA+H}$ {\it is} a basic double link of $r_\cA$, in a precisely determined way. This allows us to find the Hilbert function and the minimal free resolution of $J_{\cA+H}^{top}$, which was not possible with earlier methods. Section \ref{sec:free res certain arr} gives many classes of line arrangements for which we compute the graded Betti numbers, none of which were possible before now. We end the section with Proposition \ref{prop:build recursively}, which partially  summarizes what can be done with the new methods. Remark \ref{limitations} shows that even these new methods continue to have limitations.

Section \ref{sec:res Jacobians} shows how to apply the results of this paper to compute the graded Betti numbers of the Milnor modules and Jacobian ideals of many line arrangements. The most general result is Theorem \ref{thm:Betti from sat}:

\vspace{.2in}

\noindent {\bf Theorem.}
{\it Let $\cA \subset \PP^2$ be any line arrangement. Write the graded minimal free resolution of $\Jac (f_\cA)^{sat}$ as 
\[
0 \to G \stackrel{\gamma}{\longrightarrow} S^b (-d+1)  \oplus F \to \Jac (f_\cA)^{sat}  \to 0, 
\]
where $d = |\cA|$. If $\cA$ is not free then $b \ge 3$ and the Milnor module  $M$ and the Jacobian ideal have graded minimal free resolutions of the form 
\[
0 \to S^{b-3}(-2d+2)  \oplus F^* (-3d+3)  \to G^* (-3d+3) \to G \to S^{b-3}(-d+1)  \oplus F \to M \to 0 
\]
and 
\[
0 \to S^{b-3}(-2d+2)  \oplus F^* (-3d+3) \to G^* (-3d+3)  \to S^{3}(-d+1) \to \Jac(f_\cA) \to 0. 
\]
}

\vspace{.2in}

The following consequence is also of interest (Theorem \ref{thm:Betti disconnected pencils}). Under a a mild assumption, it explicitly gives the above graded Betti numbers. They depend only on the intersection lattice of the arrangement. 

\vspace{.2in}

\noindent {\bf Theorem.}
{\it Let $\cA \subset \PP^2$ be an  arrangement of $d$ lines. Denote by 
        \[
    S_0 (\cA) = \{ P \in S(\cA) \; \mid \; t_P \ge 3\}
    \]
the set of points $P \in \PP^2$ that are contained in at least three distinct lines of $\cA$. Assume that no line in $\cA$ contains two or more elements of $S_0(\cA)$. Denote by $s$  the number of lines of $\cA$ that do not contain any $P \in S_0(A)$. Then one has: 
\begin{itemize}

\item[{\rm (a)}] If $\cA$ consists of $d$ concurrent lines then it is free and the graded minimal free resolution of its Jacobian ideal has the form 
\[
0 \to S(-2d+2)  \to S^2 (-d+1) \to \Jac (f_\cA) \to 0. 
\]

\item[{\rm (b)}] If $\cA$ consists of $d-1$ lines through a point $P \in \PP^2$ and one line that does not contain $P$ then it is free and the graded minimal free resolution of $Jac (f_\cA)$ has the form 
\[
0 \to S(-2d+3) \oplus S(-d) \to S^3 (-d+1) \to \Jac (f_\cA) \to 0. 
\]

\item[{\rm (c)}] If $\cA$ does not satisfy the condition in either (a) or (b) then $\cA$ is not  free and the Milnor module  $M$ and the Jacobian ideal have graded minimal free resolutions of the form 
\begin{align*} 
\hspace{13cm}&\hspace{-13cm}
0 \to S^{2 |S_0 (\cA)| + s-3}(-2d+2)   \to 
\begin{array}{c}
S^{ |S_0 (\cA)| + s-1} (-2d + 3) \\
\ \ \ \ \ \oplus \\[3pt]
{\displaystyle \bigoplus_{P \in S_0 (\cA)} S(-2 d+t_P -1) }
\end{array} 
\to  \\
\begin{array}{c}
S^{ |S_0 (\cA)| + s-1} (-d) \\
\ \ \ \ \ \oplus \\
{\displaystyle \bigoplus_{P \in S_0 (\cA)} S(-d-t_P +2) }
\end{array} 
\to S^{2 |S_0 (\cA)| + s-3}(-d+1)   \to M \to 0 
\end{align*}
and 
\begin{align*} 
\hspace{13cm}&\hspace{-13cm}
0 \to S^{2 |S_0 (\cA)| + s-3}(-2d+2)   \to 
\begin{array}{c}
S^{ |S_0 (\cA)| + s-1} (-2d + 3) \\
\ \ \ \ \ \oplus \\[3pt]
{\displaystyle \bigoplus_{P \in S_0 (\cA)} S(-2 d+t_P + 1) }
\end{array}   
\to \\
S^{3}(-d+1) \to \Jac(f_\cA) \to 0. 
\end{align*}
\end{itemize}
}

\vspace{.2in}

Finally, in section \ref{sec: freeness criterion} we give a new criterion for freeness of a hyperplane arrangement in $\PP^n$ (Theorem \ref{thm:char freeness lines}). It is stated in terms of an auxiliary ideal similar to the one mentioned above.


\section{Background}
\label{sec:background} 

In this section, we fix notation and recall some important concepts and results. 

Let $S = K[x_0, x_1, \dots,x_n]$ be a polynomial ring, where  $K$ is a field of characteristic zero. Let $\mathcal A$ be a hyperplane arrangement in $\mathbb P^n$ defined by a product $f_\cA$ of linear forms. Denote by $\Jac (f_\cA)$ the corresponding Jacobian ideal, generated by the first partial derivatives of $f_\cA$.

\begin{remark} \label{ci}
    A standard fact, easily shown, is that if $\mathcal A$ consists of $e$ hyperplanes and with  the property that there is a codimension 2 linear variety $\Lambda$ such that every hyperplane of $\mathcal A$ contains $\Lambda$ then $\Jac (f_\cA)$ is a saturated complete intersection of type $(e-1,e-1)$. 
\end{remark}

More generally, if $\mathcal A$ is a free arrangement then $S/\Jac (f_\cA)$ is Cohen-Macaulay. Most arrangements, however, are not free and there are many possible reasons for the failure of $S/\Jac (f_\cA)$ to be Cohen-Macaulay. It could fail to be saturated. But even the saturation can fail to be Cohen-Macaulay. One reason could be that it has embedded components, so we can remove these and study what is left. But even the remaining ideal can fail to be Cohen-Macaulay for more mysterious reasons. We denote this remaining ideal by $\Jac (f_\cA)^{top}$. The paper \cite{MNS} studied conditions that force $\Jac (f_\cA)^{top}$ to be Cohen-Macaulay, and it studied this mysterious failure of $\Jac (f_\cA)^{top}$ to be Cohen-Macaulay in some cases.

More precisely, consider a primary decomposition of $J$:
\begin{equation} \label{primary decomp}
\Jac (f_\cA) = \mathfrak q_1 \cap \dots \cap \mathfrak q_r
\end{equation}
and let $\mathfrak p_i$ be the associated prime of $\mathfrak q_i$ for $1 \leq i \leq r$. Removing all associated primes of height $> 2$, the remaining intersection of primary ideals is well-defined, and we denote by $\Jac (f_\cA)^{top}$ this intersection. Thus in (\ref{primary decomp})  $\Jac (f_\cA)^{top}$ is the intersection of the $\mathfrak q_i$ of height 2. Similarly, the intersection
of all associated primes absorbs all the associated primes of height $> 2$, and the resulting ideal is the radical 
\[
\sqrt{\Jac (f_\cA)} = \mathfrak p_1 \cap \dots \cap \mathfrak p_r, 
\]
where  each component is the ideal of a linear variety of codimension~2. 

We will be interested in the geometry of the schemes corresponding to $\Jac (f_\cA)^{top}$ for hyperplane arrangements. We denote by $X^{top}$ such a scheme, and our main (but not exclusive) interest is when $X^{top}$ is arithmetically Cohen-Macaulay (ACM).

The papers \cite{MNS} (for hyperplane arrangements) and \cite{MN} (for hypersuface arrangements) introduced the use of liaison methods to study these propertis. In this paper we will introduce new liaison methods to study the minimal free resolutions of many Cohen-Macaulay ideals arising in this way. We first study these problems in $\mathbb P^n$, and then give a reduction to the case of line arrangements in $\mathbb P^2$. Our last results are in this latter setting.

First we recall the previous methods. Liaison Addition was originally described by P.~Schwartau in his Ph.D. thesis \cite{schwartau} (unpubished), and subsequently studied in \cite{GM}. Basic double linkage was originally  introduced by R. Lazarsfeld and P. Rao in \cite{LR} and studied more carefully in \cite{BM} and considerably extended in \cite{KMMNP}.  We refer to these sources for some of the machinery used below. 

 In the following result we denote by $L_i$ a hyperplane and by $\ell_i$ the linear form defining it (up to scalars). 

\begin{proposition}[\cite{MN} Proposition 4.1]
\label{jac = union}
    Let $\mathcal A = \bigcup_{i=1}^s L_i \subset \mathbb P^n$ be an arbitrary hyperplane arrangement. Then the top dimensional part of the scheme defined by $\Jac (f_\cA)$ is a union of complete intersections. Each of these complete intersections is supported at some linear space $\Lambda = L_i \cap L_j$  and equal
to the scheme defined by $\Jac (g)$, where $g$ is the product of the linear divisors of $f_\cA$ that are in $I_{\Lambda}$.
\end{proposition}

The following two tools were also used in \cite{MNS} but stated more formally (and in greater generality than is given here) in \cite{MN}.

\begin{theorem}[Liaison Addition for Arrangements,\cite{MN} Theorem 4.5] 
    \label{LAthm1} 
 Let $\mathcal A_{fg} = \mathcal A_f \cup \mathcal A_g \subset \PP^n$ be a hyperplane arrangement, where $f = f_1 \cdots f_s$ and $g = g_1 \cdots g_t$ are products of linear forms, such that 

\begin{enumerate}

\item \label{LA1} $\hbox{codim } (f_i, f_j,g) = 3$ whenever $i \neq j$.

\item \label{LA2} $\hbox{codim } (f, g_i, g_j) = 3$ whenever $i \neq j$.

\end{enumerate}
 (Note that we do not assume that $\hbox{codim } (f_i, f_j, f_k) = 3$ or that $\hbox{codim } (g_i, g_j, g_k) = 3$ for $i, j, k$ distinct.)

Then $\mathcal A$ has the following properties:

\begin{enumerate} 

\item $\displaystyle \Jac(fg)^{top} = \Jac(f)^{top} \cap \Jac(g)^{top} \cap \underbrace{\left [ \bigcap_{i,j} (f_i, g_j) \right ]}_{= \ (f,g)}.$

\item $\displaystyle \Jac(fg)^{top} = g \cdot \Jac(f)^{top} + f \cdot \Jac(g)^{top}$. We say that $\Jac(fg)^{top}$ is obtained by Liaison Addition from $\Jac(f)^{top}$ and $\Jac(g)^{top}$.

\item Statements (a) and (b) continue to hold if we replace $\Jac(fg)^{top}$, $\Jac(f)^{top}$ and $\Jac(g)^{top}$ with their radicals. 

\end{enumerate}
\end{theorem}

\begin{corollary} [Basic Double Linkage for arrangements, \cite{MN} Corollary 4.6]
 \label{BDL corollary}
Let $\mathcal A_f$ be a hyperplane arrangement in $\mathbb P^n$, defined  by a product of linear forms $f = f_1 \cdots f_s$. Let $h$ be a linear form. Assume: $\hbox{codim } (f_i, f_j,h) = 3$ whenever $i \neq j$ (note that this implies that $(f,h)$ is a regular sequence, and also that $h$ does not vanish on a component of $\hbox{Jac}(f)^{top}$).

\noindent Then the following hold:

\begin{enumerate} 

\item $\displaystyle \Jac(fh)^{top} = \Jac(f)^{top} \cap  \underbrace{\left [ \bigcap_{i,j} (f_i, h) \right ]}_{= \ (f,h)}.$

\item $\displaystyle \Jac(fh)^{top} = h \cdot \Jac(f)^{top} + (f)$. (We say that {\em $\Jac(fh)^{top}$ is obtained from $\Jac(f)^{top}$ by Basic Double Linkage}). 

\item Statements (a) and (b) continue to hold if we replace $\Jac(fh)^{top}$ and $\Jac(f)^{top}$ by their radicals. 

\end{enumerate}
\end{corollary}

\begin{remark}
    One of the first results that was proved about Liaison Addition in codimension two (see for instance \cite{BM} Proposition 2.7) is that there is a certain very useful short exact sequence, which we will state in the context of hyperplane arrangemnts under the assumptions of Theorem \ref{LAthm1}. Let $a = \deg (f), \ b = \deg (g)$. Then the following sequence is exact: 
    \[
0 \rightarrow S(-a-b ) 
\stackrel{ { \tiny 
\left [  
\begin{array}{c} 
\hspace{-.1in} -g \\ \, \hspace{-.05in} f 
\end{array} \hspace{-.05in}
\right ] } }{\longrightarrow}  
\hbox{Jac}(g)^{top} (-a) \oplus \hbox{Jac}(f)^{top}(-b) 
\stackrel{[f \ g]}{\longrightarrow}
\hbox{Jac}(fg)^{top} \rightarrow 0.
    \]
    The version for Basic Double Linkage is
    \[
    0 \rightarrow S(-a-1 ) 
\stackrel{ { \tiny 
\left [  
\begin{array}{c} 
\hspace{-.1in} -h \\  \hspace{-.05in} f 
\end{array} 
\hspace{-.05in}
\right ] } }{\longrightarrow}  
 S(-a) \oplus
\hbox{Jac}(f)^{top} (-1) 
\stackrel{[f \ h]}{\longrightarrow}
\hbox{Jac}(fh)^{top} \rightarrow 0.
    \]
\end{remark}

\begin{remark} \label{BDL are ok} 
    Suppose $J$ is an ideal (for example a Jacobian or its top dimensional part) and $\ell$ is a linear form not vanishing on any  component of $J$. Let $F \in [J]_a$. Let $I = \ell \cdot J + (F)$ be the corresponding basic double link. Denote by $\nu$ the number of minimal generators of an ideal. Then $\nu(I) = \nu(J)$ if and only if $F$ is a minimal generator of $J$, and otherwise $\nu(I) = \nu(J)+1$. 
 Furthermore, this information leads to a description of the minimal free resolution of $I$. See Proposition \ref{prop:MFR under BDL} and Corollary \ref{cor:MFR res after adding generic hyperplane} for details.

\end{remark}

\begin{remark}
    In part of this paper we will apply Liaison Addition and Basic Double linkage to line arrangements in $\PP^2$. For this we need to observe that since $\Jac(f)$ defines a zero-dimensional scheme, there are no nontrivial embedded components, and so $\Jac(f)^{sat}$ and $\Jac(f)^{top}$ coincide. In particular, we will use Liaison Addition and Basic Double Linkage to study the Jacobian ideal $J$ and its saturation $J^{sat}$ for line arrangements in $\PP^2$. 
    
    We will use these results to describe the minimal free resolutions of certain arrangements, and as a consequence we will make a study of the corresponding Milnor modules (defined below). Then we will describe how to use these as building blocks for larger arrangements with certain hypotheses.
\end{remark}

\begin{definition}
    Let $\mathcal A$ be a line arrangement in $\PP^2$ with Jacobian ideal $J$. Then $\mathcal A$ is {\em free} if $J = J^{sat}$. If $\mathcal A$ is not free, the corresponding {\em Milnor module} $ M(\mathcal A)$  is the module $J^{sat}/J$.
\end{definition}

There is an important fact about Milnor modules, due to Sernesi \cite{sernesi}.  Generalizations are due to Van Straten and Warmt  as well as  Eisenbud and Ulrich. We state the result in the context that we need.

\begin{theorem}[{\cite[Theorem 3.2]{sernesi}, \cite[Theorem 4.7]{VsW},  \cite[Theorem 2.1]{EU}}]
    \label{thm:Milnor duality} 
    Let $\mathcal A$ be a line arrangement in $\PP^2$ defined by a product $f$ of $r+1$ linear forms. Let $J = (\frac{\partial f}{\partial x}, \frac{\partial f}{\partial y}, \frac{\partial f}{\partial z})$ be the associated Jacobian ideal, so each generator has degree $r$.  Then $M(\mathcal A) \cong [M(\mathcal A)(3 r)]^\vee$, where $^\vee$ represents the $K$-dual. In particular, the resolution  is self-dual up to shift.
\end{theorem}


\section{Reducing to the case of line arrangements}
    \label{sec:reduction to plane} 
    
For any hyperplane arrangement $\cA \subset \PP^n$ with $n \ge 3$ and any hyperplane $H \subset \PP^n$ that is not in $\cA$, we 
may consider the set of codimension two subspaces of $\PP^n$ 
\[
\cA \cap H = \{ L \cap H \; \mid \; L \in \cA \}
\]
as a set of hyperplanes of $H$. Thus, they determine a hyperplane arrangement of $H$ to which we refer as the restriction of $\cA$ to $H \cong \PP^{n-1}$. Repeating this restriction step  $n - 2$ times, we obtain a line arrangement in a projective plane. The goal of this section is to show that some properties of the original hyperplane arrangement can be inferred from its restriction to a line arrangement.  Later on, this will allow us to simplify some arguments by focussing on line arrangements only. 

We begin with a general observation. 

\begin{lemma}
   \label{lem:primary decomp under restriction} 
Let $I \subset S$ be an unmixed ideal with minimal primary decomposititon
\[
I = \fq_1 \cap \cdots \cap \fq_t. 
\]
Assume $\dim S/I \ge 2$, and let $x \in [S]_1$ be a nonzerodivisor of $S/I$. For any ideal $\fa \subset S$ and any $f \in S$, denote by $\bar{\fa}$  and $\bar{f}$ the image of $\fa$ and $f$, respectively,  in the polynomial ring $\bar{S} = S/x S$. Then one has 
\[
(\bar{I})^{top} = \overline{\fq_1}^{top} \cap \cdots \cap \overline{\fq_t}^{top}. 
\]
\end{lemma} 

\begin{proof}
Passing to $\bar{S}$, we obtain
\[
\bar{I} \subseteq \overline{\fq_1} \cap \cdots \cap \overline{\fq_t} \subseteq \overline{\fq_1}^{top} \cap \cdots \cap \overline{\fq_t}^{top}.
\] 
Localizing at any associated prime of $S/\bar{I}^{top}$ then provides 
\[
(\bar{I})^{top} \subseteq \overline{\fq_1}^{top} \cap \cdots \cap \overline{\fq_t}^{top}. 
\]
The assumption on $x$ yields that both ideals of the inclusion have the same degree as $I$, which implies the desired equality. 
\end{proof}

The next result is likely well-known. Note though that we do not assume that $x$ is general. 

\begin{lemma}
  \label{lem:liaison under restriction} 
If ideals $I$ and $J$ of $S$ are linked by a Gorenstein ideal $\fc$, and a linear form $x \in S$ satisfies $\fc : x = \fc$, then $(\bar{I})^{top}$ and $(\bar{J})^{top}$ are linked by $\bar{\fc}$. 
\end{lemma}

\begin{proof}
The assumption on $x$ ensures that the passage from $I, J$ and $\fc$ to their images  in $\bar{S}$ preserves  degrees. 
Furthermore, $I \cdot J \subseteq \fc$ implies $\bar{I} \cdot \bar{J} \subseteq \bar{\fc}$, and so 
\[
(\bar{I})^{top} \subseteq \bar{\fc} : \bar{J} = \bar{\fc} : (\bar{J})^{top}. 
\]
Both ideals have the same degree, and so they must be equal. 
\end{proof}

Turning to Jacobian ideals, we need further notation. 
Consider any linear form $x \neq 0$ of $S$. We want to identify $S/xS$ and a suitable subring $R$ of $S$. To this end choose
linear forms $y_1,\ldots,y_n \in S$ such that $S = K[x, y_1,\ldots,y_n]$ and $\frac{\partial x}{\partial y_i} = 0$ for any $i$. Any homogeneous $g \in S$ can be uniquely written as
\begin{equation}
   \label{eq:def restriction} 
g = \sum_{ j \ge 0} x^j g_j \quad \text{ with homogeneous $g_j \in R = K[y_1,\ldots,y_n]$}. 
\end{equation}
Then $g_0$ is the restriction of $g$ to $R$ and we write $\bar{g} = g_0$. Moreover, the Jacobian ideal of $\bar{g}$, denoted by $\Jac (\bar{g})$, is the ideal of $R$ that is generated by the partial derivatives of $\bar{g}$ with respect to $y_1,\ldots,y_n$. Note that 
$\frac{\partial x}{\partial y_i} = 0$ implies $\frac{\partial y_i}{\partial x} = 0$, and so $\frac{\partial \bar{g}}{\partial x} = 0$ for any $\bar{g} \in R$. 

For a hyperplane arrangement $\cA \subset \PP^n$ and a hyperplane $H = V(x) \notin \cA$ that does not contain any intersection of two 
distinct hyperplanes of $\cA$, denote by $\bar{\cA}$ the restriction of $\cA$ to $H$. Thus, if $f_\cA = \ell_1 \cdots \ell_d \in S$, then $\bar{\cA}$ is defined by 
\begin{equation}
  \label{eq:restricted arrangement}
f_{\bar{\cA}} =  \overline{\ell_1} \cdots \overline{\ell_d} = \overline{f_\cA} \in R. 
\end{equation}
Is is again a squarefree polynomial. 

\Cref{eq:def restriction} immediately implies:  

\begin{lemma}
  \label{lem:restricted partials}
For any $g \in S$, one has 
\[
\overline{ \frac{\partial g}{\partial y_i} } = \frac{\partial \bar{g}}{\partial y_i}. 
\]
\end{lemma} 

\begin{proof}
Writing $g = \sum_{ j \ge 0} x^j g_j$ as in  \Cref{eq:def restriction} and using $\frac{\partial x}{\partial y_i} = 0$, one gets $\frac{\partial g}{\partial y_i} =  \sum_{ j \ge 0} x^j \frac{\partial g_j}{\partial y_i} $, which gives $\overline{ \frac{\partial g}{\partial y_i} } = \frac{\partial g_0}{\partial y_i}.$ We conclude as $\bar{g} = g_0$. 
\end{proof}

The following observation is needed when we restrict Jacobian ideals of hyperplane arrangements. 

\begin{lemma}
   \label{lem-restrict ci} 
Let $\Lambda \subset \PP^n$ be a codimension two subspace that is not contained in $H = V(x)$. Fix generators $\ell_1,\ell_2 \in S$ of $I_{\Lambda}$. Then one has for any homogeneous $g \in K[\ell_1, \ell_2]$ an equality of ideals of $R$, 
\[
\overline{\Jac (g)} = \Jac (\bar{g}). 
\]
\end{lemma}

\begin{proof}
By assumption, we know $x \notin (\ell_1, \ell_2)$. Thus, there are linear forms $m_1, m_2 \in R$ and $a, b \in K$ such that, possibly after scaling, one has
\[
\ell_1 = x + a m_1 \quad \text{ and } \quad \ell_2 = x + b m_2. 
\]
Since $g \in K[\ell_1, \ell_2]$, we get
\begin{align*}
\Jac (g) & = (\frac{\partial g}{\partial \ell_1}, \frac{\partial g}{\partial \ell_2}) 
 = (\frac{\partial g}{\partial x} + a  \frac{\partial g}{\partial m_1}, \frac{\partial g}{\partial x} + b \frac{\partial g}{\partial m_2}) \\
& \subseteq (\frac{\partial g}{\partial x} , \frac{\partial g}{\partial m_1}, \frac{\partial g}{\partial m_2} ) \subseteq \Jac (g). 
\end{align*}
It follows that 
\begin{align}
   \label{eq:containment} 
\frac{\partial g}{\partial x} \in ( \frac{\partial g}{\partial m_1}, \frac{\partial g}{\partial m_2}) 
\end{align}
and, by a straightforward computation using \Cref{eq:def restriction}, 
\[
\overline{\Jac (g)} = \overline{ (\frac{\partial g}{\partial x} , \frac{\partial g}{\partial m_1}, \frac{\partial g}{\partial m_2} ) } 
= (g_1, \frac{\partial g_0}{\partial m_1}, \frac{\partial g_0}{\partial m_2} ).  
\]

By our definition of the restriction, we have $\bar{g} = g_0 \in K[m_1, m_2]$, and so 
$\Jac (\bar{g}) = ( \frac{\partial g_0}{\partial m_1}, \frac{\partial g_0}{\partial m_2} )$. 
Passing to the restriction of Relation \eqref{eq:containment} to $R$, we conclude $g_1 \in ( \frac{\partial g_0}{\partial m_1}, \frac{\partial g_0}{\partial m_2} )$. Combining this with the last two equalities gives $\overline{\Jac (g)} = \Jac (\bar{g})$. 
\end{proof}

The analog of \Cref{lem-restrict ci} is not true for an arbitrary hyperplane arrangement.

\begin{example}
Let $\cA \subset \mathbb P^3$ be a hyperplane arrangement consisting of four general planes. Then $\Jac (f_\cA)$ is the saturated ideal of the six lines of a tetrahedral configuration, and so $\cA$ is free. Let $H$ be a general plane. The restriction of $\cA$ to $H$ is a general set of four lines in $H$, and thus the corresponding Jacobian ideal $\Jac (f_{\overline{\cA}})$ defines six  points with generic Hilbert function. The  ideal of these six points has four minimal generators, which shows that  $\Jac (f_{\overline{\cA}})$ is not saturated. 
Since $\Jac (f_\cA)$ is a Cohen-Macaulay ideal, so is its restriction $\overline{\Jac (f_\cA)}$. It follows that  
$\overline{\Jac (f_\cA)} \neq \Jac (f_{\overline{\cA}})$. However, one has $\overline{\Jac (f_\cA)} = \Jac (f_{\overline{\cA}})^{sat}$.

Notice,  if instead we take five general planes for $\cA$ then even $\overline{\Jac (f_\cA)}$ is not saturated. 
\end{example}

Nevertheless, an analog of \Cref{lem-restrict ci} is true if one compares the top-dimensional parts of the ideals. 

\begin{proposition}
   \label{lem"restriction top-dimensional} 
Consider a hyperplane arrangement  $\cA \subset \PP^n$ with $n \ge 3$, and let $x \in [S]_1$ be a linear form such that $\Jac (f_\cA)^{top} : x = \Jac (f_\cA)^{top}$. Then one has an equality of ideals of $R$, 
\[
\overline{\Jac (f_\cA)^{top}}^{top} = \Jac (f_{\overline{\cA}})^{top}.  
\]
\end{proposition} 

\begin{proof}
Using \cite[Proposition 4.1]{MN} as stated in \Cref{prop:jac = union}  below, 
we know 
\[
\Jac (f_{\cA})^{top} = \bigcap_{P \in S(\cA)} \Jac(g_P), 
\]
where each $g_P$ is the product of linear divisors of $f_\cA$ that are contained in the ideal $I_P$. Thus, 
 \Cref{lem-restrict ci} gives 
 $\overline{\Jac (g_P)} = \Jac ({\overline{g_P}}) = \Jac (g_{\bar{P}})$, where $\bar{P} \subset \Proj (R) = V(x)$ denotes the hyperplane 
 section of $P$. 
Invoking \Cref{prop:jac = union} again we get
$\overline{\Jac (f_\cA)^{top}}^{top} = \Jac (f_{\bar{\cA}})^{top}$. 
\end{proof}

\begin{remark}
If we strengthen the assumption on $x$ in \Cref{lem"restriction top-dimensional} to $\Jac (f_\cA)^{sat} :~x = \Jac (f_\cA)^{sat}$ then we have $\overline{\Jac (f_\cA)^{top}}^{top} = \overline{\Jac (f_\cA)}^{top}$, and so the conclusion of  \Cref{lem"restriction top-dimensional} becomes 
\[
\overline{\Jac (f_\cA)}^{top} = \Jac (f_{\overline{\cA}})^{top}.  
\]
\end{remark}

Using this result repeatedly, we obtain the following consequence. 

\begin{corollary}
   \label{cor:restrict cCM top part}
Let $\cA \subset \PP^n$ be a hyperplane arrangement  such that $S/\Jac (f_\cA)^{top}$ is Cohen-Macaulay, and let 
$\Lambda$ be a general 2-plane in $\mathbb P^n$. Denote by $\mathcal B$ the line arrangement  in $\Lambda$ obtained by restricting $\mathcal A$ to $\Lambda$. 
Then the  saturation of the Jacobian ideal  $\Jac (f_\cB)$ of $\mathcal B$ (in the coordinate ring of $\Lambda$)  has the same graded Betti numbers 
as  $\Jac (f_\cA)^{top}$. The scheme defined by $\Jac (f_{\mathcal B})$ is the intersection  of the scheme defined by 
$S/\Jac (f_\cA)^{top}$  with $\Lambda$. \end{corollary}

\begin{proof}
Fix a minimal generating set $\{\ell_1,\ldots,\ell_{n-2}\}$ for  the ideal $I_{\Lambda} \subset S$ of $\Lambda$.  By the generality of $\Lambda$, we may assume for $x = \ell_1$ that $\Jac (f_\cA)^{top} : x = \Jac (f_\cA)^{top}$. Thus, \Cref{lem"restriction top-dimensional} gives $\overline{\Jac (f_\cA)^{top}} = \overline{\Jac (f_\cA)}^{top} = \Jac (f_{\overline{\cA}})^{top}$, and so $\Jac (f_\cA)^{top}$ has the same graded Betti numbers (over $S$) as $\Jac (f_{\overline{\cA}})^{top}$ (over $R \cong S/xS$). Repeating this argument suitably gives the claim. 
\end{proof}

\begin{remark}
   \label{rem:needed generality} 
The needed generality of $\Lambda$ in \Cref{cor:restrict cCM top part} can be described more precisely. We only need to make sure that \Cref{lem"restriction top-dimensional}  can be applied $(n-2)$ times, that is, using the notation of the above proof, each $\ell_i$ 
satisfies $(\Jac (f_{\mathcal B_{i-1}})^{top} : \ell_i = \Jac (f_{\mathcal B_{i-1}})^{top}$, where ${\mathcal B_{i-1}}$ denotes the restriction of $\cA$ to the subspace of $\PP^n$ defined by $(\ell_1,\ldots,\ell_{i-1})$. 
\end{remark}

\section{General residuals of arrangements} 
\label{sec:residuals}

The purpose of this section is to introduce the concept of a general residual of a hyperplane arrangement $\mathcal A$ in $\PP^n$ and to establish some of its properties.  In particular, we describe its minimal primary decomposition, establish an upper bound on its Castelnuovo-Mumford regularity and consider its behavior under when one restricts an arrangement to a hyperplane.

We denote by  $S$ a polynomial ring $K[x_0,\ldots,x_n]$ in $n+1$ variables. For a linear form $\ell = \sum a_i x_i \in S$, we write $\ell^\vee = (a_0 : \cdots : a_n)$ for the point dual to $\ell$, though we will consider $\ell^\vee$ as a point in the original $\PP^n$. 

We begin with some needed technical results.

\begin{lemma}
   \label{lem:nonvanishing at dual point} 
If $f \in S$ is any homogeneous polynomial of positive degree and $\ell \in [S]_1$ is general then 
\[
\frac{\partial f}{\partial \ell} (\ell^{\vee}) \neq 0. 
\]
\end{lemma}

\begin{proof}
Put $d = \deg f$.  Let $\ell = \sum_{i=0}^n a_i x_i$, and so $\ell^{\vee} = (a_0 : \cdots : a_n)$. Since we assume that the characteristic of $K$ is zero, the Euler derivative of $f$ is 
\[
d \cdot f = \sum_{i = 0}^n x_i  \cdot \frac{\partial f}{\partial x_i}. 
\] 
Since 
\[
\frac{\partial f}{\partial \ell} = \sum_{i = 0}^n a_i  \cdot \frac{\partial f}{\partial x_i}
\]
the equality $\frac{\partial f}{\partial \ell} (\ell^{\vee}) = 0$ implies 
\[
d \cdot f(\ell^{\vee}) = \sum_{i = 0}^n a_i  \cdot \frac{\partial f}{\partial x_i} (\ell^{\vee})  =  \frac{\partial f}{\partial \ell} (\ell^{\vee})  = 0,  
\] 
and so $f(\ell^{\vee}) = 0$. This is a contradiction because the point $\ell^{\vee}$ is general. 
\end{proof}

\begin{lemma}
    \label{lem:divisors on line}
Let $T$ be a polynomial ring in two variables over $K$. Consider any polynomials $0 \neq \ell \in [T]_1$ and $f, g \in [T]_d$ for some $d \ge 1$. If both $f$ and $g$ are not divisible by $\ell$  then there is an equality of ideals, 
\[
(\ell, f) = (\ell, g). 
\]
\end{lemma}

\begin{proof}
This follows since $[(\ell, f)]_d = [T]_d$, which is equivalent to $[T/(\ell, f)]_d = 0$. The latter is true as $\reg (T/(\ell, f)]) = d-1$ since 
$\ell$ and  $f$ form a regular sequence. 
\end{proof}

Consider now a hyperplane arrangement $\cA = \cA (f_{\cA}) \subset \PP^n$ defined by a squarefree product of linear forms 
$ f_{\cA}$. We define the \emph{top support of the singular locus $\Sing (\cA$)} as the set of codimension two subspaces $H \cap H'$ formed by 
distinct hyperplanes $H, H'$ of $\cA$, 
\[
S(\cA) = \{ H \cap H' \ | \ H \neq H' \in \cA \}.
\] 
Thus, the ideal of $S(\cA)$ is 
\[
I_{S(\cA)} = \bigcap_{\ell \neq \ell', \text{ divisors of } f_{\cA}} (\ell, \ell'). 
\]

Moreover, for any $P \in S(\cA)$, we denote by $g_P$ the product of linear factors of $f_{\cA}$ that are contained in $I_P$. We set $t_P = \deg g_P$ to be the number of hyperplanes of $\cA$ that contain $P$. 
We refer to $t_P$ as the \emph{multiplicity} of $P$ in $\cA$.
Using this notation, we can state \Cref{jac = union}  more explicitly. 

\begin{proposition} 
    \label{prop:jac = union}
If $\cA = \cA (f_{\cA}) \subset \PP^n$ is any hyperplane arrangement then the top-dimensional part of $\Jac (f_{\cA})$ is 
\[
\Jac (f_{\cA})^{top} = \bigcap_{P \in S(\cA)} \Jac(g_P), 
\]
where each $\Jac (g_P)$ is a complete intersection of type $(t_P -1, t_P - 1)$. 
\end{proposition}

Observe that this result implies
\[
I_{S(\cA)} = \sqrt{\Jac (f_{\cA})^{top}} = \sqrt{\Jac (f_{\cA}}) . 
\]
The local and global Tjurina number are measures of singularities of any plane curve. We can extend the definition to hyperplane arrangements of arbitrary dimension as follows. 
The \emph{local Tjurina number} of the hyperplane arrangement defined by $f_\cA$ at $P \in S (\cA)$ is the length of $(S/\Jac (f_\cA))_{I_P}$. The \emph{global Tjurina number} $\tau (\cA)$ is the sum over the local Tjurina numbers. It follows that the global Tjurina number is equal to 
$\deg \Jac (f_\cA) = \deg (\Jac (f_\cA)^{top}) = \sum_{P \in S(\cA)} (t_P -1)^2$.

For the remainder of this section we denote by $\ell$  a general linear form in $S$. We write $|\cA|$ for the cardinality of $\cA$.

We want to link the Jacobian ideal of $\cA$ using a suitable complete intersection. The following result describes in particular the support of the scheme defined by this link.

\begin{lemma}
    \label{lem:decomposition of ci} 
If $\cA = \cA (f_{\cA}) \subset \PP^n$ is any hyperplane arrangement and $\ell \in [S]_1$ is a general linear form then $(f_{\cA}, \frac{\partial f_{\cA}}{\partial \ell})$ is a complete intersection with a minimal primary decomposition 
\[
(f_{\cA}, \frac{\partial f_{\cA}}{\partial \ell}) = \bigcap_{P \in S(\cA)} (g_p,  \frac{\partial g_p}{\partial \ell}). 
\]
In particular, one has 
\begin{equation}
    \label{eq:degree ci} 
|\cA|\cdot (|\cA| -1) = \sum_{P \in S(\cA)} t_p (t_p -1).
\end{equation}
\end{lemma}

\begin{proof} 
We begin with establishing \Cref{eq:degree ci}. Choose any $L \in \cA$ and set $\cA' = \cA - L = \cA \setminus \{L\}$.   Consider any $P \in S(\cA)$. Temporarily, denote the number of hyperplanes of $\cA$ and $\cA'$ that contain $P$ by $t_{P, \cA}$ and $t_{P, \cA'}$, respectively. 
We now consider three cases. 

First, assume $P \notin S(\cA')$.  Then there is some $H \in \cA'$ such that $P = H \cap L$ and $t_{P, \cA} = 2$ whereas $t_{P, \cA'} = 1$. 

Second, assume $P \in S(\cA)$ and $P \subset L$. Then, one has $t_{P, \cA} = t_{P, \cA'} +1$. 

Third, assume $P \in S(\cA)$ and $P \not \subset L$. In this case, one has $t_{P, \cA} = t_{P, \cA'}$. 

Setting
\[
\cA' \cap L = \{ H \cap L \; \mid \; H \in \cA' \}, 
\]
the above observations can be summarized as follows: 
\begin{align}
\label{eq:change of local mult}
t_{P, \cA'}  & = \begin{cases}
t_{P, \cA}  & \text{ if } P \in S(\cA) \setminus (\cA' \cap L) \\
t_{P, \cA} - 1  & \text{ if } P \in \cA' \cap L. 
\end{cases}
\end{align}

Now we use induction on $d = |\cA| \ge 2$. If $d = 2$ then $S(\cA)$ consists of one subspace and \Cref{eq:degree ci} is true. 
Assume $d \ge 3$. Using Formula \eqref{eq:change of local mult}, we get
\begin{align*}
\sum_{P \in S(A)} t_{P, \cA} (t_{P, \cA} - 1) 
& = 
\sum_{P \in S(A) \setminus (\cA' \cap L)} t_{P, \cA} (t_{P, \cA} - 1) + \sum_{P \in \cA' \cap L} t_{P, \cA} (t_{P, \cA} - 1) \\
& = 
\sum_{P \in S(A) \setminus (\cA' \cap L)} t_{P, \cA'} (t_{P, \cA'} - 1) + \sum_{P \in \cA' \cap L} ( t_{P,  \cA'} + 1) t_{P, \cA'}  \\
& = 
\sum_{P \in S(A) \setminus (\cA' \cap L)} t_{P, \cA'} (t_{P, \cA'} - 1) +  \sum_{P \in \cA' \cap L} t_{P, \cA'} (t_{P, \cA'} - 1) 
+ 2 \sum_{P \in \cA' \cap L}  t_{P, \cA'} \\
& = 
(d-1)(d-2) + 2 \sum_{P \in \cA' \cap L}  t_{P, \cA'}, 
\end{align*}
where the last equality is due to the induction hypothesis.  Since $L$ intersects each $H \in \cA'$ in a codimension two subspace we get 
\[
\sum_{P \in \cA' \cap L}  t_{P, \cA'} = |\cA'| = d  - 1. 
\]
Combining the last two equations, we obtain $\sum_{P \in S(A)} t_{P, \cA} (t_{P, \cA} - 1) = d (d-1)$, as claimed. 

Now we establish the claimed primary decomposition. 
Consider any $P \in S(\cA)$. 
To simplify notation write $f = f_{\cA} = g \cdot h$ with $g = g_P$. By definition of $g$, we know $h \notin I_P$, and so $h$ is a unit in the localization $S_{I_P}$. Thus, we obtain 
\begin{align*}
(f, \frac{\partial f}{\partial \ell})_{I_P} & = (g h, \frac{\partial g}{\partial \ell} \cdot h + g \frac{\partial h}{\partial \ell})_{I_P} \\
& = (g, \frac{\partial g}{\partial \ell})_{I_P}, 
\end{align*}
which shows that the primary component of $(f, \frac{\partial f}{\partial \ell})$ associated to $I_P$ is $(g, \frac{\partial g}{\partial \ell})$. It remains to prove that $(f, \frac{\partial f}{\partial \ell})$ has no associated prime ideals other than $I_P$ with $P \in S(\cA)$. Comparing degrees of $(f_{\cA}, \frac{\partial f_{\cA}}{\partial \ell})$ and 
$I = \bigcup_{P \in S(\cA)} (g_p,  \frac{\partial g_p}{\partial \ell})$, this follows from Equation \eqref{eq:degree ci}.  
\end{proof} 

We need another preparatory result. We continue to use the above notation. 

\begin{lemma}
   \label{lem:local residual} 
If $\ell \in [S]_1$ is a general linear form then one has, for any $P \in S(\cA)$, 
\begin{align*}
(g_P,  \frac{\partial g_P}{\partial \ell}) : \Jac (g_P) & = (\ell_P,  \frac{g_P}{\widetilde{\ell_P}}) \\
& = (\ell_P,  \frac{\partial g_P}{\partial \ell}), 
\end{align*}
where $\ell_P$ is a linear form defining the linear span of $P$ and $\ell^{\vee}$ in $\PP^n$ and $\widetilde{\ell_P}$ is any linear divisor of $g_P$. 
\end{lemma}

\begin{proof}
Possibly changing coordinates, we may assume that $I_P = (x_0, x_1)$ and $\widetilde{\ell_P} = x_0$. 
Note that the second claimed equality follows by \Cref{lem:divisors on line}. It remains to justify the first equality. 

To simplify notation  write $g_P = x_0 \cdot h$ with $h \in K[x_0, x_1]$. Thus, we have to show that 
\[
(x_0 h, \frac{\partial x_0 h}{\partial \ell}) : \Jac (x_0 h) = (\ell_P, h). 
\]

Put $d = \deg h$.  Since $(x_0 h, \frac{\partial x_0 h}{\partial \ell})$ and $\Jac (x_0 h)$ are complete intersections of type $(d, d+1)$ and $(d, d)$, respectively, liaison theory gives that $(x_0 h, \frac{\partial x_0 h}{\partial \ell}) : \Jac (x_0 h)$ is a complete intersection of type $(1, d)$. Hence, it suffices to show that 
\[
(\ell_P, h) \cdot \Jac (x_0 h) \subseteq (x_0 h, \frac{\partial x_0 h}{\partial \ell}). 
\]
Since $(\ell_P, h) = (\ell_P, \frac{\partial x_0 h}{\partial \ell})$ we are done if we know that 
\[
\ell_P  \cdot \Jac (x_0 h) \subseteq (x_0 h, \frac{\partial x_0 h}{\partial \ell}). 
\]

To verify this, we simplify notation. Let $\ell = \sum_{i=0}^n a_i x_i$ with $a_i \in K$ and set $p = \frac{\partial x_0 h}{\partial x_0}$ and 
$q = \frac{\partial x_0 h}{\partial x_1}$. Note that the hyperplane spanned by $P$ and $\ell^{\vee}$ is defined by 
$\ell_P = a_1 x_0 - a_0 x_1$. Thus, using also $(d+1) \cdot x_0 h = x_0 p + x_1 q$,  the last inclusion becomes
\[
(a_1 x_0 - a_0 x_1)  \cdot (p, q) \subseteq (x_0 p + x_1 q, a_0 p + a_1 q). 
\]
It is true because 
\[
(a_1 x_0 - a_0 x_1)  \cdot p = a_1 \cdot (x_0 p + x_1 q) - x_1 (a_0 p + a_1 q)
\]
and 
\[
(a_1 x_0 - a_0 x_1)  \cdot q = - a_1 \cdot (x_0 p + x_1 q) + x_0 (a_0 p + a_1 q). 
\]
\end{proof}

We are ready to link the Jacobian ideal and to describe the residual.

\begin{proposition}
   \label{prop:decomposition of residual} 
For any general linear form $\ell \in S$,  the residual 
\[
r_\cA = (f_{\cA}, \frac{\partial f_{\cA}}{\partial \ell}) : \Jac (f_\cA)
\]
has minimal primary decompositions
\[
r_\cA = \bigcap_{P \in S(\cA)} (\ell_P,  \frac{\partial g_P}{\partial \ell}) =  \bigcap_{P \in S(\cA)}   (\ell_P,  \frac{g_P}{\widetilde{\ell_P}}). 
\]
\end{proposition}

\begin{proof}
Note that $(f_{\cA}, \frac{\partial f_{\cA}}{\partial \ell}) : \Jac (f_\cA) = (f_{\cA}, \frac{\partial f_{\cA}}{\partial \ell}) : \Jac (f_\cA)^{top}$. 
Since $(f_{\cA}, \frac{\partial f_{\cA}}{\partial \ell})$ and $\Jac (f_\cA)^{top}$ have the same associated prime ideals by \Cref{lem:decomposition of ci} and \Cref{prop:jac = union}, the result follows by computing locally at $I_P$ for any $P \in S(\cA)$. Hence we conclude by applying \Cref{lem:local residual}. 
\end{proof} 

Obviously, the residual $(f_{\cA}, \frac{\partial f_{\cA}}{\partial \ell}) : \Jac (f_\cA)$ depends on $\ell$. However, by the generality of $\ell$, its structure and, in particular, its graded Betti numbers are independent of $\ell$. Thus, abusing notation slightly, we introduce: 

\begin{definition}
   \label{def:residual} 
For any general linear form $\ell \in S$,  the ideal 
$(f_{\cA}, \frac{\partial f_{\cA}}{\partial \ell}) : \Jac (f_\cA)$ is said to be a \emph{general residual of $\cA$} (with respect to $\ell$) and denoted by $r_{\cA}$. 
\end{definition}

Using \Cref{lem:divisors on line}, one can describe the primary decomposition of the general residual without making any choices (besides $\ell$).  

\begin{corollary}
   \label{cor:decomposition of residual}
For any general linear form $\ell \in S$,  the minimal primary decomposition of $r_\cA$ is 
\[
r_\cA = \bigcap_{P \in S(\cA)} (\ell_P, I_P^{t_P - 1}), 
\]
where $t_P$ denotes the number of hyperplanes of $\cA$ that contain $P \in S(\cA)$. In particular, the radical of $r_\cA$ is determined by the intersection lattice of $\cA$ as 
\[
\sqrt{r_\cA} = I_{S(\cA)}. 
\]
\end{corollary}

Our next goal is to estimate the Castelnuovo-Mumford regularity of $r_\cA$. One ingredient is the following observation. 

\begin{lemma}
    \label{lem:residual with added hyperplane} 
Consider any hyperplane $H \subset \PP^n$ that is not in $\cA$. If $H$ is defined by a linear form $x \in S$, then one has
\[
r_{\cA + H} : x = r_\cA,  
\]
where $\cA + H$ is short for the hyperplane arrangement $\cA \cup \{H\}$.
\end{lemma}

\begin{proof}  
By definition of the top support of the singular locus of $\cA + H$, one has $S (\cA + H) = S (\cA) \cup (\cA \cap H)$, where, slightly abusing notation, we set  
\[
\cA \cap H  = \{ H' \cap H \; \mid \; H' \in \cA \}. 
\] 
Thus, applying \Cref{prop:decomposition of residual}  to the arrangement $A + H$,  we get
\[
r_{\cA + H} = \bigcap_{P \in S(\cA) - (\cA \cap H)}  (\ell_P,  \frac{g_P}{\widetilde{\ell_P}}) \cap \bigcap_{P \in \cA \cap H} (\ell_P,  \frac{g_P}{\widetilde{\ell_P}} \cdot x), 
\]
where,  for $P \in (\cA \cap H) - S(\cA)$, the polynomials $g_P$ and $\widetilde{\ell_P}$  both define the unique hyperplane in $\cA$ containing $P$, and so $(\ell_P,  \frac{g_P}{\widetilde{\ell_P}} \cdot x) = (\ell_P, x) = I_P$. 

If $P$ is in $(\cA \cap H) \cap S(\cA) = S(\cA) \cap H$ then $t_P \ge 2$ and $(\ell_P,  \frac{g_P}{\widetilde{\ell_P}}) \neq S$ is a complete intersection of type $(1, t_P -1)$. It follows that 
\begin{align*}
r_{\cA + H} : x & = \bigcap_{P \in S(\cA) - (\cA \cap H)}  (\ell_P,  \frac{g_P}{\widetilde{\ell_P}}) \cap \bigcap_{P \in S(\cA) \cap H} (\ell_P,  \frac{g_P}{\widetilde{\ell_P}}) \\
& = \bigcap_{P \in S(\cA)}  (\ell_P,  \frac{g_P}{\widetilde{\ell_P}}) \\
& = r_\cA, 
\end{align*}
where the last equality is true by \Cref{prop:decomposition of residual}. 
\end{proof}

The above proof also gives the following observation. 

\begin{corollary}
   \label{cor:degree change adding} 
If a hyperplane $H = V(x)$ is not in $\cA$, then 
\[
\deg r_{\cA + H} = \deg r_\cA + |\cA \cap H|. 
\]
\end{corollary}

We want to apply the following special case of \cite[Corollary 3.7]{DS}. We denote by $\fm = (x_0,\ldots,x_n)$ the homogeneous maximal ideal of $S$. 

\begin{proposition}
    \label{prop:DS result}
Consider proper homogeneous ideals $J_1,\ldots,J_t$ and $J$ of $S$. If there are homogeneous ideals $I_1,\ldots,I_t$ such that 
\[
I_k \cdot J_k \subseteq J \subseteq J_k
\]
for each $k$ and $I_1 + \cdots + I_t = \fm$, then 
\[
\reg J  \le 1 + \max \{\reg J_k \; | \; 1 \le k \le t\}. 
\]
\end{proposition}

The main result of this section is: 

\begin{theorem}
    \label{thm:reg residual} 
For any hyperplane arrangement $\cA \subset \PP^n$, one has
\[
\reg r_\cA \le |\cA| - 1. 
\]
\end{theorem}

\begin{proof}
It there is a point that is contained in every hyperplane of $\cA$, passing to the projection from this point, 
we may consider $\cA$ in a lower-dimensional space in order to establish the claim. Thus, we may assume that the ideal $I \subset S$ that is generated by the linear divisors of $f_\cA$ equals $\fm$. This implies $|\cA| \ge n+1$. 

We use induction on  $|\cA|$. If $|\cA| = n+1$, then we may assume $f_\cA = x_0 \cdots x_n$, and so 
\[
r_\cA = \bigcap_{i \neq j} (x_i, x_j). 
\]
Thus, $r_\cA$ defines a star configuration and $\reg (r_\cA) = n = |\cA| - 1$, as desired. 

Assume $|\cA| > n+1$. By induction, we know for any linear divisor $x$ of $f_\cA$ that 
\[
\reg r_{\cA - V(x)} \le |\cA| - 2. 
\]
Using the definition of a colon ideal and \Cref{lem:residual with added hyperplane}, we get
\[
x \cdot r_{\cA - V(x)} \subseteq r_\cA \subseteq r_{\cA - V(x)}, 
\]
where $\cA - V(x)$ is short for the arrangement $\cA \setminus \{V(x)\}$.
Since we assumed that $I = (x \; | \; \text{ $x$ a linear divisor of $f_\cA$})$ equals $\fm$, applying \Cref{prop:DS result} with $J = r_\cA$ and $J_x = r_{\cA - V(x)}$ gives $\reg r_\cA \le |\cA| - 1$. 
\end{proof}

We conclude this section with a result in the spirit of \Cref{sec:reduction to plane}. 
As for Jacobian ideals of hyperplane arrangements, it is  useful to consider restrictions of general residuals. Recall that the scheme defined by a general residual $r_\cA$ is supported on $S(\cA)$. We use the notation introduced in \Cref{sec:reduction to plane}, where we identified $\bar{S} = S/\ell S$ and a polynomial subring $R$ of $S$  (see Equation (\ref{eq:def restriction})). In particular, $\bar{g}$ denotes the image of $g \in S$ in $R$. 

\begin{proposition}
     \label{lem:restrict residual}
Consider a hyperplane arrangement  $\cA \subset \PP^n$ with $n \ge 3$, and let $\ell \in [S]_1$ be a general linear form. Let $x \in I_{\ell^\vee}$ be a linear form such that $r_\cA : x = r_\cA$. (Such $x$ exists because $\codim I_{\ell^\vee} = n > 2 = \codim r_\cA$.) Then one has: 

\begin{itemize}

\item[(a)] $\overline{(f_\cA, \frac{\partial f_{\cA}}{\partial \ell})} = (f_{\bar{\cA}}, \frac{\partial f_{\bar{\cA}}}{\partial \bar{\ell}} )$. 

\item[(b)] $\overline{r_\cA}^{top} = r_{\bar{\cA}}$, where $r_{\bar{\cA}}$ is the general residual of $\bar{\cA}$ with respect to $\bar{\ell}$. 

\end{itemize}
 
\end{proposition}

\begin{proof}
Changing coordinates, we may assume $x = x_n$, and so we identify $\overline{S}$ and $R = K[x_0,\ldots,x_{n-1}]$. If follows for $\ell^\vee = (a_0 : \cdots : a_n)$ that $a_n = 0$, that is, $\ell = \sum_{i=0}^{n-1} a_i x_i$ is in $R$. Hence, \Cref{lem:restricted partials} gives 
\[
\overline{\frac{\partial f_{\cA}}{\partial \ell}} = {\frac{\partial \overline{f_{\cA}}}{\partial \bar{\ell}}} = \frac{\partial f_{\bar{\cA}}}{\partial \bar{\ell}} . 
\]
Combined with Equality \eqref{eq:restricted arrangement}, Claim (a) follows. 

For (b), recall that the general residual of $\cA$ with respect to $\ell$ is 
\[
r_\cA = \bigcap_{P \in S(\cA)} (\ell_P, I_P^{t_P - 1}), 
\]
where $\ell_P$ is a linear form defining the linear span of $P$ and $\ell^\vee$ in $\PP^n$. Consider $\bar{P} = P \cap V(x)$ as a subscheme of $V(x)$. Note that $x \in I_{\ell^\vee}$ gives $\overline{I_P} = I_{\bar{P}}$. Furthermore, each of the ideals $(\ell_P, I_P^{t_P - 1})$ is a complete intersection of codimension two. Since $x$ is a nonzerodivisor, the same is true for $\overline{(\ell_P, I_P^{t_P - 1})}$. 
Hence,  \Cref{lem:primary decomp under restriction} implies
\[
\overline{r_\cA}^{top} = \bigcap_{P \in S(\cA)} \overline{(\ell_P, I_P^{t_P - 1})}  = \bigcap_{\bar{P} \in S(\bar{\cA)}} (\ell_{\bar{P}}, I_{\bar{P}}^{t_{\bar{P} - 1}})  = r_{\bar{\cA}}, 
\]
where the last equality follows from the description of the general residual in \Cref{prop:decomposition of residual}. 
\end{proof}

Using \Cref{lem:restrict residual}(b) instead of \Cref{lem"restriction top-dimensional} , the arguments for establishing \Cref{cor:restrict cCM top part} give the following analogous consequence for general residuals. 

\begin{corollary}
   \label{cor:restrict  aCM residual}
Let $\cA \subset \PP^n$ be a hyperplane arrangement  such that $S/\Jac (f_\cA)^{top}$ is Cohen-Macaulay, and let 
$\Lambda$ be a general 2-plane in $\mathbb P^n$. Denote by $\mathcal B$ the line arrangement  in $\Lambda$ obtained by restricting $\mathcal A$ to $\Lambda$.    Then the general residual $r_{\mathcal B}$ of $\mathcal B$ (in the coordinate ring of $\Lambda$)  has the same graded Betti numbers 
as  $r_\cA$. The scheme defined by $r_{\mathcal B}$ is the intersection  of the scheme defined by $r_\cA$ with $\Lambda$. 
\end{corollary}


\section{A lower degree bound for generators of the top-dimensional part} 
\label{sec:initial degree} 

The first goal of this section is to show that, for any hyperplane arrangement $\cA$, the initial degree of $\Jac(f_\cA)^{top}$ is exactly $|\cA| - 1$. This result has two applications. First, we use it to determine the Castelnuovo-Mumford regularity of any general residual. 
Second,  in \Cref{sec:res Jacobians} we will use it to infer the graded Betti numbers of $\Jac(f_\cA)$ for any line arrangement in $\mathbb P^2$ from the graded Betti numbers of $\Jac(f_\cA)^{top}$. 
Our second goal, for use in subsequent sections, is to  derive properties of an ideal that is closely related to a general residual (see \Cref{prop:add general point}). 

We begin by introducing some useful notation.  For any graded $S$-module $M$, the \emph{$S$-dual} of $M$ is the graded module $M^* = Hom_S (M, S)$. The \emph{$K$-dual} of $M$ is the graded $S$-module $M^{\vee} = Hom_K (M, K)$, where $K \cong S/\fm$ is considered as a graded module concentrated in degree zero. The \emph{initial degree} of $M$ is 
\[
a(M) = \inf \{j \in \ZZ \; | \; [M]_j \neq 0 \}. 
\]
Its \emph{end} is 
\[
e(M) = \sup \{j \in \ZZ \; | \; [M]_j \neq 0 \}. 
\]
In particular, in the case where $M = 0$ one has $e(M) = - \infty$ and $a(M) = \infty$. If $M \neq 0$ then $a(M)$ and $e(M)$ are integers. 
Moreover, we set 
\[
e^+ (M) = e (M/\fm M). 
\]
It denotes the maximum degree of a minimal generator of $M$.  
Thus, for a finitely generated graded  free $S$-module $F$,  one has $a (F^*) = - e^+ (F)$. 

We need some concepts and results from liaison theory and use \cite{N-liais} as our main reference. Since we consider only ideals of codimension two we specialize the general liaison result to this situation. 
Let $I \subset S$ be a homogeneous ideal of $S$ having codimension two. 
Following \cite[Definition 3.1]{N-liais}, an \emph{$E$-type resolution of $I$}  is an exact sequence of finitely generated graded $S$-modules
\[
0 \to E \to F \to I \to 0, 
\]
where $F$ is a free $S$-module. An \emph{$N$-type resolution of $I$}  is an exact sequence of finitely generated graded $S$-modules
\[
0 \to G \to N \to I \to 0, 
\]
where $G$ is a free $S$-module and $H^n_{\fm} (N) = 0$. 
We use  $H^i_\fm (M)$ to denote the $i$-th local cohomology module of $M$ with support in $\fm$.
The above resolutions are said to be \emph{minimal} if it is not possible to split
off free direct summands. 

If $S/I$ is Cohen-Macaulay then its resolutions of $E$- or $N$-type are free resolutions. 

\begin{lemma}
    \label{lem:E-type res} 
Any homogenous codimension two ideal $I$ of $S$ admits  minimal resolutions of $E$-type and $N$-type. These are uniquely determined up to isomorphisms of exact sequences of graded $S$-modules. 

Moreover, an $N$-type resolution $0 \to G \to N \to I \to 0$ is minimal if and only if there is a graded isomorphism $G^* \otimes_K S/\fm \cong Ext^1_S (I, R) \otimes_K S/\fm$. 
\end{lemma} 

\begin{proof}
Combine \cite[Lemma II.1.4]{N-hab} and \cite[Remark 3.2]{N-liais}. 
\end{proof}

If $I$ and $J$ are two ideals linked by a complete intersection then an $N$-type resolution of $I$ determines an $E$-type resolution of $J$, see \cite[Proposition 3.8]{N-liais}. This is crucial for the proof of the main result of this section. 

\begin{theorem}
    \label{thm:initial degree} 
If $\cA = \cA (f_\cA) \subset \PP^n$ is any (finite) hyperplane arrangement then the initial degree of $\Jac(f_\cA)^{top}$ is at least $|\cA| - 1$, that is, 
\[
a (\Jac(f_\cA)^{top}) = |\cA| - 1. 
\]
Moreover, $\Jac(f_\cA)^{top}$ has as at least as many minimal generators of degree $|\cA| - 1$ as $\Jac(f_\cA)$. 
\end{theorem}

In particular, this result says that the minimal generators of $\Jac(f_\cA)^{top}$ do not have degrees less than the degree of the minimal generators of $\Jac(f_\cA)$.

\begin{proof}[Proof of \Cref{thm:initial degree}] 
First, we show the estimate
\[
a (\Jac(f_\cA)^{top}) \ge |\cA| - 1. 
\]

Let $\ell \in S$ be a general linear form and put $d = |\cA|$. 
Since $(f_{\cA}, \frac{\partial f_{\cA}}{\partial \ell})$ is contained in $r_\cA$, we obtain a commutative diagram 
\begin{align*}
     \label{commutative diagram_multi factor}
\minCDarrowwidth20pt
\begin{CD}
0 @>>>  S(- 2d+1)  @>>>  S(-d+1) \oplus S(-d)  @>>>  (f_{\cA}, \frac{\partial f_{\cA}}{\partial \ell})  @>>>  0 \\
 @. @.  @.    @VVV \\
0 @>>>  G @>>>  N  @>>>   r_\cA @>>>  0,  \\[3pt]
\end{CD}
\end{align*} 
where the vertical map is injective and the horizontal sequences are minimal resolutions of $N$-type of $(f_{\cA}, \frac{\partial f_{\cA}}{\partial \ell})$ and $r_\cA$, respectively. Since
\[
\Jac(f_\cA)^{top} = (f_{\cA}, \frac{\partial f_{\cA}}{\partial \ell}) : r_\cA, 
\] 
lifting the vertical map to a morphism of complexes, one obtains the following $E$-type resolution of $\Jac(f_\cA)^{top}$, 
\[
0 \to N^* (-2d+1)  \to G^* (-2d+1) \oplus S(-d) \oplus S(-d+1) \to \Jac(f_\cA)^{top}  \to 0. 
\]
It follows that 
\begin{align} 
    \label{eq:minimum} 
a(\Jac(f_\cA)^{top}) & \ge \min \{d-1,  a (G^* (-2d+1)) \} \nonumber \\
&  =  \min \{d- 1,  a(G^*) + 2d-1 \}.   
\end{align}
where equality is true if the $E$-type resolution is minimal. By duality (see, e.g., \cite[Proposition 2.1]{N-liais}), there is a graded isomorphism 
\[
Ext_S^1 (r_\cA, S) \cong H^{n-1}_\fm (S/r_\cA)^\vee(n+1). 
\]
Hence, using also  \Cref{lem:E-type res}, the minimality of the $N$-type resolution of $r_\cA$ implies
\[
a (G^*) = a (Ext_S^1 (r_\cA, S)) = - e(H^{n-1}_\fm (S/r_\cA)) - n -1.   
\]
By \Cref{thm:reg residual} and the cohomological characterization of the Castelnuovo-Mumford regularity, one has 
\[
e(H^{n-1}_\fm (S/r_\cA)) + n-1 \le \reg (S/r_\cA) \le d-2,  
\]
and so we get
\begin{align*}
a(G^*) + 2d-1 & =  - e(H^{n-1}_\fm (S/r_\cA)) - n -2 + 2d \\
& \ge n-1 - d+2 - n-2+2d = d-1. 
\end{align*}
Taking into account Inequality \eqref{eq:minimum}, this yields
\[
a(\Jac(f_\cA)^{top}) \ge d-1, 
\]
as claimed. 

Second, to complete the argument recall that the minimal generators of $\Jac(f_\cA)$ have degree $d-1$. Since $\Jac(f_\cA)$ is contained in $\Jac(f_\cA)^{top}$ the above degree estimate implies that the minimal generators of $\Jac(f_\cA)$ are also minimal generators of $\Jac(f_\cA)^{top}$. 
\end{proof}

As a consequence, we show that the estimate in \Cref{thm:reg residual} is in fact an equality. 

\begin{corollary}
    \label{cor:reg estimate sharp}
If $\cA = \cA (f_\cA) \subset \PP^n$ is any (finite) hyperplane arrangement then the general residual has a minimal generator of degree $|\cA| - 1$ and Castelnuovo-Mumford regularity 
\[
\reg r_\cA = |\cA| - 1. 
\]

Moreover, if $\Jac(f_\cA)^{top}$  is Cohen-Macaulay and $t \ge 2$ denotes the number of minimal generators of $\Jac(f_\cA)^{top}$ with degree $|\cA| - 1$ then the general residual $r_\cA$ has exactly $t-1$ syzygies of degree $|\cA|$. 
\end{corollary}

\begin{proof}
We reverse the roles of $r_\cA$ and $\Jac(f_\cA)^{top}$  in the proof of \Cref{thm:initial degree}. 

Let $\ell \in S$ be a general linear form and put $d = |\cA|$. 
Since $(f_{\cA}, \frac{\partial f_{\cA}}{\partial \ell})$ is contained in $\Jac(f_\cA)^{top}$, we obtain a commutative diagram 
\begin{align*}
     \label{commutative diagram_multi factor}
\minCDarrowwidth20pt
\begin{CD}
0 @>>>  S(- 2d+1)  @>>>  S(-d+1) \oplus S(-d)  @>>>  (f_{\cA}, \frac{\partial f_{\cA}}{\partial \ell})  @>>>  0 \\
 @. @.  @.    @VVV \\
0 @>>>  E @>>> S^t(-d +1) \oplus  F  @>>>   \Jac(f_\cA)^{top} @>>>  0,  \\[3pt]
\end{CD}
\end{align*} 
where the vertical map is injective and the horizontal sequences are minimal resolutions of $E$-type of  $(f_{\cA}, \frac{\partial f_{\cA}}{\partial \ell})$ and $\Jac(f_\cA)^{top}$, respectively. By definition of $t$, we have $a (F) \ge d$. 
Since
\[
 r_\cA= (f_{\cA}, \frac{\partial f_{\cA}}{\partial \ell}) :  \Jac(f_\cA)^{top}, 
\] 
lifting the vertical map to a morphism of complexes, one obtains the following $N$-type resolution of $r_\cA$, 
\[
0 \to F^* (-2d+1) \oplus S^{t-1}(-d)  \to E^* (-2d+1) \oplus S(-d+1) \to  r_\cA \to 0,  
\]
where we have taken into account the cancellation caused by the fact that $\frac{\partial f_{\cA}}{\partial \ell}$ is a minimal generator of 
$\Jac(f_\cA)^{top}$. Since $f_\cA$ is not a minimal generator of $\Jac(f_\cA)$ it also is not a minimal generator of 
$ \Jac(f_\cA)^{top}$. Hence, there is no further cancellation in the mapping cone, and so the stated $N$-type resolution of $r_\cA$ is minimal. It shows that $r_\cA$ has a minimal generator of degree $d-1$. It also follows that $\reg (r_\cA) \ge d-1$. Together with \Cref{thm:reg residual}, this gives $\reg (r_\cA) = d-1$.

For the additional claim, note that $E$ is a free $S$-module if $ \Jac(f_\cA)^{top}$ is Cohen-Macaulay. Thus, the above $N$-type resolution is a minimal free resolution of $r_\cA$.  
Using that $a (F) \ge d$, it follows that 
\[
e^+ (F^* (-2d+1)) = e^+ (F^*) + 2d-1 = -a(F) + 2d-1 \le d-1. 
\]
Hence,  $r_\cA$ has exactly $t-1$ first syzygies of degree $d$,  and $d$ is the maximum degree of a first syzygy. 
\end{proof}

By the definition of the general residual $r_\cA$, it is clear that it contains $\frac{\partial f_{\cA}}{\partial \ell}$. The following result shows that this polynomial is an element of a minimal generating set of  $r_\cA$. Thus, it identifies a minimal generator of degree $|\cA| -1 $ of $r_\cA$ whose existence is guaranteed by \Cref{cor:reg estimate sharp}. 

\begin{proposition}
     \label{prop:add general point}
Let $r_\cA$ be the general residual of a hyperplane arrangement $\cA \subset \PP^n$ with respect to a general linear form 
$\ell$. Set 
\[
I = r_\cA \cap I_{\ell^{\vee}}. 
\]
Then one has: 
\begin{itemize}

\item[(a)] There is a minimal generating set of $r_\cA$ that contains $\frac{\partial f_{\cA}}{\partial \ell}$ and 
\[
r_\cA = (I, \frac{\partial f_{\cA}}{\partial \ell} ). 
\]

\item[(b)] The Hilbert function of $S/I$ is 
\begin{align*}
h_{S/I} (j ) & = \begin{cases}
h_{S/r_\cA} (j) & \text{ if } j \le |\cA|-2; \\
h_{S/r_\cA} (j) + 1 & \text{ if } j \ge |\cA| -1. 
\end{cases}
\end{align*}

\item[(c)]  The Castelnuovo-Mumford regularity of $I$ is 
\[
\reg I = |\cA|. 
\]
\end{itemize}

\end{proposition}

\begin{proof}
For ease of notation, set $d = |\cA|$, $f = f_\cA$ and $Q = \ell^\vee$. 

We begin by establishing the following claim: 
\begin{align}
    \label{eq:low degrees eq}
[r_\cA]_j = [I]_j \quad \text{ if } j \le d-2,
\end{align}
which is the heart of the proof of the Proposition.

First, we reduce to the case where $n =2$. We use the notation introduced in \Cref{sec:reduction to plane}.  
Assume $n \ge 3$ and fix a general linear form $y$  in the ideal $I_Q$.  Denote by 
$\bar{g}$  the image of $g \in S$ in the subring $R \cong S/yS$  of $S$ and similarly for ideals of $S$. We may assume $y = x_n$. 
Suppose on the contrary there is a polynomial 
$g \in [r_\cA]_j$ with $ j \le d-2$ that is not in $I$. The latter is true if and only if $g$ is not in $I_Q$. 
By the definition of the residual, the first part means that 
\[
g \cdot \Jac (f) \subseteq (f, \frac{\partial f}{\partial \ell} ). 
\]
Using \Cref{lem:restricted partials}, 
 this implies
\[
\bar{g} \cdot \Jac (\bar{f}) \subseteq (\bar{f}, \frac{\partial \bar{f}}{\partial \bar{\ell}} ). 
\]
It remains to show that $\bar{g}$ is not in the ideal $I_{\bar{\ell}^\vee}$ of $R$. Since $x_n$ is in $I_Q$ one has in $R$ for the point $\bar{\ell}^\vee \in \Proj (R)$ that $I_{\bar{\ell}^\vee} = \overline{I_Q}$. Hence $g \notin I_Q$ implies $\bar{g} \notin I_{\bar{\ell}^\vee}$, as desired. Using induction on $n$, it follows that it suffices to show Claim \eqref{eq:low degrees eq} if $n =2$.

Assume $n = 2$. We begin by considering the cases where $d \in \{2, 3\}$. If $S(\cA)$ consists of one point, $P$, then \Cref{prop:decomposition of residual} gives that $r_\cA$ is a complete intersection of type $(1, d-1)$. In fact, taking also \Cref{lem:divisors on line} into account we get $r_\cA = (\ell_P, \frac{\partial f}{\partial \ell} )$ with $\ell_P \in I_Q$. Since $\ell_P \in I_Q$, we get that $I = (\ell_P,g)$ for some multiple $g$ of $\frac{\partial f}{\partial \ell}$ with degree $d$, and so $r_\cA = (I, \frac{\partial f}{\partial \ell} )$ as desired. The only other option is that $d = 3$ and $r_\cA$ is the ideal of three reduced points. Again, the computation is straightforward. So now we assume $d \ge 4$. We proceed in several steps. \\[2pt]

(I) Note that any syzygy $(h_0, h_1, h_2)^T$ of $\Jac (f)$ gives rise to a syzygy of the ideal $(\frac{\partial f}{\partial x_0}, \frac{\partial f}{\partial x_1}, x_2 \frac{\partial f}{\partial x_2})$ because 
\[
h_0 \cdot \frac{\partial f}{\partial x_0} + h_1 \cdot \frac{\partial f}{\partial x_1} + h_2 \cdot \frac{\partial f}{\partial x_2} = 0
\]
implies
\[
(h_0 x_2) \cdot \frac{\partial f}{\partial x_0} + (h_1 x_2) \cdot \frac{\partial f}{\partial x_1} + h_2 (x_2 \cdot \frac{\partial f}{\partial x_2}) = 0. 
\]
Thus, we get a graded homomorphism of $S$-modules 
\[
\ffi \colon \Syz (\Jac (f))(-1) \to \Syz \left (\frac{\partial f}{\partial x_0}, \frac{\partial f}{\partial x_1}, x_2 \frac{\partial f}{\partial x_2} \right ), \;
 (h_0, h_1, h_2)^T \mapsto (x_2 h_0, x_2 h_1, h_2)^T. 
\]
As $S$ is an integral domain, $\ffi$ is injective. \\[2pt]

(II) Choose two further general linear forms $\ell_1, \ell_2 \in S$. Since $\Jac (f)$ has codimension two, the derivatives 
$\frac{\partial f}{\partial \ell}$ and $\frac{\partial f}{\partial \ell_1}$ form a regular sequence in $S$. Furthermore, the fact that $(\ell, \ell_1, \ell_2) = (x_0, x_1, x_2)$ implies that 
\[
( \frac{\partial f}{\partial \ell}, \frac{\partial f}{\partial \ell_1}) : \ell_2 = ( \frac{\partial f}{\partial \ell}, \frac{\partial f}{\partial \ell_1}). 
\]
It follows that 
\begin{align*}
\left (\frac{\partial f}{\partial \ell}, \frac{\partial f}{\partial \ell_1}, \ell_2 \frac{\partial f}{\partial \ell_2} \right ) : \ell_2 
& = \left  (\frac{\partial f}{\partial \ell}, \frac{\partial f}{\partial \ell_1} \right  ) : \ell_2  + \left (\frac{\partial f}{\partial \ell_2} \right ) \\
& = \left (\frac{\partial f}{\partial \ell}, \frac{\partial f}{\partial \ell_1}, \frac{\partial f}{\partial \ell_2} \right ) \\
& = \Jac (f). 
\end{align*}
Changing coordinates, we assume $\ell = x_0, \ \ell_1 = x_1$ and $\ell_2 = x_2$. Note that then $I_Q = (x_1,x_2)$.
Thus, we get a short exact sequence of graded $S$-modules 
\begin{equation}
   \label{eq:ses}
0 \to (S/\Jac (f)) (-1) \to S/\left (\frac{\partial f}{\partial x_0}, \frac{\partial f}{\partial x_1}, x_2 \frac{\partial f}{\partial x_2} \right ) \to 
S/\left (\frac{\partial f}{\partial x_0}, \frac{\partial f}{\partial x_1}, x_2  \right ) \to 0. 
\end{equation}
A graded minimal free resolution of the module on the right-hand side is given by the Koszul complex: 
\[
0 \to S(-2d+1) \to S^2(-d) \oplus S(-2d + 2) \to S^2 (-d + 1) \oplus S (-1) \to S \to 
S/\left (\frac{\partial f}{\partial x_0}, \frac{\partial f}{\partial x_1}, x_2  \right ) \to 0.
\]
Hence, the long exact $\Tor$ sequence induced by Sequence \eqref{eq:ses} induces a homomorphism 
$
\Tor_1^S (S/\Jac (f), K ) (-1)  \to  \Tor_1^S (S/(\frac{\partial f}{\partial x_0}, \frac{\partial f}{\partial x_1}, x_2 \frac{\partial f}{\partial x_2}), K )$, 
which gives an  isomorphism 
\[
[\Tor_1^S (S/\Jac (f), K )]_{j-1} \to  [\Tor_1^S (S/ \left (\frac{\partial f}{\partial x_0}, \frac{\partial f}{\partial x_1}, x_2 \frac{\partial f}{\partial x_2} \right ), K )]_{j}
\]
in every degree $j$ with $d < j < 2d-2$.  Combined with the injectivity of $\ffi$, we conclude that 
\begin{equation}
   \label{eq:iso of syz}
\ffi \colon [\Syz (\Jac (f))(-1)]_{j-1}  \to \big [\Syz \big (\frac{\partial f}{\partial x_0}, \frac{\partial f}{\partial x_1}, x_2 \frac{\partial f}{\partial x_2} \big ) \big]_j \quad \text{ is an isomorphism if } d < j < 2d-2. 
\end{equation}

(III) Assume on the contrary that there is a polynomial $g \in [r_\cA]_j$ with $ j \le d-2$ that is not in $I$, that is, $g$ is not in $I_Q$. If $j < d-2$, then, for a general form $h \in S$ of degree $d-2-j$, one also has $gh \in [r_\cA]_{d-2} \setminus I_Q$. Thus, we may assume there is some 
\[
g \in [r_\cA]_{d-2} \setminus I_Q. 
\]
Since $r_\cA = (f, \frac{\partial f}{\partial \ell} ) : \Jac (f) = (f, \frac{\partial f}{\partial x_0} ) : \Jac (f)$ (because we have assumed $\ell = x_0$), there are homogeneous polynomials $s_1, s_2 \in S$ such that
\[
g \cdot \frac{\partial f}{\partial x_1} = s_1 \cdot (d f) + s_2 \cdot \frac{\partial f}{\partial x_0}. 
\]
 Note that this implies  $\deg s_1 = d-3$ and $\deg s_2 = d-2$.  Invoking the Euler relation, this yields
\[
0 = (s_1 x_0 + s_2) \cdot  \frac{\partial f}{\partial x_0} + (s_1 x_1 - g ) \cdot  \frac{\partial f}{\partial x_1} + s_1 \cdot  x_2 \frac{\partial f}{\partial x_2}, 
\]
which corresponds to a syzygy of $\Syz \big (\frac{\partial f}{\partial x_0}, \frac{\partial f}{\partial x_1}, x_2 \frac{\partial f}{\partial x_2} \big )$ of degree $2d-3$. Using Isomorphism \eqref{eq:iso of syz}, it follows that $(s_1 x_0 + s_2, s_1 x_1 - g, s_1)$ is in the image of $\ffi$. In particular, this gives that  $s_1 x_1 - g$ is a multiple of $x_2$, and so $g \in (x_1, x_2) = I_Q$,  which is a contradiction to our assumption 
$g \in [r_\cA]_{d-2} \setminus I_Q$. This concludes our argument for establishing Claim \eqref{eq:low degrees eq}. \\[2pt]

We now turn to the assertions of the statement and consider an arbitrary hyperplane arrangement $\cA \subset \PP^n$. By \Cref{lem:nonvanishing at dual point}, we know that $ \frac{\partial f}{\partial \ell}$ is not in $I_Q$. Combined with Fact \eqref{eq:low degrees eq}, we conclude that the least degree of a form in $r_\cA$ that is not in $I$ is $d-1$. Since $I_Q$ is the ideal of a point we get  that $[(I_Q, \frac{\partial f}{\partial \ell})]_j = [S]_j$ if $j \ge d-1$. 
It follows that
\[
r_\cA/I \cong (r_\cA + I_Q)/I_Q \cong (S/I_Q) (-d+1).
\]
This isomorphism shows that $r_\cA/I$ is a cyclic module that is generated by the image of $\frac{\partial f}{\partial \ell}$, which implies $r_\cA = (I, \frac{\partial f_{\cA}}{\partial \ell} )$, as claimed in Part (a). 
Since $\frac{\partial f_{\cA}}{\partial \ell} $ is not in $I$, it can be chosen as part of a minimal generating set of $r_\cA$. 
 Moreover, we get an exact sequence of graded $S$-modules 
\begin{equation}
    \label{eq:SES}
0 \to (S/I_Q)(-d+1) \to S/I \to S/r_\cA \to 0. 
\end{equation}
It yields Claim (b). Furthermore, it induces a long exact local cohomology sequence, which provides 
isomorphisms
\[
H^i_{\fm} (S/I) \cong H^i_{\fm} (S/r_\cA) \quad \text{ if } i \ge 2
\]
and an exact sequence 
\[
0 \to H^1_{\fm} (S/I_Q) (-d+1) \to H^1_{\fm} (S/I) \to H^1_{\fm} (S/r_\cA) \to 0. 
\]
The largest degree $j$ such that $[H^1_{\fm} (S/I_Q)(-d+1)]_j$ is not zero is $d-2$.  Since $\reg S/r_\cA = d-2$ by  \Cref{cor:reg estimate sharp}, we obtain $\reg S/I = d-1$, as claimed in Part (c). 
\end{proof}
 
We record a useful consequence of \Cref{prop:add general point}(a). 

\begin{corollary}
    \label{cor:gens of residual} 
For any hyperplane arrangement $\cA$ one has: 
\begin{itemize}

\item[(a)] The general residual $r_\cA$ with respect to $\ell \in [S]_1$ has a minimal generating set 
$G = \{g_1,\ldots,g_s, \frac{\partial f_{\cA}}{\partial \ell} \}$ with homogeneous polynomials $g_i \in I_{\ell^\vee}$ whose degree is at most $|\cA| - 1$. 

\item[(b)] If $G$ is a minimal generating set of $r_\cA$  as in (a), then 
\[
r_\cA \cap I_{\ell^\vee} = (g_1,\ldots,g_s, \frac{\partial f_{\cA}}{\partial \ell} \cdot I_{\ell^\vee} ). 
\]
\end{itemize}

\end{corollary}

\begin{proof}
Set $d = |\cA|$ and $I = r_\cA \cap I_{\ell^\vee}$. 

(a) Using that $\reg r_\cA = d-1$ by \Cref{cor:reg estimate sharp} and that $\frac{\partial f_{\cA}}{\partial \ell}$ can be chosen as  minimal generator of $r_\cA$ by \Cref{prop:add general point}(a), it follows that $r_\cA$ has a minimal generating set 
$G = \{g_1,\ldots,g_s, \frac{\partial f_{\cA}}{\partial \ell} \}$ with homogeneous polynomials $g_i$ whose degree is at most $|\cA| - 1$. We have to show that we may assume that the polynomials $g_i$ are in $I_{\ell^\vee}$. Since $\frac{\partial f_{\cA}}{\partial \ell}$ has degree $d-1$ this is automatically true if $\deg g_i < d-1$  for any $i$, because $r_\cA = (I, \frac{\partial f_{\cA}}{\partial \ell})$ by \Cref{prop:add general point}(a). Assume  that for some $i$, $\deg g_i = d-1$ and $g_i$ is not in  $I_{\ell^\vee}$. Using that $I_{\ell^\vee}$ is the 
ideal of a point as we did above, we get $[(I_{\ell^\vee}, \frac{\partial f_{\cA}}{\partial \ell})]_{d-1} = [S]_{d-1}$, which implies 
\[
g_i = \tilde{g}_i + \lambda \frac{\partial f_{\cA}}{\partial \ell} 
\]
for some $\tilde{g}_i \in I_{\ell^\vee}$ and $\lambda \in K$. Exchanging $\tilde{g_i}$ for $g_i$ in $G$ successively, Claim (a) follows. 

(b) Using $g_i \in I_{\ell^\vee}$, one gets
\[
(g_1,\ldots,g_s, \frac{\partial f_{\cA}}{\partial \ell} \cdot I_{\ell^\vee} ) \subseteq (g_1,\ldots,g_s, \frac{\partial f_{\cA}}{\partial \ell}) \cap  I_{\ell^\vee} = I.
\]
\Cref{prop:add general point}(c) yields that the degrees of  the minimal generators of $I$ are at most $d-1$. Taking into account \Cref{prop:add general point}(a) and 
\[
\dim_K [r_\cA]_{d-1} = \dim_K [I]_{d-1}   -1 
\]
by the exact Sequence \eqref{eq:SES}, it follows that if there is any minimal generator of $I$ that is not in $(g_1,\ldots,g_s)$ then it is of the form $m \cdot \frac{\partial f_{\cA}}{\partial \ell}$ for some $m \in [S]_1$. However, since $I_{\ell^\vee}$ is a prime ideal and $\frac{\partial f_{\cA}}{\partial \ell}$ is not in $I_{\ell^\vee}$, we get  $m \in I_{\ell^\vee}$, as claimed. 
\end{proof}


\section{Residual-like ideals} 
\label{sect:residuals}

In this section, we introduce  a class of ideals, called residual-like ideals, that contains all general residuals. Residual-like ideals define particularly structured non-reduced schemes. It would be interesting to explore to what extend results on general residuals can be extended to any residual-like ideal. We leave this for future investigations. Here, we characterize for later use which residual-like ideals are actually general residuals. 

Let $\cA$ be a finite set of hyperplanes in $\PP^n$ defining a hyperplane arrangement $\bigcup_{H \in \cA} H$. 
Recall that for any $P \in S (A)$, we denote by $t_P \ge 2$ 
 the number of hyperplanes in $\cA$ that contain $P$. 
Let $\ell \in [S]_1$ be a general linear form.  By \Cref{cor:decomposition of residual}, 
the general residual of the hyperplane arrangement $\cA$ with respect to $\ell$ is 
\[
r_\cA = \bigcap_{P \in S(\cA)} (\ell_P, I_P^{t_P - 1}), 
\]
where $\ell_P$ is a linear form defining the linear span of $P$ and $\ell^\vee$ in $\PP^n$. 

Below, we will consider ideals with a similar structure, which we codify now.

\begin{definition}
    \label{def:residual-like} 
Let $Y \subset \PP^n$ be a non-empty finite set of codimension two linear subspaces and let $E = (e_P)_{P \in Y}$ be a sequence of positive integers. Choose a general point $Q \in \PP^n$ and denote by $\ell_P$ a linear form defining the linear span of $P \in Y$ and $Q$ in $\PP^n$. Consider the ideal 
\begin{equation}
    \label{eq:res-like} 
I(Y, E)  = \bigcap_{P \in Y} (\ell_P, I_P^{e_P}). 
\end{equation}
We refer to it as the \emph{residual-like ideal determined by $Y$ and $E$ (with respect to $Q$)}. 
\end{definition} 

\begin{remark}
Note that each of the ideals $(\ell_P, I_P^{e_P})$ is a complete intersection generated by two polynomials, of degree one and  $e_P$, respectively. Thus, the right-hand side of Equation \eqref{def:residual-like} is the unique minimal primary decomposition of $I(Y, E)$. 
\end{remark}

If $n = 2$ then $Y$ is a finite set of points and the scheme defined by $I(Y,E)$ is contained in the fat point scheme $\sum_{P \in Y} e_P P$. Algebraically, this means
\[
\bigcap_{P \in Y} I_P^{e_P} \subseteq I(Y, E) \subseteq I_Y = \bigcap_{P \in Y} I_P. 
\]

Considering again any $n \ge 2$, 
observe that $I(Y,E)$ defines a locally complete intersection scheme, whereas $\sum_{P \in Y} e_P P$ does not have this property in general. 

If $|Y| = 1$, then $I(Y, E)$ is the general residual of any hyperplane arrangement consisting of $e_P + 1$ hyperplanes of $\PP^n$ containing the unique subspace $P \in Y$.  If $Y$ has at least two elements it is no longer true that every ideal $I(Y, E)$ is a general residual.  

\begin{proposition}
    \label{prop:char residuals} 
Consider $Y, E$ and $Q$ as in \Cref{def:residual-like}. Assume  $|Y| \ge 2$ and denote by $\cL$ the set of hyperplanes of 
$\PP^n$ that are spanned by two distinct subspaces of $Y$. (Note that $\cL$ can be the empty set.) Set 
\[
d - 1 = \max \Big \{ \sum_{P \in Y \cap H}  e_P \; \mid \; H \in \cL \Big \},  
\]
where $Y \cap H$ denotes the subset of elements of $Y$ that are contained in $H$ (and not the scheme-theoretic intersection of $Y$ and $H$)  and $d = 1$ if $\cL = \emptyset$. 
Then the following two conditions are equivalent: 
\begin{itemize}

\item[(i)] The ideal $I(Y, E)$ is the general residual of a hyperplane arrangement with respect to the linear form $Q^\vee$. 

\item[(ii)] The set 
\[
\cK = \Big \{ H \in \cL \; \mid \; \sum_{P \in Y \cap H}  e_P = d-1 \Big \}
\]
has cardinality $d$ and, for each $P \in Y$, there are exactly $e_P +1$ hyperplanes in $\cK$ that contain $P$.  In this case, 
$I(Y, E)$ is the general residual of the hyperplane arrangement determined by $\cK$ and $Y = S(\cK)$,  and so $t_P = e_P+1$.  
\end{itemize}

\end{proposition}

\begin{proof} 
In the argument, we will consider various sets of codimension two subspaces. Thus,  for any such set $Y$, we write $\cL (Y)$ for the set of hyperplanes that are spanned by two distinct subspaces of $Y$. 

We begin by showing that (i) implies (ii) using induction on the cardinality of $\cA$. 
Assume 
\[
I(Y,E) = r_\cA = \bigcap_{P \in S(\cA)} (\ell_P, I_P^{t_P - 1}), 
\]
for some hyperplane arrangement  $\cA$. Set $d = |\cA|$. 
Since we assume that $Y = S(\cA)$ has at least two elements we get $d \ge 3$. Moreover, if $d = 3$ then the three hyperplanes in $\cA$ must be in general position, i.e., $|S(\cA)| = 3$ and $t_P = 2$ for each $P \in S(\cA)$. It follows that $I(Y, E)$ satisfies the conditions in (ii).

Now assume $ d = |\cA|  \ge 3$. Let $H \subset \PP^n$ be any hyperplane that is not in $\cA$. We have to show that the residual 
$r_{\cA + H}$ also satisfies the conditions in (ii). To distinguish multiplicities for $\cA$ and for $\cA + H = \cA \cup \{H\}$ write 
\[
r_{\cA + H}  = \bigcap_{P \in S(\cA + H)} (\ell_{P}, I_P^{t_{P,  \cA +H} - 1}).
\]
Notice that $H$ meets each of the $d$  hyperplanes of $\cA$ in a codimension two subspace $P \in S(\cA + H)$. Moreover, by the definition of $t_{P,  \cA +H}$, the hyperplane $H$ intersects precisely $t_{P,  \cA +H} - 1$ hyperplanes of $\cA$ in $P$. It follows that 
\[
\sum_{P \in S(\cA + H) \cap H} (t_{P,  \cA +H} - 1) = d.  
\]

Consider now any hyperplane $L \neq H$ of $\PP^n$. If $L$ is in $\cA$, then we apply the above argument to $(\cA +H) - L$ and obtain
\[
\sum_{P \in S(\cA + H) \cap L} (t_{P,  \cA +H} - 1) = d.  
\]
Setting $e_P = t_{P, \cA +H} - 1$, this shows that the set
\[
\cK = \Big \{ L \in \cL(S(\cA + H)) \; \mid \; \sum_{P \in S(\cA + H) \cap L}  e_P = d \Big \}
\]
contains $\cA + H$. To show equality 
it remains to check that 
\[
\sum_{P \in S(\cA + H) \cap L} (t_{P,  \cA +H} - 1) < d  
\]
if $L \subset \PP^n$ is a hyperplane that is not in  $\cA+H$. Consider the codimension two subspace $P' = L  \cap H$. If $P'$ is not in $S (\cA + H)$ then one has, for any $P \in S (\cA + H) \cap L$ that $P \in S(\cA)$ and  $t_P = t_{P,  \cA +H}$. Thus, the induction hypothesis gives 
\[
\sum_{P \in S(\cA + H) \cap L} (t_{P,  \cA +H} - 1)  = \sum_{P \in S(\cA) \cap L} (t_{P} - 1)  < d-1. 
\]
If $P'$ belongs $S (\cA + H)$ then $t_{P,  \cA +H} = 2$ if $P$ is not in $S(\cA)$ and $t_{P,  \cA +H} = t_P + 1$ otherwise. In both cases we get 
\[
\sum_{P \in S(\cA + H) \cap L} (t_{P,  \cA +H} - 1) \le  1 +  \sum_{P \in S(\cA) \cap L} (t_{P} - 1)  < 1 + d-1 = d, 
\]
as desired. 
\smallskip

We now turn to the reverse implication  (ii) implies (i). Since each $e_P$ is positive, Condition (ii) gives that  every $P \in Y$ is contained in 
$e_P + 1 \ge 2$ hyperplanes of $\cK$, and so $P$ is in $S(\cK)$ and $e_P + 1 = t_{P, \cK}$. It also follows that $Y \subseteq S(\cK)$. We need to show equality.  
To this end, assume $S (\cK) - Y$ is not empty and set $e_{P'} = t_{P', \cK} - 1 \ge 1$  for  $P' \in S (\cK) - Y$. 
Let $H \in K$ be a hyperplane containing $P'$.
Then we obtain 
\begin{align*}
1 \le e_{P'} \le \sum_{P \in (S (\cK) - Y) \cap H} e_P & = \sum_{P \in S (\cK) \cap H} e_P  - \sum_{P \in  Y \cap H} e_P \\
& =  \sum_{P \in  S(\cK) \cap H} (t_{P, \cK} - 1)  - (d-1)\\
& = 0,  
\end{align*}
where we used (ii) to get the equality in the second line and applied the first implication to $\cK$ to see the last equality. This contradiction proves that $Y = S(\cK)$. 
Applying now \Cref{cor:decomposition of residual} to the hyperplane arrangement determined by $\cK$, we see that the general residual $r_\cK$ of $\cK$ with respect to $Q^\vee$ equals $I(Y,E)$. 
\end{proof}

Recall that in \Cref{cor:reg estimate sharp} we determined the Castelnuovo-Mumford regularity of any general residual. It would be interesting to establish an upper bound on the regularity of an arbitrary residual-like ideal, which is sharp for any general residual.


\section{Bounding the degree of the Jacobian ideal} 
\label{sect:degree bounds Jacobian}


 For a curve in $\PP^2$, the global Tjurina number $\tau(C)$  is a measure of its singularities. In particular, one might ask what is the maximum possible Tjurina number given the degree of the curve, and in the case of line arrangements one can also ask for the minimum. This problem was studied by Du Plessis and Wall  in \cite{PW} in general, and by Beorchia and Mir\'o-Roig  \cite{BMR} for conic-line arrangements. In both cases, it was observed that if $\cA = \cA (f_{\cA})$ consists of $d$ concurrent lines then $\Jac (f_{\cA})$ is a complete intersection and $\tau (\cA) = (d-1)^2$ is the maximum possible. After this case, \cite{BMR} showed that for conic-line arrangements, there is a gap: $ \tau(C) \leq (d-1)(d-2)+1$, and they classify the conic-line arrangements for which equality holds. In particular, one possibility is a line arrangement (no conics) with $d-1$ of the lines concurrent. In this section we continue the study of possible Tjurina numbers for line arrangements. We establish that there is another gap in the range of possible $\tau (\cA)$. We also provide new arguments for the quoted results of du Plessis and Wall and of Beorchia and Mir\'o-Roig, and we  establish a lower bound.

For any hyperplane arrangement $\cA \subset \PP^n$, the global Tjurina number $\tau (\cA)$ is equal to the degree of the Jacobian ideal $\Jac (f_{\cA})$. Thus, combined with our results on restrictions of Jacobian ideals  in Section \ref{sec:reduction to plane}, the result on line arrangements implies an analogous statement for arbitrary hyperplane arrangements in any projective space, which we state as Corollary  \ref{cor:deg Jacobian}.

The following estimates comprise the main result of this section: 

\begin{theorem}
   \label{thm:deg Jacobian} 
For any arrangement  $\cA = \cA (f_{\cA})$ of $d \ge 3 $ lines in $\PP^2$, one has: 
\begin{itemize} 
\item[(a)] 
\[
\deg (\Jac (f_{\cA})) \ge \binom{d}{2},  
\]
and equality is true if and only if no three lines of $\cA$ are concurrent. 

\item[(b)] If the $d$ lines of $\cA$ are not concurrent then 
\[
\deg (\Jac (f_{\cA})) \le d^2 - 3d + 3.  
\]
Furthermore, equality is true if and only if $d-1$ lines, but not $d$ lines of $\cA$ are concurrent.  

\item[(c)] If $\cA$ does not have a subset of $d-1$ concurrent lines  (hence in particular $d \geq 4$) then 
\[
\deg (\Jac (f_{\cA})) \le d^2 - 4d + 7.  
\]
Moreover, if $d \ge 5$ then equality is true if and only if $\cA$ consists of $d-2$ concurrent lines and the two other lines meet in a point which is on one of the $d-2$ lines. 
\end{itemize} 
\end{theorem}

\begin{remark}
(i)
Part (b) corresponds to the mentioned results by Du Plessis and Wall (see \cite[Theorem 3.2]{PW} and Beorchia and Mir\'o-Roig (see \cite[Theorem 3.5]{BM} and also \cite[Proposition 4.7]{DIM}, \cite[Proposition 1.1]{PW2}). 

(ii)
There is not another immediate gap in the range of degrees of $\Jac (f_{\cA})$ below the bound given in (c). Indeed, if $\cA$ consists of $d-2$ concurrent lines and two more generic lines then 
\[
\deg (\Jac (f_{\cA})) =  d^2 - 4d + 6. 
\]
\end{remark}

\begin{proof}[Proof of \Cref{thm:deg Jacobian}] 
We argue by using the general residual $r_{\cA}$ of $\cA$. Since 
\[
\deg r_\cA = d (d-1) - \deg (\Jac (f_{\cA}))
\]
it is enough to show the corresponding bounds for $r_\cA$.  Recall that by \Cref{cor:degree change adding}, we know for any 
 hyperplane $H = V(x)$ that is not in $\cA$ that 
\begin{align}
  \label{eq:degree change} 
\deg r_{\cA + H} = \deg r_\cA + |\cA \cap H| \le  \deg r_\cA + |\cA|,  
\end{align} 
where 
\[
\cA \cap H  = \{ H' \cap H \; \mid \; H' \in \cA \}. 
\] 
In particular, $|\cA \cap H|$ counts the points in the support of the scheme $(\bigcup_{H' \in \cA} H' ) \cap H$ (i.e. it ignores multiplicity). 
Moreover, every primary component of $r_\cA$ is supported at a point and its degree is one less than the number of lines in $\cA$ passing through this point (see \Cref{prop:decomposition of residual}).  
\smallskip

We begin by establishing the lower bound on the degree of $\Jac (f_{\cA})$. 
 \Cref{thm:initial degree} gives that $\Jac(f_\cA)^{top}$ does not contain any nonzero polynomial of degree $d - 2$. 
 Since the Hilbert function of a zerodimensional scheme $X$ is strictly increasing until it is equal to the degree of $X$ it
  follows that 
 \begin{align*}
 \deg (\Jac (f_{\cA})) = \deg (\Jac (f_{\cA})^{sat}) & \ge \dim_K [S/\Jac (f_{\cA})^{sat}]_{d-2}   \\
 & = \dim_K [S]_{d-2} = \binom{d}{2},  
 \end{align*}
 as claimed. 
 To discuss sharpness of this bound assume $\cA$ contains three concurrent lines. The generic residual to the arrangement of these three lines alone has degree two. Thus, adding one line at a time, repeated application of Inequality \ref{eq:degree change} gives 
 \[
 \deg r_\cA \le 2 + 3 + 4 + \cdots + d-2 + d-1 = \binom{d}{2} - 1, 
 \]
 which yields $\deg (\Jac (f_{\cA})) \ge \binom{d}{2} + 1$ and so completes the proof of Claim (a) 
 \smallskip

 We now consider the upper bounds for $\deg (\Jac (f_{\cA}))$. We use induction on the cardinality $d$ of $\cA$ and denote by $H \subset \PP^2$ a line that is not in $\cA$. The following statements play a key role in our arguments. 
\smallskip  
 
\noindent
\textbf{Claim}: 
\begin{enumerate}
    \item \label{parta} If $|S(\cA)| \ge 2$ and $H \notin \cA$ then $|\cA \cap H| \ge 2$.

    \item Equality is true in (\ref{parta}) if and only if the following holds: $\cA$ is a disjoint union $\cA' \cup \cA''$ of nonempty subsets $\cA', \cA''$ of concurrent lines,  with at least two lines in one of the subsets, say $\cA'$, and if $\cA''| = 1$ then $H$ does not contain the center of the pencil $\cA'$. 
(We refer to this configuration of $\cA$ as a \emph{disconnected 2-pencil}.) 
\end{enumerate}

\smallskip

Indeed, notice that $|S(\cA)| = 1$ if and only if the $d$ lines of $\cA$ are concurrent. Moreover, $|S(\cA)| \ge 2$ forces $d \ge 3$.  Hence, $H$ meets $\cA$ in at least two distinct points, say, $P'$ and $P''$. If there is a line of $\cA$ that is not passing through either $P'$ or $P''$ then $H$ meets this line, and so $|\cA \cap H| \ge 3$. Therefore, the condition $|\cA \cap H| = 2$ forces that every line of $\cA$ passes through one of the points $P'$ and $P''$. 
It follows that  $\cA$ is a disconnected 2-pencil, which shows that (1) implies (2). The reverse implication is clear.
\smallskip

Note that Assertion (b) in the statement of the theorem is equivalent to the following statement:

\textbf{Claim (b'):} If the lines in $\cA$ are not concurrent then 
\begin{equation}
    \label{eq:degree residual weak}
\deg r_\cA \ge 2 |\cA| -3,  
\end{equation}
 and equality if true if and only if $|\cA| - 1$ lines of $\cA$, but not all lines of $\cA$ are concurrent. 
 \smallskip
 
We use induction on $d = |\cA| \ge 3$ to show Claim (b'). If  $d = 3$ then, by assumption, $S(\cA)$ consists of three distinct points and $\deg r_\cA = 3$, and so Inequality \eqref{eq:degree residual weak} is an equality. 

Let $d \ge 3$ and denote by $H$ any line that is not in $\cA$. As above, we write $\cA + H$ for $\cA \cup \{H\}$. We want to show Claim (b') for $\cA + H$ assuming it is true of $\cA$.
We consider two cases. 
\smallskip

\emph{Case 1}: Assume the lines in $\cA$ are concurrent.  It follows that $\deg r_\cA = d-1$. By assumption, the lines in $\cA + H$ are not concurrent, that is, $H$ meets $\cA$ in $d$ distinct points. Hence, we get from Equation \eqref{eq:degree change} 
\[
\deg r_{\cA + H} = d-1 + d = 2 | \cA + H | - 3, 
\]
as claimed. 
\smallskip

\emph{Case 2}:The lines in $\cA$ are not concurrent. Thus, the induction hypothesis gives 
\[
\deg r_\cA \ge 2d -3. 
\]
Using the above Claim and Equation \eqref{eq:degree change}, we get
\[
\deg r_{\cA + H} \ge 2d -3 + 2 = 2 | \cA + H | - 3, 
\]
establishing the asserted lower bound. Moreover, equality is true if and only if $\deg r_\cA = 2d -3$ and $| \cA \cap H | = 2$. The former condition forces that $\cA$ contains a pencil of $d-1$ lines, and now the latter condition implies that $H$  must pass through the center of this pencil. Hence $\cA +H$ contains a pencil of $ | \cA + H | -1 $ lines, which completes the argument for Claim (b'). 
 \smallskip
 
 We proceed similarly to establish Claim (c). Note that it is equivalent to the following statement:

\textbf{Claim (c'):} If $\cA$ does not have a subset of $|\cA| -1$ concurrent lines then 
\begin{equation}
\deg r_\cA \ge 3 |\cA| -7.  
\end{equation}
Moreover, if $|\cA| \ge 5$ then equality is true if and only if $\cA$ is the union of a pencil of $|\cA| - 2 \ge 3$ concurrent lines and a pencil of three concurrent lines. (Observe that this forces that the line that is spanned by the centers of the pencils belongs to $\cA$. Thus, we refer to this line arrangement as a \emph{connected 2-pencil}.) 
 \smallskip
 
 Again, we use induction of $d = |\cA|$. Notice that the assumption forces $d \ge 4$. Assume $d = 4$. Thus, one has $\deg r_\cA = 6$ if $\cA$ does not contain three concurrent lines and $\deg r_\cA = 5$ otherwise. 
 
 Let $H$ be a line not in $\cA$.  We now consider two cases.

\emph{Case 1}: If $\deg r_\cA = 6$ then precisely two lines of $\cA$ pass through any point of $S (\cA)$. It follows that $\cA \cap H$ 
consists of two points if and only if $H$ passes through two points of $S(\cA)$. In this case, one has $\deg r_{\cA+H} = 8$, $\cA + H$ is a connected 2-pencil and Claim (c') is sharp, as claimed. 
In any other case, $\cA \cap H$
contains at least three points.  Hence, Equation \eqref{eq:degree change} gives 
\[
\deg r_{\cA + H} \ge 6 + 3 >  8 = 3 | \cA + H | - 7.
\]
Thus, the lower bound in Claim (c') is satisfied but never sharp.

\emph{Case 2}: Assume $\deg r_\cA = 5$. Thus, $\cA$ has a subset $\cA'$ of three concurrent lines meeting in a point $P'$.  It follows that $\cA \cap H$ consists of two points if and only if $H$ passes through $P'$. But then $\cA' + H$ is a set of $4 = | \cA +H | -1$ concurrent lines, which is forbidden by our assumption on $\cA +H$ in Claim (c'). Hence, we obtain $| \cA \cap H | \ge 3$, and Equation \eqref{eq:degree change} yields
\[
\deg r_{\cA + H} \ge 5 + 3 = 3 | \cA + H | - 7, 
\]
establishing the lower bound in this case. Moreover, equality is true if and only if $| \cA \cap H | = 3$, which means that $H$ passes through a point of $S (\cA) - P'$. In this case, $\cA + H$ is the desired connected 2-pencil. 
This completes the argument for Claim (c') if $| \cA +H | = 5$. 

Assume now $d \ge 5$.  Again, we consider two cases. 
 
\emph{Case 1}: Assume $\cA$ contains a subset of $d-1$ concurrent lines meeting in a point $P'$. Then we know by Claim (b'),
\[
\deg r_\cA = 2 d-3. 
\]
By assumption on $\cA + H$, the line $H$ is not allowed to pass through $P'$. There are exactly two lines of $\cA$ that pass through any point 
of $S (\cA)$ other than $P'$. Hence,  either $H$ meets $\cA$ in such a point and $d-2$ intersection points with the other lines of $\cA$ or $| \cA \cap H| = d$. In any case, $| \cA \cap H| \ge d -1$, and Equation \eqref{eq:degree change} gives 
\[
\deg r_{\cA + H} \ge 2 d - 3 + d - 1 =  3 | \cA + H | - 7.
\]
Furthermore, equality is true if and only if $| \cA \cap H| = d -1$, which means that $\cA + H$ is the desired connected 2-pencil.

\emph{Case 2}: Assume $\cA$ has no subset of $d-1$ concurrent lines. Then the induction hypothesis gives that 
\[
\deg r_\cA \ge  3 d-7,  
\]
and equality is true if and only if $\cA$ is a connected 2-pencil with pencils of $d-2$ and $3$ lines.  We want to prove  Claim (c') for $\cA +H$. We consider two subcases. 

\emph{Case 2.1}: If $\cA$ is a connected 2-pencil with pencils of $d-2$ and $3$ lines, respectively, then the above Claim provides  $| \cA \cap H| \ge 3$, and so Equation \eqref{eq:degree change} yields 
\[
\deg r_{\cA + H} = \deg r_\cA + |\cA \cap H|  \ge 3 d - 7 + 3 =  3 | \cA + H | - 7.
\]
Moreover, $ |\cA \cap H| = 3$ if and only if $H$ passes through the intersection of the $d-2$ concurrent lines and meets the other two lines in two distinct points. In this case, $\cA + H$ is a connected 2-pencil with pencils of $d-1$ and $3$ lines, as desired.

\emph{Case 2.2}: Assume $\cA$ is not a connected 2-pencil with pencils of $d-2$ and $3$ lines, respectively. Then the induction hypothesis gives 
\[
\deg r_\cA \ge  3 d-6.  
\]
Thus, using Equation \eqref{eq:degree change} and the above Claim we obtain
\[
\deg r_{\cA + H} = \deg r_\cA + |\cA \cap H|  \ge 3 d - 6 + 2 =  3 | \cA + H | - 7.
\]
Furthermore, equality is true if and only $\deg r_\cA =  3 d-6$ and $ |\cA \cap H| = 2$. By the Claim, the latter condition means that $\cA$ is a disconnected 2-pencil with pencils of $s$ and $t$ lines, respectively, satisfying $s+t = d$ and $a, b \ge 2$. However, then we get 
\[
\deg r_\cA = s - 1 + t-1 + st = d-2 + s (d-s).    
\]
where we may assume $s \le \frac{d}{2}$. The function $g \colon \RR \to \RR$ with $g (s) = d-2 + s (d-s)$ is strictly increasing on the interval $[0, \frac{d}{2}]$. It follows that 
\[
\deg r_\cA \ge g(2) =  3d - 6,    
\]
and equality is true if and only if $s= 2$. In this case, $\cA + H$ is a connected 2-pencil with pencils of $d-1$ and $3$ lines, respectively. 
This completes  the argument. 
\end{proof}

As promised, \Cref{thm:deg Jacobian}  gives a general result for hyperplane arrangements.

\begin{corollary}
   \label{cor:deg Jacobian} 
Let  $\cA = \cA (f_{\cA})$ be an arrangement of $d \ge 3 $ hyperplanes in $\PP^n$ and let $S(\cA)$ be the top support of the singular locus of $\cA$. One has: 
\begin{itemize} 
\item[(a)] 
\[
\deg (\Jac (f_{\cA})) \ge \binom{d}{2},  
\]
and equality is true if and only if no three hyperplanes of $\cA$ share a common element of $S(\cA)$. 

\item[(b)] If the $d$ hyperplanes of $\cA$ do not all share a common element of  $S(\cA)$ then
\[
\deg (\Jac (f_{\cA})) \le d^2 - 3d + 3.  
\]
Furthermore, equality is true if and only if $d-1$ hyperplanes, but not $d$ hyperplanes, of $\cA$ share an element of $S(\cA)$.  

\item[(c)] If $\cA$ does not have a subset of $d-1$ hyperplanes sharing an element of $S(\cA)$ then 
\[
\deg (\Jac (f_{\cA})) \le d^2 - 4d + 7.  
\]
Moreover, if $d \ge 5$ then equality is true if and only if $\cA$ consists of $d-2$ hyperplanes sharing an element of $S(\cA)$ 
and two other hyperplanes, which share an element of $S(\cA)$ with one of the first $d-2$ hyperplanes.
\end{itemize} 
\end{corollary}

\begin{proof}
Intersecting $\cA$ with a general 2-dimensional linear space $\Lambda$ in $\mathbb P^n$ gives a line arrangement $\overline{\cA}$ in $\Lambda$ and a bijection between $S(\cA)$ and $S(\overline{\cA})$. Applying 
\Cref{lem"restriction top-dimensional} 
$n-2$ times, we obtain the top-dimensional part of the  Jacobian ideal of $\overline{\cA}$. The result follows since the degree of a subscheme is preserved under general hyperplane sections.
\end{proof}

We record the following consequence. 

\begin{corollary}
   \label{cor:char ci}
The Jacobian $\Jac (f_\cA)$ of any hyperplane arrangement determined by $\cA \subset \PP^n$ is a complete intersection if and only if all hyperplanes of $\cA$ contain a subspace $\Lambda \subset \PP^n$ of dimension $n-2$. 
\end{corollary} 

\begin{proof}
If all hyperplanes of $\cA$ contain a subspace $\Lambda \subset \PP^n$ of dimension $n-2$ then the minimal generating set of $\Jac (f_\cA)$ consists of two polynomials, and so $\Jac (f_\cA)$ is a complete intersection. 

For the reverse implication, assume there is no such subspace $\Lambda$. Note that this forces $d = |\cA| \ge 3$. Hence, \Cref{thm:deg Jacobian}  gives
\[
\deg (\Jac (f_\cA) ) \le d^2 - 3d + 3 < (d-1)^2, 
\] 
which shows that $\Jac (f_\cA)$ cannot be a complete intersection generated by two forms of degree $d-1$. 
\end{proof}


\section{Minimal free resolutions for certain arrangements} 
\label{sec:free res certain arr}

The goal of this section is to show that there are large classes of hyperplane arrangements for which the graded Betti numbers of the general residual and of the top-dimen\-sional part of the Jacobian ideal are combinatorially determined. Our methods use liaison-theoretic techniques. 

Given a hyperplane arrangement $\cA \subset \PP^n$ of $d$ hyperplanes, in \cite{MN},  we used a basic double link to describe the change of the top-dimensional parts of the Jacobian ideals when passing from $\cA$ to $\cA + H$, where $H$ is a generic hyperplane, that is, $H$ intersects the hyperplanes of $\cA$ in $d$ distinct codimension two subspaces. Here, we show that the change of general residuals can be described similarly. A key new insight is that this latter description is in fact true in some greater generality, namely for any {\it almost generic} hypeplane. It allows us to determine the graded Betti numbers of general residuals in many cases. 
We then show that the graded Betti numbers of the top-dimensional part of the Jacobian ideal are determined by those of the general residual if the latter is Cohen-Macaulay. We illustrate our results by establishing explicit results for arrangements satisfying a mild condition, introduced in \cite{MNS}. Then we consider some arrangements which do not satisfy this condition. 

\subsection{Adding an almost generic hyperplane} 
\label{subsec:adding almost generic} 
\mbox{ } \\
We use the notation introduced in \Cref{sec:residuals}. The following result is a central tool for this section.

\begin{proposition}
    \label{prop:add almost generic hyperplane} 
Let $\cA \subset \PP^n$ be any set of $d$ hyperplanes. Consider a hyperplane $H$ of $\PP^n$ that is defined by a linear form $x \in S$ and not in $\cA$.  If  $H$ contains at most one subspace $P \in S(\cA)$, then the general residual of $\cA + H$ is a basic double link of $r_\cA$, that is, 
\[
r_{\cA+H} = x \cdot r_\cA + (p), 
\]
where $p \in r_\cA$ is a form of degree $|\cA \cap H|$, $\cA \cap H$ denotes the set of codimension two subspaces $H' \cap H$ with $H' \in \cA$ and $x, p$ is a regular sequence. 
\end{proposition}

If a hyperplane $H$ satisfies the condition in the above statement, we say that $H$ is \emph{almost generic} (with respect to $\cA$). If $H$ does not contain any subspace of $S(A)$ then $H$ is a called \emph{generic} hyperplane (with respect to $\cA$).

\begin{proof}[Proof of \Cref{prop:add almost generic hyperplane}]
According to \Cref{cor:decomposition of residual}, we have
\[
r_\cA = \bigcap_{P \in S(\cA)} (\ell_P, I_P^{t_P - 1}), 
\]
where $\ell_P$ is a linear form defining the linear span of $P$ and $\ell^\vee$ in $\PP^n$. We consider two cases. 

First assume that $H$ is a generic hyperplane for $\cA$. Then, we get $|\cA \cap H| = |\cA| = d$ and 
\[
r_{\cA+H} = r_\cA \cap \bigcap_{P \in \cA \cap H} I_P.   
\]
Since $S(\cA) \cap H = \emptyset$ by the genericity of $H$, for each $P \in \cA \cap H$ there is a unique hyperplane in $\cA$ containing $P$. This hyperplane is defined by a linear divisor of $f_\cA$. It follows that 
\[
\bigcap_{P \in \cA \cap H} I_P = (x, f_\cA).  
\]
Since $x, f_\cA$ is a regular sequence and $f_\cA$ is in $r_\cA$, the basic double link $x \cdot r_\cA + (f_\cA)$ is unmixed and has degree $d + \deg r_\cA$. It is contained in 
\[
r_{\cA+H} = r_\cA \cap  (x, f_\cA).   
\]
Comparing degrees we conclude that $r_{\cA+H} = x \cdot r_\cA + (f_\cA)$, as desired. 

Second, assume there is unique subspace of $S(\cA)$ that is contained in $H$. Denote it by $Q$. 
 We begin with the special case where $|S(\cA)| = 1$. Thus, by  \Cref{cor:decomposition of residual},
\[
 r_\cA = (\ell_Q, I_Q^{d-1}),  
\]
and since $x \in I_Q$ as $Q$ is in $H$ by assumption, we get 
\[
(\ell_Q, I_Q^{d} ) = r_{\cA + H} =  x \cdot r_\cA + (\ell_Q)
\]
as desired. 

 Now assume $|S (\cA)| \ge 2$. 
Define ideals 
\[
\fa = \bigcap_{P \in S(\cA) - Q} (\ell_P, I_P^{t_P - 1})
\]
and 
\[
\fb = (\ell_Q , I_Q^{t_Q -1}) = (\ell_Q, g_Q), 
\]
where $g_Q$ is the product of $t_Q - 1$ distinct linear divisors of $f_\cA$ that each define a hyperplane containing $Q$.  Thus, $r_\cA = \fa \cap \fb$. 

Using \Cref{cor:decomposition of residual} again, we get 
\[
r_{\cA + H} = \fa \cap (\ell_Q, x g_Q) \cap \bigcap_{P \in (\cA \cap H) - Q} I_P. 
\]
For each $P \in (\cA \cap H) - Q$, there is exactly one linear divisor of $f_\cA$ that defines a hyperplane containing $P$. Hence $I_P$ is generated by this divisor and $x$. Denote by $h$ the product of these divisors. Then it follows that 
\[
\bigcap_{P \in (\cA \cap H) - Q} I_P = (x, h), 
\]
where the degree of $h$ is $|\cA \cap H| - 1$. Notice the following containment
\[
(\ell_Q, x g_Q) \cap  (x, h) \supseteq  (x \ell_Q, x g_Q, \ell_Q h) = x \cdot \fb + (\ell_Q h). 
\]
Since $x, \ell_Q h$ is a regular sequence and $\ell_Q h \in \fb$, the ideal on the right-hand side is a basic double link.  So it is unmixed of codimension two with degree $\deg \fb + \deg \ell_Q h = |\cA \cap H| + t_Q -1$, which is equal to the degree of the unmixed ideal on the left-hand side. Hence, the two ideals are equal, and we obtain
\[
r_{\cA + H} = \fa \cap \big (x \cdot \fb + (\ell_Q h) \big ).  
\]
By definition of $h$, we have $f_\cA = h \cdot q$, where $q$ is the product of linear divisors of $f_\cA$ that define hyperplanes containing $Q$. Consider any $P \in S(\cA) - Q$. There are $t_P$ linear divisors of  $f_\cA$ that define hyperplanes containing $P$. At most one of them also contains $Q$. Hence, there are at least $t_P - 1$ linear divisors of  $f_\cA$ that do not have the latter property, and so these divisors must divide $h$. It follows that $h$ is in $(\ell_P, I_P^{t_P-1})$ for every $P \in S(\cA) - Q$, and hence $h \in \fa$. Since $\ell_Q$ is in $\fb$, we get  
\[
r_{\cA + H} = \fa \cap \big (x \cdot \fb + (\ell_Q h) \big ) \supseteq x \cdot (\fa \cap \fb) + (\ell_Q h) = x \cdot r_\cA + (\ell_Q h),  
\]
where $\ell_Q h$ is in $r_\cA$ and $x, \ell_Q h$ are a regular sequence. Hence the ideal on the right-hand side is a basic double link of $r_\cA$, and so it is unmixed and its degree is 
\[
\deg \big ( x \cdot r_\cA + (\ell_Q h) \big ) = \deg r_\cA + \deg h  + 1. 
\]
By \Cref{cor:degree change adding}, the unmixed ideal on the left-hand side has degree
\[
\deg r_{\cA+ H} = \deg r_\cA + |\cA \cap H| = \deg r_\cA + \deg h + 1. 
\]
Comparing degrees, we conclude that the two ideals are equal, that is, $r_{\cA + H} = x \cdot r_\cA + (\ell_Q h)$, as desired. 
\end{proof}

\begin{remark}
   \label{rem:adding alomost gen hyperplane}
The proof of \Cref{prop:add almost generic hyperplane} gives a description of the form $p$ mentioned in its statement. In fact, if $S(\cA) \cap H = \emptyset$ then we may use $p = f_\cA$. If $S(\cA) \cap H = Q$ then $p$ may be chosen as the product of $\ell_Q$ and the linear divisors of $f_\cA$  that define a hyperplane not containing $Q$. 
\end{remark}

\begin{remark}
It is natural to wonder if  \Cref{prop:add almost generic hyperplane}  holds beyond the assumption that $H$ contains at most one element of $S(\cA)$. The first natural example is already a counterexample. Let $P_1,P_2$ be distinct points in $\mathbb P^2$ and let $\ell_1,\ell_2,\ell_2 \in I_{P_1}$, $m_1,m_2,m_3 \in I_{P_2}$ be linear forms such that none is in $I_{P_1} \cap I_{P_2}$. Let $x$ define the line joining $P_1$ and $P_2$. One can check (e.g. using a computer or using the methods below) that $r_\cA$ has minimal generators of degrees $(4,4,5,5)$ while $r_{\cA+H}$ has minimal generators of degrees $(4,4,6)$ (it is a free arrangement -- see \Cref{exa:two connected pencils}  below). Clearly it is impossible that $r_{\cA+H}$ is of the form $\ell \cdot r_\cA + (p)$ where $\deg \ell = 1$, no matter what the degree of $p$ is.
\end{remark}

Adding more than one almost general hyperplane, we get the following consequence. 

\begin{corollary}    
     \label{cor:add pencil} 
Let $\cA \subset \PP^n$ be any set of $d$ hyperplanes. Let $Q \in S(\cA)$ and let $\ell_1,\dots,\ell_r \in I_Q$ be linear forms defining $r$ distinct hyperplanes $H_1,\ldots,H_r$, respectively, so that no $H_i$ contains any element of $S(\cA +H_1 + \cdots +H_{i-1})$ other than possibly $Q$. 
Let $\mathcal H = H_1 + \cdots + H_r$. Then the general residual $r_{\cA + \mathcal H}$ of $\cA + \mathcal H$ is a basic double link of $r_\cA$; that is,
\[
r_{\cA+\mathcal H} = (\ell_1 \cdots \ell_r) \cdot r_\cA + (p), 
\]
where $p \in r_\cA$ is a form of degree $|\cA \cap H_1|$, $\cA \cap H_1$ denotes the set of codimension two subspaces $H' \cap H_1$ with $H' \in \cA$, and $(\ell_1\cdots \ell_r), p$ is a regular sequence. This is also true if $|S(\cA)| = 1$ and $S(\cA) \cap \mathcal H = \emptyset$. 
\end{corollary}

\begin{proof}
    This follows inductively from the easy fact that \[
    \ell_2 [\ell_1 \cdot r_\cA + (p)] + (p) = \ell_1 \ell_2 \cdot r_\cA + (p), 
    \]
    and so the sequence of two basic double links to obtain $r_{\cA+H_1 + H_2}$ from $r_\cA$ can be done in one step.
\end{proof}

\begin{proposition}[{\cite[Corollary 4.5]{MN-95}}]
   \label{prop:MFR under BDL}
Let $\fa \subset S$    be an unmixed homogeneous ideal of codimension two and let $\fb = q \cdot \fa + (p)$, where $p \in \fa$ and $q \in S$ are homogeneous and form a regular sequence (so $\fb$ is a basic double link of $\fa$). Put $h = \deg q$ and $e =  \deg p$.    
Assume that $\fa$ has a minimal free resolution
\[
0 \rightarrow \dots \rightarrow F_3 \rightarrow  F_2 \rightarrow 
 F_1 \rightarrow \fa \rightarrow 0. 
\]

\begin{itemize}
    \item[(a)] If $\fa$ has a minimal generating set which includes $p$ then $F_1 \cong G \oplus S(-e)$ for some free $S$-module $G$ and 
    $\fb$ has a minimal free resolution of the form
    \[
    0 \rightarrow \dots \rightarrow  F_3 (-h)  \rightarrow  F_2 (-h) \rightarrow  G(-h)  \oplus S(-e) \rightarrow \fb \rightarrow 0.
    \]
%

    \item[(b)] If $p$ is not part of any minimal generating set of $\fa$ then $\fb$ has a minimal free resolution of the form
\[
0 \rightarrow \dots \rightarrow  F_3 (-h)  \rightarrow 
\begin{array}{c}
 F_2(-h) \\
\oplus \\
S(-e-h) 
\end{array}
\rightarrow 
\begin{array}{c}
 F_1 (-h) \\
\oplus \\
S(-e)
\end{array}
\rightarrow \fb \rightarrow 0.
\]
    
\end{itemize}

\end{proposition}

Denoting by $\fm$ the maximal ideal $(x_0,\ldots,x_n)$, the condition on $p$ in Part (a) can be expressed as $p \notin \fm \cdot \fa$, whereas the condition in Part (b) is $p \in \fm \cdot \fa$.

\begin{corollary}
    \label{cor:MFR res after adding generic hyperplane}
Consider any set $\cA \subset \PP^n$ of $d$ hyperplanes and a hyperplane $H$ that is not in $\cA$. Assume $r_\cA$ has a graded minimal free resolution of the form 
\[
0 \to F_s \to \cdots \to F_1 \to r_\cA \to 0. 
\]
Then one has: 

\begin{itemize}

\item[(a)] If $H$ is a generic hyperplane with respect to $\cA$ then $r_{\cA + H}$ has a graded minimal free resolution of the form 
\[
0 \to F_s (-1)  \to \cdots \to F_3 (-1) \to 
\begin{array}{c}
 F_2(-1) \\
\oplus \\
S(-d-1) 
\end{array}
\to 
\begin{array}{c}
 F_1(-1) \\
\oplus \\
S(-d) 
\end{array}
\to r_{\cA+H} \to 0.  
\]

\item[(b)] Assume $H$ is an almost generic hyperplane with respect to $\cA$ so that, using the notation of \Cref{prop:add almost generic hyperplane}, $r_{\cA+H} = x \cdot r_\cA + (p)$. 

\begin{itemize}

\item[(i)] If $p \in \fm \cdot r_\cA$ then  $r_{\cA + H}$ has a graded minimal free resolution of the form 
\[
0 \to F_s (-1)  \to \cdots \to F_3 (-1) \to 
\begin{array}{c}
 F_2(-1) \\
\oplus \\
S(-|\cA \cap H|-1) 
\end{array}
\to 
\begin{array}{c}
 F_1(-1) \\
\oplus \\
S(-|\cA \cap H|) 
\end{array}
\to r_{\cA+H} \to 0.  
\]

\item[(ii)] If $p \notin \fm \cdot r_\cA$  then  $F_1 \cong G \oplus S(-|\cA \cap H|)$ for some free $S$-module $G$, and 
$r_{\cA + H}$ has a graded minimal free resolution of the form 
\[
0 \to F_s (-1)  \to \cdots \to F_3 (-1) \to 
 F_2(-1) 
 \to 
\begin{array}{c}
 G(-1) \\
\oplus \\
S(-|\cA \cap H|) 
\end{array}
\to r_{\cA+H} \to 0.  
\]

\end{itemize}

\end{itemize}

\end{corollary}


\begin{proof}
(a) By \Cref{prop:add almost generic hyperplane}, we have a basic double link $r_{\cA+H} = x \cdot r_\cA + (f_\cA)$. Moreover, the degrees of the minimal generators of $r_\cA$ are at most $d-1$ by \Cref{cor:reg estimate sharp}. Hence, $f_{\cA}$ is not a minimal generator of $r_\cA$, and we conclude using \Cref{prop:MFR under BDL}. 

(b) follows immediately by combining \Cref{prop:add almost generic hyperplane} and \Cref{prop:MFR under BDL}. 
\end{proof}

Both cases in Part (b) of \Cref{cor:MFR res after adding generic hyperplane} do occur, as shown in the following example. 

\begin{example}
   \label{exaMFR when adding almost generic} 
(i) Consider the line arrangement $\cB \subset \PP^2$ defined by $f_\cB = x_0 x_1 x_2$. Then its general residual is the ideal of the three coordinate points, and so $r_\cB = (x_0 x_1, x_0 x_2, x_1 x_2)$. 

The line $H$ defined by $x_0 + x_1$ is not generic, but almost generic for $\cB$ because it contains the coordinate point $Q = (0 : 0 : 1)$.  Thus, \Cref{prop:add almost generic hyperplane} gives
\begin{align*}
r_{\cB + H} & = (x_0 + x_1) \cdot r_\cB + (\ell_Q x_2). 
\end{align*}
Since $\ell_Q  x_2$ is part of a minimal generating set of $r_\cB$ \Cref{cor:MFR res after adding generic hyperplane} shows that both $r_\cB$ and $r_{\cB + H}$ have three minimal generators.

(ii) Consider the line arrangement $\cA \subset \PP^2$ defined by $f_\cA = (x_0^2 - x_1^2) ( x_0^2 -  x_2^2) (x_1^2 - x_2^2)$. 
Its general residual has three minimal generators.

The line $H$ defined by $x_0 + 2 x_1$ is not generic, but almost generic for $\cA$ because it contains the coordinate point $Q = (0 : 0 : 1)$.  Thus, \Cref{prop:add almost generic hyperplane} gives
\begin{align*}
r_{\cA + H} & = (x_0 + 2 x_1) \cdot r_\cA + (\ell_Q ( x_0^2 -  x_2^2) (x_1^2 - x_2^2)). 
\end{align*}
However, one checks that $\ell_Q ( x_0^2 -  x_2^2) (x_1^2 - x_2^2)$ is not a minimal generator of  $r_\cA$, and so the general residual $r_{\cA + H}$ has four minimal generators. More precisely, $r_\cA$ has three minimal generators of degrees $3, 4$ and $5$, and thus $r_{\cA + H}$ has four minimal generators of degrees $4, 5, 5$ and $6$ by \Cref{cor:MFR res after adding generic hyperplane}(b)(i). 
\end{example}

If $S/r_\cA$ is Cohen-Macaulay then its graded Betti numbers determine those of $\Jac (f_\cA)^{top}$. 
    
\begin{proposition}
    \label{prop:Betti top from residual} 
Consider a hyperplane arrangement $\cA \subset \PP^n$  with $|\cA| = d$, such that $S/r_\cA$ is Cohen-Macaulay. Then
the minimal graded free resolution of $r_\cA$ has the form
\[
0 \to G \to S(-d+1) \oplus F \to r_\cA \to 0,  
\]
and the minimal free resolution of $\Jac (f_\cA)^{top}$ is 
\[
0 \to F^* (-2 d+1) \to S(-d+1) \oplus G^*(-2d+1) \to  \Jac (f_\cA)^{top} \to 0. 
\]
\end{proposition}

\begin{proof}
According to \Cref{prop:add general point}(a), the general residual $r_\cA$ has a minimal generating set that includes $\frac{\partial f_\cA}{\partial \ell}$, where $\ell \in S$ is a general linear form. It follows that $r_\cA$ has a minimal free resolution as claimed. 
By definition, the ideals $ \Jac (f_\cA)^{top}$ and $r_\cA$ are linked by   $(f_{\cA}, \frac{\partial f_{\cA}}{\partial \ell})$. 
Consider now the mapping cone induced by the diagram
\begin{align*}
     \label{commutative diagram_multi factor}
\minCDarrowwidth20pt
\begin{CD}
0 @>>>  S(- 2d+1)  @>>>  S(-d+1) \oplus S(-d)  @>>>  (f_{\cA}, \frac{\partial f_{\cA}}{\partial \ell})  @>>>  0 \\
 @. @.  @.    @VVV \\
0 @>>>  G @>>>  S(-d+1) \oplus F   @>>>   r_\cA @>>>  0,   \\[3pt]
\end{CD}
\end{align*} 
where the vertical map is given by inclusion. 
The fact that $\frac{\partial f_\cA}{\partial \ell}$ is a minimal generator of both $(f_{\cA}, \frac{\partial f_{\cA}}{\partial \ell})$ and $r_\cA$ produces a cancellation. Taking it into account, the claimed minimal free resolution of $ \Jac (f_\cA)^{top}$ follows. 
\end{proof}

We illustrate the above results by a simple example. 

\begin{example}
    \label{exa:two connected pencils} 
Given two distinct codimension two subspaces $P_1, P_2$ of $\PP^n$  that meet in codimension 3 if $n \geq 3$, consider the hyperplane arrangement $\cA \subset \PP^n$ with $f_\cA = f g m$, where $m$ defines the (unique) hyperplane containing $P_1$ and $P_2$, $f$ is the product of linear divisors of $f_\cA$ other than $m$ that define hyperplanes containing $P_1$ and $g$ is the product of linear divisors of $f_\cA$ other than $m$ that define hyperplanes containing $P_2$. 
We refer to $\cA$ as two connected pencils. It is illustrated in Figure \ref{fig: two pencils, with line} if $n = 2$. %


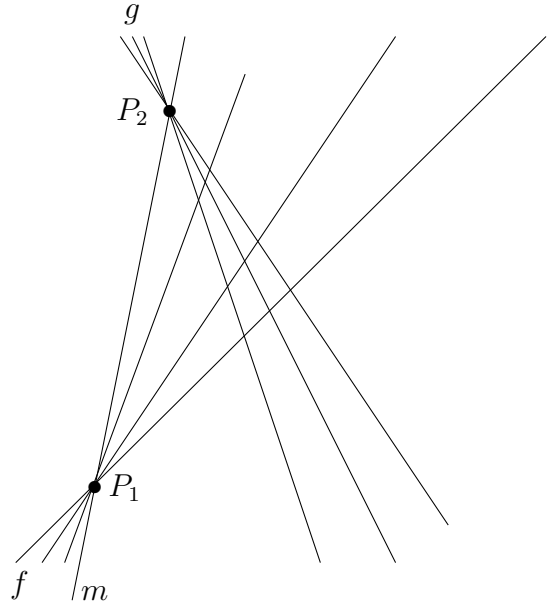
\begin{figure}[h]
\begin{center}
\begin{tikzpicture}

 \draw (-.7,-1) -- (4,6); 
 \draw (-1.05,-1) -- (6,6);
 \draw (-.4,-1) -- (2,5.5);

 \draw (.65,6) -- (3,-1);
 \draw (.5,6) -- (4,-1);
 \draw (.34,6) -- (4.7,-.5);


\draw (1.2,6) -- (-.3,-1.5);

\node at (0,0) {$\bullet$};  
\node at (1,5) {$\bullet$};  

\node at (.4,0) {$P_1$};
\node at (.5,5) {$P_2$};

\node at (0,-1.4) {$m$};
\node at (-1,-1.3) {$f$};
\node at (.5,6.3) {$g$};

\end{tikzpicture}
\end{center}
\caption{Two pencils plus the common line $m$}
\label{fig: two pencils, with line}
\end{figure}

This arrangement is supersolvable, hence free, and so the desired resolution of $\Jac(f_\cA)$ can be obtained via the addition-deletion theorem. However, we use our liaison approach because it showcases our methods in the simplest situation. 

Without loss of generality assume $a =  \deg f  \le b = \deg g$. By  \Cref{cor:add pencil},  we get
\begin{align*}
r_{\cA} & = g \cdot r_{\cA(mf)}  + (f \ell_{P_2}),  
\end{align*}
where the minimal generators of $r_{\cA(m f)}$ have degrees 1 and $a$.  Since $f \ell_{P_2}$ has degree $b+1 > a$, it  is not a minimal generator of  $r_{\cA(mf)}$ for degree reasons.  Thus,  using $h=b$ and $e = a+1$, \Cref{prop:MFR under BDL} gives the minimal free resolution 
\[
0 \rightarrow S(-a-b-1)^2 \rightarrow 
\begin{array}{c}
S(-a-1) \\
\oplus \\
S(-b-1) \\
\oplus \\
S(-a-b)
\end{array}
\rightarrow r_{\cA} \rightarrow 0.
\]

Applying \Cref{prop:Betti top from residual}, we see that $\Jac(f_\cA)^{top}$ has three minimal generators of degree $d-1$, where $d = |\cA| = a + b +1$. Hence,  $\cA$ is free and the minimal graded free resolution of  $\Jac(f_\cA) = \Jac(f_\cA)^{top}$ is 
\[
0 \rightarrow 
\begin{array}{c}
S(-d-a+1) \\
\oplus \\
S(-d-b+1)
\end{array}
\rightarrow
S(-d+1)^3 \rightarrow \Jac(f_\cA) \rightarrow 0.
\]
\end{example}

We apply the above results to specific classes of arrangements in the following subsections.


\subsection{Unions of disconnected pencils} 
\label{subsec:disconnected pencils} 
\mbox{ } \\
We recall the following Cohen-Macaulayness result.  

\begin{theorem}[{\cite[Theorem 3.7]{MNS}}] 
    \label{MNS result}
Let $\mathcal A$ be a hyperplane arrangement in $\mathbb P^n$ defined by a product $f_\cA$ of linear forms, and assume that no factor of $f_\cA$ is contained in the associated prime of any two nonreduced components of $\Jac (f_\cA)^{top}$. Then $S/Jac (f_\cA)^{top}$ is Cohen-Macaulay. 
\end{theorem}

The goal of this subsection is to determine the graded Betti numbers of $\Jac (f_\cA)^{top}$. Using the notation introduced in this article, the assumption of the above statement can be rephrased as saying that there is no hyperplane of $\cA$ that contains two subspaces of the top-dimensional part of the singular locus $S(\cA)$  that both have  multiplicities greater than two in $\cA$.

\begin{theorem}
\label{thm:MFR residual with special assumption}
Let $\cA \subset \PP^n$ be an  arrangement of $d$ hyperplanes. Denote by 
        \[
    S_0 (\cA) = \{ P \in S(\cA) \; \mid \; t_P \ge 3\}
    \]
the set of codimension two subspaces $P \in \PP^n$ that are contained in at least three distinct hyperplanes of $\cA$. Assume that no hyperplane in $\cA$ contains two or more elements of $S_0(\cA)$. Denote by $s$ the number of hyperplanes of $\cA$ that do not contain any $P \in S_0(\cA)$. Then the general residual of $\cA$ has a minimal free resolution of the form
\begin{align*}
0 \to S^{2 |S_0 (\cA)| + s - 1} (-d) \to  
\begin{array}{c}
S^{ |S_0 (\cA)| + s} (-d+ 1) \\
\ \ \ \ \ \oplus \\
{\displaystyle \bigoplus_{P \in S_0 (\cA)} S(-d+t_P -1) }
\end{array}
\to r_\cA \to 0. 
\end{align*}
In particular, $S/r_\cA$ is Cohen-Macaulay and even the graded Betti numbers are combinatorially determined. 
\end{theorem}

\begin{proof}
This follows by repeatedly applying \Cref{cor:MFR res after adding generic hyperplane} and induction on $d = |\cA|$. If $d = 2$ then $r_\cA$ defines a codimension two linear subspace, and the result is clear. If $d \ge 3$ choose any $H \in \cA$. By the induction hypothesis, we know the graded Betti numbers of $r_{\cA - H}$. If $H$ contains some $P \in S_0 (\cA - H)$ then $H$ is an almost generic line with respect to $\cA - H$, otherwise it is a generic line. In either case we conclude by applying \Cref{cor:MFR res after adding generic hyperplane}. 
We construct $\cA$ by adding one hyperplane at a time to produce first the pencils with multiplicities grater than two and second, if needed,  the pencils of multiplicitiy two. 
\end{proof}

If $\cA \subset \PP^2$ is a line arrangement then the lines through any point $P \in S_0 (A)$ form a pencil consisting of $t_P \ge 3$ lines. We refer to $P$ as the center of this pencil of lines in $\cA$. 
The assumptions of \Cref{thm:MFR residual with special assumption} on $\cA$ require that no two points of $S_0 (\cA)$ are contained in any line of $\cA$. With this picture in mind, we refer to the hyperplane arrangements in $\PP^n$ satisfying the assumption of \Cref{thm:MFR residual with special assumption}  as a \emph{union of disconnected pencils}.  

Using \Cref{prop:Betti top from residual}, the above result has the following immediate consequence. 

\begin{corollary}
     \label{cor:MFR with special assumption}
Let $\mathcal A \subset \PP^n$ be a union of disconnected pencils. Using the notation of  \Cref{thm:MFR residual with special assumption}, the graded minimal free resolution of $\Jac (f_\cA)^{top}$ has the form
 \begin{align*}
 0 \rightarrow
\begin{array}{c}
S^{ |S_0 (\cA)| + s-1} (-d) \\
\ \ \ \ \ \oplus \\
{\displaystyle \bigoplus_{P \in S_0 (\cA)} S(-d-t_P +2) }
\end{array}
\rightarrow S^{2 |S_0 (\cA)| + s} (-d+1) \rightarrow \Jac (f_\cA)^{top} \rightarrow 0.
 \end{align*} 
\end{corollary}

We illustrate this result in the simplest cases, namely for line arrangements with few pencils.

\begin{example}
   \label{exa:few pencils} 
Let $\cA \subset \PP^2$ be an arrangement of $d$ lines forming a union of disconnected pencils. 

(i) If there is exactly one pencil, that is, $\cA$ consists of $d$ concurrent lines then $\Jac (f_\cA)^{top}$ is a complete intersection (see \Cref{cor:char ci}), and its minimal free resolutions has the form
\[
0 \to S(-2d+2) \to S^2(-d+1) \to \Jac (f_\cA)^{top} \to 0. 
\]

(ii) Assume $\cA$ consists of two disconnected pencils with centers $P_1$ and $P_2$ (see Figure~\ref{two pencils, no line}). 


\begin{figure}[h] 
\begin{center}
\begin{tikzpicture}

 \draw (-.7,-1) -- (4,6); 
 \draw (-1.05,-1) -- (6,6);
 \draw (-.4,-1) -- (2,5.5);

 \draw (.65,6) -- (3,-1);
 \draw (.5,6) -- (4,-1);
 \draw (.34,6) -- (4.7,-.5);


\node at (0,0) {$\bullet$};  
\node at (1,5) {$\bullet$};  

\node at (.4,0) {$P_1$};
\node at (.5,5) {$P_2$};

\end{tikzpicture}
\end{center}
\caption{Two pencils without a common line.}
\label{two pencils, no line}
\end{figure}
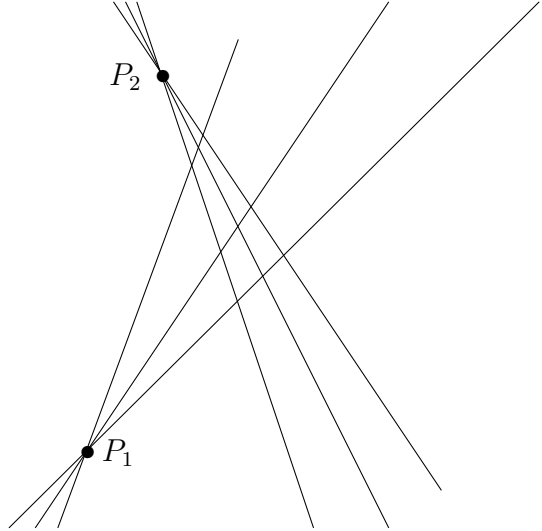

Then $\Jac (f_\cA)^{top}$ has a minimal free resolutions of the form
\[
0 \to 
\begin{array}{c}
S (-d) \\
\ \ \ \ \ \oplus \\
S(-d-t_{P_2} +2) \\
\oplus \\
S(-d-t_{P_1} +2)
\end{array} 
\to S^4(-d+1) \to \Jac (f_\cA)^{top} \to 0.
\]

In particular, $S/\Jac (f_\cA)^{top}$ is Cohen-Macaulay, but $\cA$ is not a free arrangement. 

\vspace{.2in}

(iii) Assume $\cA$ consists of three disconnected pencils with centers $P_1, P_2$ and $P_3$ 
(see Figure \ref{three pencils, no line}). 


\begin{figure}[h]
\begin{center}
\begin{tikzpicture}[scale=0.90]

 \draw (-.7,-1) -- (4,6); 
 \draw (-1.05,-1) -- (6,6);
 \draw (-.4,-1) -- (2,5.5);

 \draw (.65,6) -- (3,-1);
 \draw (.5,6) -- (4,-1);
 \draw (.34,6) -- (4.7,-.5);

 \draw (7,1) -- (-1,1);
 \draw (7,1.2) -- (-1,.5);
 \draw (7,.85) -- (-1,1.5);

\node at (0,0) {$\bullet$};  
\node at (1,5) {$\bullet$};  
\node at (5,1) {$\bullet$};  

\node at (.4,0) {$P_1$};
\node at (.5,5) {$P_2$};
\node at (5,1.5) {$P_3$};

\end{tikzpicture}
\end{center}
\caption{Three disconnected pencils.}
\label{three pencils, no line}
\end{figure}
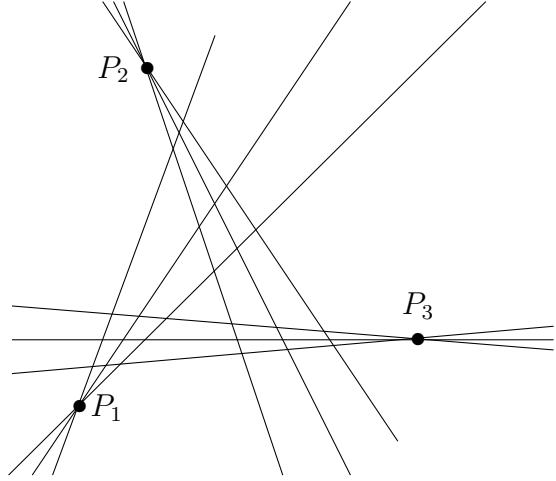

In this case,  $\Jac (f_\cA)^{top}$ has a minimal free resolution of the form
\[
0 \to 
\begin{array}{c}
S^2 (-d) \\
\ \ \ \ \ \oplus \\
S(-d-t_{P_3} +2) \\
\oplus \\
S(-d-t_{P_2} +2) \\
\oplus \\
S(-d-t_{P_1} +2)
\end{array} 
\to S^6(-d+1) \to \Jac (f_\cA)^{top} \to 0.
\]

\end{example}

\subsection{Liaison addition for residuals} 
\label{subsec:liaison addition residuals} 
\mbox{ } \\
Liaison addition for the top-dimensional parts of Jacobian ideals played a key role for obtaining the main results of \cite{MN}. In this subsection, we establish liaison addition for general residuals. The result is analogous to \Cref{LAthm1}.  Recall that we write $\cA (f)$ for the hyperplane arrangement defined by a product of linear forms $f$.

\begin{proposition}[Liaison Addition for General Residuals] 
    \label{prop:liaison add residuals} 
Let $\mathcal A ({fg}) = \mathcal A (f) \cup \mathcal A (g) \subset \PP^n$ be a hyperplane arrangement, where $f = f_1 \cdots f_s$ and $g = g_1 \cdots g_t$ are products of linear forms, such that 

\begin{itemize}

\item[(i)] \label{LA1} $\hbox{codim } (f_i, f_j,g) = 3$ whenever $i \neq j$.

\item[(ii)] \label{LA2} $\hbox{codim } (f, g_i, g_j) = 3$ whenever $i \neq j$.

\end{itemize}
 (Note that we do not assume that $\hbox{codim } (f_i, f_j, f_k) = 3$ or that $\hbox{codim } (g_i, g_j, g_k) = 3$ for $i, j, k$ distinct.)

Then $\mathcal A (f g)$ has the following properties:

\begin{itemize} 

\item[(\rm{a})] $\displaystyle r_{\cA (f g)} = r_{\cA (f)} \cap r_{\cA (g)} \cap \underbrace{\left [ \bigcap_{i,j} (f_i, g_j) \right ]}_{= \ (f,g)}.$

\item[(\rm{b})] $\displaystyle r_{\cA (f g)}  = g \cdot r_{\cA (f)} + f \cdot  r_{\cA (g)}$. In particular, $r_{\cA (f g)}$ is obtained by Liaison Addition from $r_{\cA (f)} $ and $r_{\cA (g)} $. 

\item[(\rm{c})] If $r_{\cA (f)}$ and $r_{\cA (g)}$ have graded minimal free resolutions of the form 
\[
0 \to F_u \to \cdots \to F_1 \to r_{\cA (f)} \to 0 
\]
and 
\[
0 \to G_v \to \cdots \to G_1 \to r_{\cA (g)} \to 0,  
\]
respectively, then 
\[
0 \to 
\begin{array}{c}
F_w (-d') \\
 \oplus \\
 G_w (-d) 
\end{array}
\to \cdots \to 
\begin{array}{c}
F_3 (-d') \\
 \oplus \\
 G_3 (-d) 
\end{array} 
\to 
\begin{array}{c}
F_2 (-d') \\
 \oplus \\
 G_2 (-d) \\
 \oplus \\
 S(-d-d')
\end{array}
  \to 
\begin{array}{c}
F_1 (-d') \\
 \oplus \\
 G_1 (-d) 
\end{array} \to r_{\cA (f g)} \to 0, 
\]
is a graded minimal free resolution of $r_{\cA (f g)}$, where $d = \deg f$, $d' = \deg g$,  $w = \max \{u, v \}$ and we define $F_i = 0$ if $u < i \le w$ as well as $G_j = 0$ if $v< j \le w$.  


\end{itemize}
\end{proposition}

\begin{proof} Observe that Condition (i) means that no subspace $P \in S(\cA (f))$ is contained in any hyperplane of $\cA (g)$. Similarly,  Condition (ii) gives that no subspace $P \in S(\cA (g))$ is contained in any hyperplane of $\cA (f)$. Hence the multiplicities of any $P \in 
S(\cA (f))$ are the same in $\cA (f)$ and $\cA (f g)$ and the analogous statement is true for any $P \in S(\cA (g))$. Thus, the description of the minimal primary decomposition of a general residual given in \Cref{cor:decomposition of residual} proves Claim (a). 

We know that $f$ is in $r_{\cA (f)}$ and that $g$ is in $r_{\cA (g)}$. Moreover, $f$ and $g$ are relatively prime by Conditions (i) and (ii).  Thus, Claim (b) is a consequence of Liaison Addition. It also follows that there is a short exact sequence
\[
0 \to S (-d-d')  \to r_{\cA (f)} (-d') \oplus  r_{\cA (g} (-d)   \to  r_{\cA (f g)} \to 0. 
\]
Thus, the mapping cone procedure shows that the resolution given in Claim (c) is a graded free resolution of $r_{\cA (f g)}$. It is in fact 
a minimal free resolution because $f$ is not a minimal generator of $r_{\cA (f)}$ and $g$ is not a minimal generator of $r_{\cA (g)}$ by \Cref{thm:reg residual} (see \cite[Proposition 4.4]{MN-95}). 
\end{proof}

We illustrate the above result by a simple example. 

\begin{example}
   \label{exa:liai add}
Consider the plane arrangement $\cA (f) \subset \PP^3$ defined by 
\[
f = x_0 x_1 x_2 x_3(x_0 + x_1)(x_1 + x_2)(x_2 + x_3)(x_0 + x_3)(x_0 + x_1 + x_2+x_3) 
\]
(see \cite[Example 4.3(4)]{MNS}). 
Its general residual is not Cohen-Macaulay. It has a graded minimal free resolution of the form 
\[
0 \to S (-9) \to 
\begin{array}{c}
S^3 (-9) \\
\oplus \\
S^4(-8) 
\end{array} 
\to 
\begin{array}{c}
S (-8) \\
 \oplus \\
S^6(-7) 
\end{array} 
\to r_{\cA (f)} \to 0. 
\] 
Let $\cA (g) \subset \PP^3$ be an arrangement defined by three general planes containing a line $P$ that is not in $S (\cA (f))$. Thus, 
$r_{\cA (g)}$ is a complete intersection of type $(1,2)$. Hence, \Cref{prop:liaison add residuals}  shows that $r_{\cA (f g)}$ has 
a graded minimal free resolution of the form 
\[
0 \to S (-12) \to 
\begin{array}{c}
S^5 (-12) \\
\oplus \\
S^4(-11) 
\end{array} 
\to 
\begin{array}{c}
S^2 (-11) \\
 \oplus \\
S^7(-10) 
\end{array} 
\to r_{\cA (f g)} \to 0. 
\] 
\end{example}


\subsection{Further arrangements} 
\label{subsec:further arr} 
\mbox{ } \\
The goal of this section is to illustrate how the results of Subsections~\ref{subsec:adding almost generic} and \ref{subsec:liaison addition residuals} can be used to determine the graded Betti numbers of hyperplane arrangements that do not satisfy the assumption made in 
\Cref{MNS result}. 
The resulting ideals   $\Jac (f_\cA)^{top}$ turn out to be Cohen-Macaulay. 

We begin with a complete analysis of hyperplane arrangements $\cA \subset \PP^n$ with the property that $S (\cA)$ contains three subspaces $P_1, P_2$ and $P_3$ such that every other  subspace of $S (\cA)$  has multiplicity two and any two of $P_1, P_2, P_3$ span a hyperplane. We refer to $\cA$ as an arrangement with three pencils with centers $P_1, P_2$ and $P_3$. There are several cases, depending on  whether two centers $P_i$ and $P_j$ are connected, that is, the hyperplane spanned by $P_i$ and $P_j$, belongs to $\cA$ or not.

\begin{example}
   \label{exa:three pencils} 
In each case, we write $f_\cA$ as a product  of $ f g h$  and some linear forms connecting centers, where $f, g, h$ have degrees $a, b$ and $c$, respectively. 

(i) If no two centers are connected, the graded Betti numbers have been determined in \Cref{thm:MFR residual with special assumption} and \Cref{cor:MFR with special assumption}; see also \Cref{exa:few pencils}(iii).

(ii) Assume that exactly two of the centers, say $P_1$ and $P_2$, are connected. See Figure~\ref{three pencils, one line} if $n = 2$. 
 
 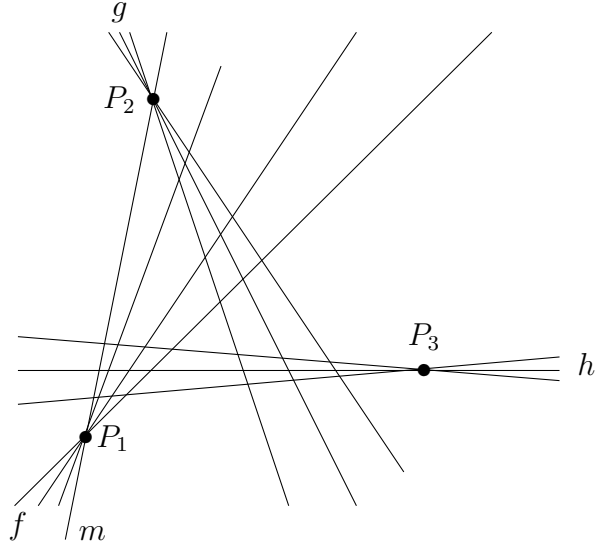
\begin{figure}[h]
\begin{center}
\begin{tikzpicture}[scale=0.90]

 \draw (-.7,-1) -- (4,6); 
 \draw (-1.05,-1) -- (6,6);
 \draw (-.4,-1) -- (2,5.5);

 \draw (.65,6) -- (3,-1);
 \draw (.5,6) -- (4,-1);
 \draw (.34,6) -- (4.7,-.5);

 \draw (7,1) -- (-1,1);
 \draw (7,1.2) -- (-1,.5);
 \draw (7,.85) -- (-1,1.5);

\node at (0,0) {$\bullet$};  
\node at (1,5) {$\bullet$};  
\node at (5,1) {$\bullet$};  

\node at (.4,0) {$P_1$};
\node at (.5,5) {$P_2$};
\node at (5,1.5) {$P_3$};

\draw (1.2,6) -- (-.3,-1.5);

\node at (0.1,-1.4) {$m$};
\node at (-1,-1.3) {$f$};
\node at (.5,6.3) {$g$};
\node at (7.4,1.1) {$h$};

\end{tikzpicture}
\end{center}
\caption{Three pencils with one common line}
\label{three pencils, one line}
\end{figure}

Thus, $f_\cA = fgh m$. Note that $\cA$ is the disjoint union of the arrangements $\cA (f g m)$ and $\cA (h)$. The former has been considered in \Cref{exa:two connected pencils}. Thus, setting $d = a + b + c +1$, \Cref{prop:liaison add residuals} and \Cref{prop:Betti top from residual} give minimal free resolutions of the form: 
\[
0 \to S^4(-d) \to S^2(1-d) 
    \oplus \\
\begin{array}{c}
    S(-d+a) \\
    \oplus \\
    S(-d+b) \\
    \oplus \\
    S(-d-1 +c)
\end{array}  \to  r_\cA \to 0
\]
and  
\[
0 \rightarrow S(-d) 
    \oplus \\
\begin{array}{c}
    S(1-d-a) \\
    \oplus \\
    S(1-d-b) \\
    \oplus \\
    S(2-d-c)
\end{array}
\rightarrow S^5 (1-d) \rightarrow \Jac (f_\cA)^{top} \rightarrow 0.
\]

(iii) Assume the three pencils are connected by two distinct hyperplanes defined by linear forms $m$ and $v$ connecting $P_1$ and $P_2$ and $P_2$ and $P_3$, respectively, and so $f_\cA  =  m v f g h $. See Figure~\ref{3 pencils, noncollin centers} if $n = 2$. 

\begin{figure}[h]
\begin{center}
\begin{tikzpicture}[scale=0.90]

 \draw (-.7,-1) -- (4,6); 
 \draw (-1.05,-1) -- (6,6);
 \draw (-.4,-1) -- (2,5.5);

 \draw (.65,6) -- (3,-1);
 \draw (.5,6) -- (4,-1);
 \draw (.34,6) -- (4.7,-.5);

 \draw (7,1) -- (-1,1);
 \draw (7,1.2) -- (-1,.5);
 \draw (7,.85) -- (-1,1.5);

\node at (0,0) {$\bullet$};  
\node at (1,5) {$\bullet$};  
\node at (5,1) {$\bullet$};  

\node at (.4,0) {$P_1$};
\node at (.5,5) {$P_2$};
\node at (5,1.5) {$P_3$};

\draw (1.2,6) -- (-.3,-1.5);

\draw (-.24,6.3) -- (5.8,.2);

\node at (.13,-1.45) {$m$};
\node at (-.5,6.2) {$v$};
\node at (-1,-1.3) {$f$};
\node at (.5,6.3) {$g$};
\node at (7.3,1.1) {$h$};

\end{tikzpicture}
\end{center}
\caption{Three pencils, two common lines, $m \neq v$}
\label{3 pencils, noncollin centers}
\end{figure}
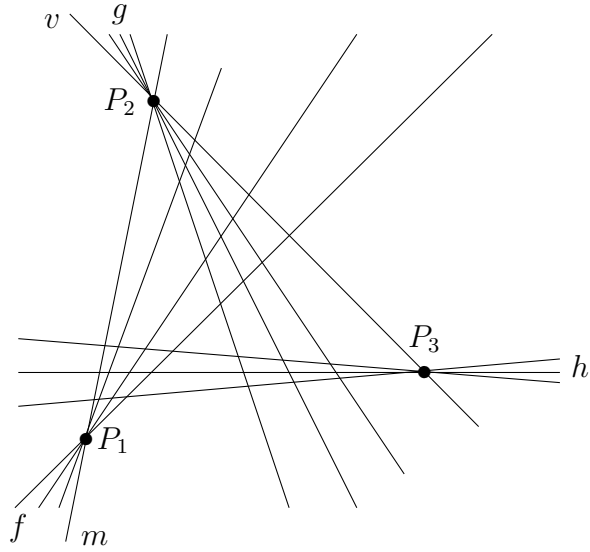

In this case, \Cref{cor:add pencil} shows
\[
r_\cA =h \cdot r_{\cA(m v  f g)} + (\ell_{P_3} m f g). 
\]
The minimal generators of $r_{\cA(m  v f g}$ have degrees $a+1, b+  2$ and $a+b+1$, which are smaller than $a+ b+2 = \deg (\ell_{P_3} m f g)$. Hence $\ell_{P_3} m f g$ is not a minimal generator of $r_{\cA(m f g)}$ and \Cref{cor:MFR res after adding generic hyperplane} gives a minimal free resolution of the form 
\[
0 \to S^3(-d) \to S(-d+1) \oplus 
\begin{array}{c}
    S(-d+a) \\
    \oplus \\
    S(-d+b+1) \\
    \oplus \\
    S(-d+c)
\end{array}  \to  r_\cA \to 0.
\]
Hence, \Cref{prop:Betti top from residual} provides a minimal free resolution
\[
0 \rightarrow 
\begin{array}{c}
    S(1-d-a) \\
    \oplus \\
    S(-d-b) \\
    \oplus \\
    S(1-d-c)
\end{array}
\rightarrow S^4(1-d) \rightarrow \Jac (f_\cA)^{top} \rightarrow 0.
\]

(iv) Assume the centers of the three pencils are contained in a hyperplane  defined by linear forms $m$ that connects the pencils, and so $f_\cA = f g h m$. See Figure~\ref{three lines, collin centers} for an illustration of $n=2$. 

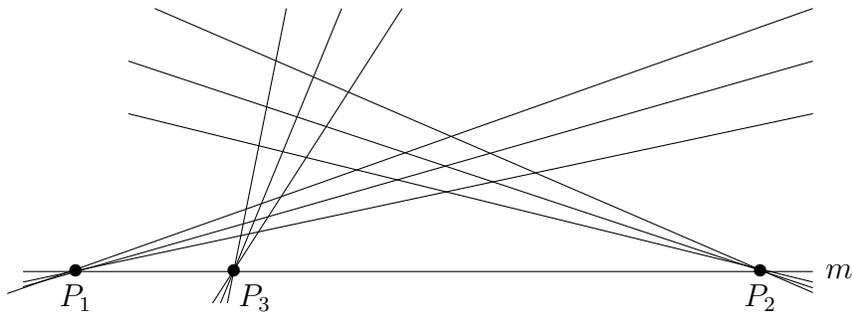
\begin{figure}[h]
\begin{center}
\begin{tikzpicture}[scale=0.70]

\draw (-5,-.29) -- (10,4);
\draw (-5,-.2) -- (10,3);
\draw (-5.3,-.42) -- (10,5);

\draw (10,-.31) -- (-3,4);
\draw (10,-.4) -- (-2.5,5);
\draw (10,-.2) -- (-3,3);

\draw (-1.25,-.6) -- (1.05,5);
\draw (-1.4,-.6) -- (2.2,5);
\draw (-1.12,-.6) -- (0,5);

\draw (-5,0) -- (10,0);

\node at (-4,-.5) {$P_1$};
\node at (-4,0) {$\bullet$};

\node at (9,-.5) {$P_2$};
\node at (9,0) {$\bullet$};

\node at (-1,0) {$\bullet$};
\node at (-.6,-.5) {$P_3$};

\node at (10.5,0) {$m$};

\end{tikzpicture}
\end{center}
\caption{Three pencils with collinear centers and the common line}
\label{three lines, collin centers}
\end{figure}

Arguments very similar to the ones for Case (iii) give the following minimal free resolutions: 
\[
0 \to S^3(-d) \to S(-d+1) \oplus 
\begin{array}{c}
    S(-d+a) \\
    \oplus \\
    S(-d+b) \\
    \oplus \\
    S(-d+c)
\end{array}  \to  r_\cA \to 0
\]
and 
\[
0 \rightarrow 
\begin{array}{c}
    S(1-d-a) \\
    \oplus \\
    S(1-d-b) \\
    \oplus \\
    S(1-d-c)
\end{array}
\rightarrow S^4(1-d) \rightarrow \Jac (f_\cA)^{top} \rightarrow 0.
\]

(v) Finally, assume that the three centers are connected by three distinct hyperplanes defined by linear forms $m, v, w$, and so $f_\cA = f g h m v w$. See Figure~\ref{all three lines} for an illustration of $n=2$. 

\begin{figure}[h]
\begin{center}
\begin{tikzpicture}[scale=0.90]

 \draw (-.7,-1) -- (4,6); 
 \draw (-1.05,-1) -- (6,6);
 \draw (-.4,-1) -- (2,5.5);

 \draw (.65,6) -- (3,-1);
 \draw (.5,6) -- (4,-1);
 \draw (.34,6) -- (4.7,-.5);

 \draw (7,1) -- (-1,1);
 \draw (7,1.2) -- (-1,.5);
 \draw (7,.85) -- (-1,1.5);

\node at (0,0) {$\bullet$};  
\node at (1,5) {$\bullet$};  
\node at (5,1) {$\bullet$};  

\node at (.4,0) {$P_1$};
\node at (.5,5) {$P_2$};
\node at (5,1.5) {$P_3$};

\draw (1.2,6) -- (-.3,-1.5);

\draw (-.24,6.3) -- (5.8,.2);

\draw (-1.3,-.25) -- (8,1.6);

\node at (.13,-1.45) {$m$};
\node at (-.5,6.3) {$v$};
\node at (-1,-1.3) {$f$};
\node at (.5,6.3) {$g$};
\node at (7.3,1.1) {$h$};
\node at (-1.6,-.2) {$w$};

\end{tikzpicture}
\end{center}
\caption{Three pencils, three common lines}
\label{all three lines}
\end{figure}
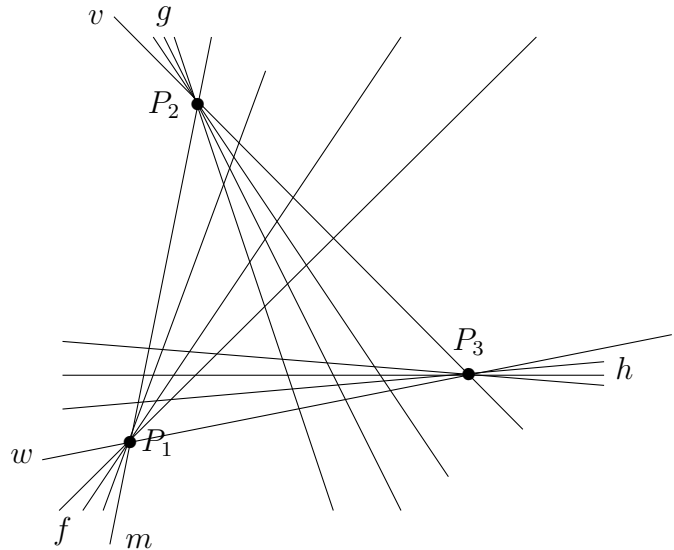

If one of the forms $f, g, h$ equals the identity, then $\cA$ consists of two connected pencils $\cA (f w)$ and $\cA (g v)$, which we considered in \Cref{exa:two connected pencils}. Thus, we may assume that $a, b, c$ are positive.  
We fix a general linear form $\ell$ that we use to obtain the general residuals of $\cA$ and any of its subarrangements. Applying \Cref{cor:add pencil}, we obtain 
\[
r_\cA = h \cdot r_{\cA(fgmvw)} + (\ell_{P_3} m f g). 
\]
Note that $\cA(fgmvw)$ is an arrangement of two connected pencils $\cA (f w)$ and $\cA (g v)$. Hence, \Cref{exa:two connected pencils}  gives  the minimal generating set
\[
 r_{\cA(fgmvw)} =  (\ell_{P_1} v g, \ell_{P_2} w f, v w f g). 
\]
We want to show that $\ell_{P_3} m f g$ is not a minimal generator of $r_{\cA(fgmvw)}$. Assume the contrary. Note that the minimal generators of $r_{\cA(fgmvw)}$ have degrees $a+2, b+2$ and $a + b + 2$. Since $\ell_{P_3} m f g$ also has degree $a+b +2$ it follows that the ideal 
\[
\fa = (\ell_{P_1} v g, \ell_{P_2} w f, \ell_{P_3} m f g) 
\]
must be equal to $ r_{\cA(fgmvw)}$. However, the two lowest degree generators of $\fa$ form a regular sequence, and so we get
\begin{align*}
(\ell_{P_1} v g, \ell_{P_2} w f) & = (\ell_{P_1}, wf) \cap (\ell_{P_2}, vg) \cap (vg, wf) \cap (\ell_{P_1}, \ell_{P_2}) \\
& = r_{\cA(fgmvw)} \cap (\ell_{P_1}, \ell_{P_2}),  
\end{align*}
where the second equality follows from the primary decomposition of general residuals given in \Cref{cor:decomposition of residual}. 
Since $\ell_{P_3} m f g$ is contained in $ r_{\cA(fgmvw)}$, this implies 
\[
\fa = r_{\cA(fgmvw)} \cap (\ell_{P_1}, \ell_{P_2}, \ell_{P_3} m f g). 
\]
By the generality of the point $\ell^{\vee}$, its ideal  does not contain any of the associated prime ideals of $r_{\cA(fgmvw)}$. As 
$(\ell_{P_1}, \ell_{P_2}, \ell_{P_3} m f g)$ is contained in $I_{\ell^\vee}$ the same is true for the ideal $(\ell_{P_1}, \ell_{P_2}, \ell_{P_3} m f g)$, which proves $\fa \neq r_{\cA(fgmvw)}$. This contradiction shows that $\ell_{P_3} m f g$ is not a minimal generator of 
$r_{\cA(fgmvw)}$. Applying \Cref{cor:MFR res after adding generic hyperplane}, we obtain for the minimal free resolution of $r_\cA$: 
\[
0 \to 
\begin{array}{c}
S^2(-d) \\
\oplus \\
S(-d+1) 
\end{array}
\to S(-d+1) \oplus 
\begin{array}{c}
    S(-d+a+1) \\
    \oplus \\
    S(-d+b+1) \\
    \oplus \\
    S(-d+c+1)
\end{array}  \to  r_\cA \to 0. 
\]
Thus, \Cref{prop:Betti top from residual} provides a minimal free resolution
\[
0 \rightarrow 
\begin{array}{c}
    S(-d-a) \\
    \oplus \\
    S(-d-b) \\
    \oplus \\
    S(-d-c)
\end{array}
\rightarrow 
\begin{array}{c}
S(-d) \\
\oplus \\
S^3(-d+1) 
\end{array}
\rightarrow \Jac (f_\cA)^{top} \rightarrow 0.
\]

\end{example}


Some of our results in this Section can be summarized as follows. 

\begin{proposition}
   \label{prop:build recursively} 
Let $\cA \subset \PP^n$ be an arrangement consisting of $d$ hyperplanes. Assume that $\cA$ can be built up recursively by adding almost generic hyperplanes, that is, starting with an arrangement  $\cA_2$ consisting of two hyperplanes of $\cA$, produce an arrangement $\cA_3$ by adding to $\cA_2$ a hyperplane $H_2$ of $\cA$ that is almost generic with respect to $\cA_2$, and continue such that 
$\cA_d = \cA$. Then $S/r_\cA$ and $S/\Jac( f_\cA)^{top}$ are Cohen-Macaulay and their 
Hilbert functions 
are uniquely determined and can be computed explictly from the knowledge of the numbers $|\cA_i \cap H_i|$ for $i = 2,3,\ldots,d-1$, where $\cA_i \cap H_i$ denotes the set of codimension two subspaces $H \cap H_i$ with $H \in \cA_i$. 
\end{proposition}

\begin{proof}
The general residual of $\cA_2$ is a complete intersection, and so we know its Hilbert function. By \Cref{prop:add almost generic hyperplane}, the general residual of a hyperplane arrangement that is obtaining by adding an  almost generic hyperplane to a given arrangement is a specific basic double link of the general residual of the original arrangement. The gives the claimed Cohen-Macaulayness results. For considering Hilbert functions, fixing $i$ and avoiding some indices, one has 
\[
r_{\cA_{i+1}} = x \cdot r_{\cA_i} + (p), 
\]
where $x$ defines $H_i$ and $\deg p = |\cA_i \cap H_i|$. Using the short exact sequence
\[
0 \to S (-1 - \deg p)  \to S (- \deg p) \oplus  r_{\cA_i} (-1)   \to  r_{\cA_{i+1}} \to 0,  
\]
it follows that the Hilbert function of $r_{\cA_{i+1}}$ can be computed from knowing the Hilbert function of $r_{\cA_{i}}$ and $\deg p$. Using this repeatedly, the assertion follows for $r_\cA$. We conclude for the Hilbert function of $\Jac (f_\cA)^{top}$ using a standard fact from liaison theory, see, e.g, \cite[Theorem(iii)]{N-liais-hilb}. 
\end{proof}

\begin{remark} \label{limitations}
\Cref{thm:MFR residual with special assumption} gives an explicit condition on a hyperplane arrangement, which guarantees that it can be built by adding almost generic hyperplanes, and so \Cref{prop:build recursively} applies to it. This follows from the proof of \Cref{thm:MFR residual with special assumption}. However, \Cref{prop:build recursively} is applicable to many arrangements which do not satisfy the assumptions of \Cref{thm:MFR residual with special assumption}. For instance,  this is true for  all the arrangements discussed in  \Cref{exa:three pencils}(ii)-(v). 

Note that a more careful analysis often allows one to determine for an arrangement that can be obtained by adding almost generic hyperplanes not only the Hilbert function but also the graded Betti numbers as illusttrated in  \Cref{exa:three pencils}. 

There are arrangements to which \Cref{prop:build recursively} cannot be applied. One of the first such examples are Fermat arrangements.   In $\PP^2$ these are defined by the polynomial $(x^k-y^k)(y^k-z^k)(z^k-x^k)$. Such an arrangement has multiplicity $k$ at each of the three coordinate points, but also it has triple points at all points of intersection of $(x^k-y^k)$ and $(y^k-z^k)$ since $z^k-x^k$ is in the pencil defined by $(x^k-y^k)$ and $(y^k-z^k)$.  
 Note however, that we can use the tools in this section to analyze many more arrangements than covered by \Cref{prop:build recursively}. For example, using arrangements for which we know the graded Betti numbers of the general residuals as building blocks and using \Cref{prop:liaison add residuals}, one can determine the graded Betti numbers of the general residual of the resulting arrangement. 

For instance, it is well-known that any Fermat arrangement $\cA \subset \PP^2$ is free, and so we know the graded Bettti numbers of $r_\cA$. Add to $\cA$ three generic lines. These three lines meet in three points. Now add almost generic lines passing to one of the three points. 
Denote  by $\cB$ the arrangement formed by the lines added to $\cA$. In \Cref{exa:three pencils}, we determined the graded Betti numbers of $r_\cB$. Hence, \Cref{prop:liaison add residuals} gives the graded Betti numbers of $r_{\cA \cup \cB}$, and so of $\Jac (f_{\cA \cup \cB})$ by \Cref{prop:Betti top from residual}. 
\end{remark}


\section{Resolutions of certain Milnor modules and of Jacobian ideals} 
\label{sec:res Jacobians}

The main result of this section shows that one can read off the graded Betti numbers of the Milnor module and the Jacobian ideal of any line arrangement from the graded Betti numbers of the saturation of the Jacobian ideal. As an application we show that these graded Betti numbers are combinatorially determined for any union of disconnected pencils. 

\begin{theorem}
   \label{thm:Betti from sat}
Let $\cA \subset \PP^2$ be any line arrangement. Write the graded minimal free resolution of $\Jac (f_\cA)^{sat}$ as 
\[
0 \to G \stackrel{\gamma}{\longrightarrow} S^b (-d+1)  \oplus F \to \Jac (f_\cA)^{sat}  \to 0, 
\]
where $d = |\cA|$. If $\cA$ is not free then $b \ge 3$ and the Milnor module  $M$ and the Jacobian ideal have graded minimal free resolutions of the form 
\[
0 \to S^{b-3}(-2d+2)  \oplus F^* (-3d+3)  \to G^* (-3d+3) \to G \to S^{b-3}(-d+1)  \oplus F \to M \to 0 
\]
and 
\[
0 \to S^{b-3}(-2d+2)  \oplus F^* (-3d+3) \to G^* (-3d+3)  \to S^{3}(-d+1) \to \Jac(f_\cA) \to 0. 
\]
\end{theorem}

\begin{proof}
If $\Jac(\cA)$ has two minimal generators then it is a complete intersection, and so $\cA$ is a free arrangement; a contradiction to our assumption. Hence, $\Jac(\cA)$ has exactly three minimal generators of degree $d-1$. Thus, \Cref{thm:initial degree} gives $b \ge 3$. 

Define a graded $S$-module $E$ as first syzygy module of $\Jac(\cA)$, and so we have an exact sequence 
\begin{equation}
   \label{eq:def E} 
0 \to E \to S^{3}(-d+1) \to \Jac(\cA) \to 0. 
\end{equation}
Note that $E$ is a reflexive $S$-module of rank two and that the vector bundle $\widetilde{E}$ has first Chern class $-3d+3$. Thus, \cite[Proposition 1.10]{H} implies 
\begin{equation}
   \label{eq:self-dual} 
E^* \cong E (3 d -3). 
\end{equation}

Consider the exact sequence 
\[
0 \to \Jac(f_\cA) \to \Jac (f_\cA)^{sat}  \to M \to 0. 
\]
Taking into account the cancellation due to the fact that the minimal generators of $\Jac(\cA)$ are part of a minimal generating set of 
$\Jac(\cA)^{sat}$, the mapping cone provides the exact sequence
\begin{equation}
    \label{eq:E-type of M}
0 \to E \to  G \stackrel{\delta}{\longrightarrow} S^{b-3}(-d+1)  \oplus F \to M \to 0.  
\end{equation} 
Notice that $\delta$ is a minimal map because $\gamma$ has this property by the assumed minimality of the given resolution of $\Jac(\cA)^{sat}$. 

Consider a graded minimal free resolution of $E$ over $S$, 
\[
0 \to P_2 \to P_1 \to E \to 0. 
\]
Now we dualize the sequence \eqref{eq:E-type of M}. Since $M$ has finite length we obtain the exact sequence 
\[
0 \to F^* \oplus S^{b-3}(d-1) \stackrel{\delta^*}{\longrightarrow} G^* \to E^* \to 0, 
\]
where $\delta^*$ is a minimal map because its dual $\delta$ has this property. Hence, this sequence is a graded minimal free resolution of $E^*$. Using the uniqueness of such a resolution up to isomorphism of complexes and  $E^* \cong E (3 d -3)$, it follows that there are graded isomorphisms
\begin{equation}
  \label{eq:iso}
G^* \cong P_1 (3d-3) \quad \text{ and } \quad F^* \oplus S^{b-3}(d-1) \cong P_2 (3d-3). 
\end{equation}
Splicing the resolution of $E$ and the Sequence \eqref{eq:E-type of M}, we get an exact sequence
\begin{equation}
   \label{eq:free res M} 
0 \to P_2 \to P_1 \stackrel{\rho}{\longrightarrow}  G \stackrel{\delta}{\longrightarrow} S^{b-3}(-d+1)  \oplus F \to M \to 0, 
\end{equation}
which is a graded free resolution of $M$. We claim it is minimal. Otherwise, the map $\rho$ is not minimal, and one can split off a free direct summand, say $S(-t)$. That means, one can write $\rho$ as a direct sum of  a map $\rho' \colon P_1' \to G'$ and an isomorphism 
$\iota \colon S(-t) \to S(-t)$, where $P_1 = P_1' \oplus S(-t)$ and $G = G' \oplus S(-t)$. It follows that 
\[
E = \ker \delta = \im \rho \cong S(-t) \oplus \im \rho'. 
\]
Hence, $\im \rho'$ is a reflexive module of rank one, which implies that it is free, and so $E$ is a free $S$-module as well. This is a contradiction to our assumption that $\cA$ is not a free arrangement. It shows that the Sequence \eqref{eq:free res M} is a graded minimal free resolution of $M$. Taking into account the Isomorphisms \eqref{eq:iso}, our first assertion follows. 

We get the second assertion by splicing the free resolution of $E$ and Sequence \eqref{eq:def E}. 
\end{proof}

\begin{remark}
Note that the isomorphism $E^* \cong E (3 d -3)$ implies an isomorphism $M^{\vee} \cong M (3d-3)$, whose existence is true in greater generality by results of Sernesi \cite{sernesi}, van Straten and Warmt \cite{VsW} and Eisenbud and Ulrich \cite{EU}. 
\end{remark}

We illustrate the above result by a simple example, namely generic line arrangements. 

\begin{example}
    \label{exa:generic line arr}
Let   $\cA \subset \PP^2$ be an arrangements of $d$ lines such that no three lines meet in a point. Then its Jacobian ideal has a graded minimal free resolution of the form 
\[
0 \to S^{d-3} (-2d+2) \to S^{d-1} (-2d+3) \to S^3 (-d+1) \to \Jac(f_\cA) \to 0. 
\]
Indeed, denote by $\ell_1,\ldots,\ell_d$ the linear forms defining the lines so that $f_\cA = \ell_1 \cdots \ell_d$. 
\cite[Corollary 4.2]{MN} gives for the ideal of the saturation  
\[
\Jac(f_\cA)^{sat} =  \left ( \frac{f_\cA}{\ell_1}, \dots, \frac{f_\cA}{\ell_d} \right ).
\]
Its minimal free resolution is of the form 
\[
0 \to S^{d-1} (-d) \to S^d (-d+1) \to \Jac(f_\cA)^{sat} \to 0. 
\]
Thus, the claim follows by \Cref{thm:Betti from sat}. 
\end{example}

Note that for generic arrangements even the maps in the above minimal free resolution have been described independently by Rose and Terao \cite{LT} and Yuzvinski \cite{Y}. 

It is worth noting the following consequence of \Cref{thm:Betti from sat}, where we exclude the easy case of concurrent lines (in which $\Jac (f_\cA)$ is a complete intersection). 

\begin{corollary}
    \label{cor:max reg} 
For any line arrangement $\cA \subset \PP^2$  other than a set of concurrent lines, one has for the Castelnuovo-Mumford regularity 
\[
\reg \Jac (f_\cA) \le 2 |\cA| - 4, 
\]
where equality is true if and only if  one of the following two conditions is satisfied:  
\begin{itemize}

\item[(i)] $\cA$ is free, does not consist of concurrent lines and $\Jac (f_\cA)$ has a linear syzygy. 

\item[(ii)] $\cA$ is not free  and $\Jac (f_\cA)^{sat}$ has at least four independent generators of degree $|\cA| -1$. 
\end{itemize}
\end{corollary} 

\begin{proof}
If $\cA$ is free this is well-known and follows from properties of minimal free resolutions. If $\cA$ is not free the assertion is implied by   \Cref{thm:Betti from sat} by noting, using its notation,  that we may assume $a(F) \ge d$ and so  $a(G) \ge d+1$ if $b = 3$. 
\end{proof}

\begin{remark}
The inequality $\reg \Jac (f_\cA) \le 2 |\cA| - 4$ was originally shown by Schenck using different methods \cite{Schenck}. The characterization of equality is new. 
\end{remark}

Using \Cref{cor:MFR with special assumption} it follows that the above graded Betti numbers are combinatorially determined for any union of disconnected pencils. 
Recall that $S(\cA)$ and $t_P$ were introduced above \Cref{prop:jac = union}.

\begin{theorem}
        \label{thm:Betti disconnected pencils} 
Let $\cA \subset \PP^2$ be an  arrangement of $d$ lines. Denote by 
        \[
    S_0 (\cA) = \{ P \in S(\cA) \; \mid \; t_P \ge 3\}
    \]
the set of points $P \in \PP^2$ that are contained in at least three distinct lines of $\cA$. Assume that no line in $\cA$ contains two or more elements of $S_0(\cA)$. Denote by $s$  the number of lines of $\cA$ that do not contain any $P \in S_0(A)$. Then one has: 
\begin{itemize}

\item[{\rm (a)}] If $\cA$ consists of $d$ concurrent lines then it is free and the graded minimal free resolution of its Jacobian ideal has the form 
\[
0 \to S(-2d+2)  \to S^2 (-d+1) \to \Jac (f_\cA) \to 0. 
\]

\item[{\rm (b)}] If $\cA$ consists of $d-1$ lines through a point $P \in \PP^2$ and one line that does not contain $P$ then it is free and the graded minimal free resolution of $Jac (f_\cA)$ has the form 
\[
0 \to S(-2d+3) \oplus S(-d) \to S^3 (-d+1) \to \Jac (f_\cA) \to 0. 
\]

\item[{\rm (c)}] If $\cA$ does not satisfy the condition in either (a) or (b) then $\cA$ is not  free and the Milnor module  $M$ and the Jacobian ideal have graded minimal free resolutions of the form 
\begin{align*} 
\hspace{13cm}&\hspace{-13cm}
0 \to S^{2 |S_0 (\cA)| + s-3}(-2d+2)   \to 
\begin{array}{c}
S^{ |S_0 (\cA)| + s-1} (-2d + 3) \\
\ \ \ \ \ \oplus \\[3pt]
{\displaystyle \bigoplus_{P \in S_0 (\cA)} S(-2 d+t_P -1) }
\end{array} 
\to  \\
\begin{array}{c}
S^{ |S_0 (\cA)| + s-1} (-d) \\
\ \ \ \ \ \oplus \\
{\displaystyle \bigoplus_{P \in S_0 (\cA)} S(-d-t_P +2) }
\end{array} 
\to S^{2 |S_0 (\cA)| + s-3}(-d+1)   \to M \to 0 
\end{align*}
and 
\begin{align*} 
\hspace{13cm}&\hspace{-13cm}
0 \to S^{2 |S_0 (\cA)| + s-3}(-2d+2)   \to 
\begin{array}{c}
S^{ |S_0 (\cA)| + s-1} (-2d + 3) \\
\ \ \ \ \ \oplus \\[3pt]
{\displaystyle \bigoplus_{P \in S_0 (\cA)} S(-2 d+t_P + 1) }
\end{array}   
\to \\
S^{3}(-d+1) \to \Jac(f_\cA) \to 0. 
\end{align*}
\end{itemize}

\end{theorem}

\begin{proof} 
Note that in Case (a), one has $|S_0 (\cA)| = 1$ and $s=0$. Hence,  \Cref{cor:MFR with special assumption} shows that $\Jac(f_\cA)^{sat}$ is minimally generated by two forms of degree $d-1$, and so it is equal to $\Jac(f_\cA)$. 

In Case (b), one has $|S_0 (\cA)| = 1$ and $s=1$. Thus,  \Cref{cor:MFR with special assumption} shows that $\Jac(f_\cA)^{sat}$ is minimally generated by three forms of degree $d-1$. So, again it equals $\Jac(f_\cA)$ and \Cref{cor:MFR with special assumption} gives its minimal free resolution. 

Note that any arrangement of at most three lines satisfies the condition of case (a) or (b). Assuming that neither of these conditions is satisfied by $\cA$ we get that $2 |S_0 (\cA)|  + s \ge 4$, and so  \Cref{cor:MFR with special assumption} shows that $\Jac(f_\cA)$ is not a  saturated ideal, that is, $\cA$ is not free. Furthermore, \Cref{cor:MFR with special assumption} gives all graded Betti numbers of $\Jac (f_\cA)^{sat}$. 
We now conclude using \Cref{thm:Betti from sat}. 
\end{proof}

We conclude this section by providing resolutions in cases that are mostly not covered by the above result. 

\begin{example}
    \label{exa:res for three pencils} 
Using the notation of \Cref{exa:three pencils}, we consider an arrangement $\cA \subset \PP^2$ consisting of three pencils of lines. 
In each case, we determine the graded Betti numbers of $\Jac( f_\cA)$ and the Milnor module $M_\cA = \Jac( f_\cA)^{sat}/\Jac( f_\cA)$ using the results of  \Cref{exa:three pencils} as well as \Cref{thm:Betti from sat}. 

(i) Assume no two centers are connected. Then \Cref{thm:Betti disconnected pencils}(c) gives 
\[
0 \to S^3 (-2d+2) \to 
\begin{array}{c}
S^2 (-2d+3) \\
\ \ \ \ \ \oplus \\
S(-2d+a +1) \\
\oplus \\
S(-2d+b +1) \\
\oplus \\
S(-2d+c +1)
\end{array} 
\to 
\begin{array}{c}
S^2 (-d) \\
\ \ \ \ \ \oplus \\
S(-d-a +2) \\
\oplus \\
S(-d-b +2) \\
\oplus \\
S(-d-c +2)
\end{array} 
\to S^3(-d+1) \to M_\cA \to 0
\]
and 
\[
0 \to S^3 (-2d+2) \to 
\begin{array}{c}
S^2 (-2d+3) \\
\ \ \ \ \ \oplus \\
S(-2d+a +1) \\
\oplus \\
S(-2d+b +1) \\
\oplus \\
S(-2d+c +1)
\end{array} 
\to 
 S^3(-d+1) \to \Jac (f_\cA) \to 0.
\]

(ii) Assume that exactly two centers are connected. Combining \Cref{exa:three pencils}(ii) and \Cref{thm:Betti from sat}, we get
\[
0 \to S^2 (-2d+2) \to 
\begin{array}{c}
S (-2d+3) \\
\ \ \ \ \ \oplus \\
S(-2d+a +2) \\
\oplus \\
S(-2d+b +2) \\
\oplus \\
S(-2d+c +1)
\end{array} 
\to 
\begin{array}{c}
S (-d) \\
\ \ \ \ \ \oplus \\
S(-d-a +1) \\
\oplus \\
S(-d-b +1) \\
\oplus \\
S(-d-c +2)
\end{array} 
\to S^2(-d+1) \to M_\cA \to 0 
\]
and 
\[
0 \to S^2 (-2d+2) \to 
\begin{array}{c}
S (-2d+3) \\
\ \ \ \ \ \oplus \\
S(-2d+a +2) \\
\oplus \\
S(-2d+b +2) \\
\oplus \\
S(-2d+c +1)
\end{array} 
\to 
 S^3(-d+1) \to \Jac (f_\cA) \to 0.
\]

(iii) Assume the three centers are connected by two distinct lines. In this case,  \Cref{exa:three pencils}(iii) and \Cref{thm:Betti from sat} yield
\[
0 \to S (-2d+2) \to 
\begin{array}{c}
S(-2d+a +2) \\
\oplus \\
S(-2d+b +3) \\
\oplus \\
S(-2d+c +2)
\end{array} 
\to 
\begin{array}{c}
S(-d-a +1) \\
\oplus \\
S(-d-b ) \\
\oplus \\
S(-d-c +1)
\end{array} 
\to S(-d+1) \to M_\cA \to 0 
\]
and 
\[
0 \to S (-2d+2) \to 
\begin{array}{c}
S(-2d+a +2) \\
\oplus \\
S(-2d+b +3) \\
\oplus \\
S(-2d+c +2)
\end{array} 
\to 
 S^3(-d+1) \to \Jac (f_\cA) \to 0.
\]

(iv) Assume the three centers are connected by one line. Then \Cref{exa:three pencils}(iv) and \Cref{thm:Betti from sat} give
\[
0 \to S (-2d+2) \to 
\begin{array}{c}
S(-2d+a +2) \\
\oplus \\
S(-2d+b +2) \\
\oplus \\
S(-2d+c +2)
\end{array} 
\to 
\begin{array}{c}
S(-d-a +1) \\
\oplus \\
S(-d-b+1 ) \\
\oplus \\
S(-d-c +1)
\end{array} 
\to S(-d+1) \to M_\cA \to 0 
\]
and 
\[
0 \to S (-2d+2) \to 
\begin{array}{c}
S(-2d+a +2) \\
\oplus \\
S(-2d+b +2) \\
\oplus \\
S(-2d+c +2)
\end{array} 
\to 
 S^3(-d+1) \to \Jac (f_\cA) \to 0.
\]

(v) Finally, assume the three centers are connected by three distinct lines. Then \Cref{exa:three pencils}(v) and \Cref{thm:Betti from sat} provide
\[
0 \to S (-2d+3) \to 
\begin{array}{c}
S(-2d+a +3) \\
\oplus \\
S(-2d+b +3) \\
\oplus \\
S(-2d+c +3)
\end{array} 
\to 
\begin{array}{c}
S(-d-a) \\
\oplus \\
S(-d-b ) \\
\oplus \\
S(-d-c)
\end{array} 
\to S(-d) \to M_\cA \to 0 
\]
and 
\[
0 \to S (-2d+3) \to 
\begin{array}{c}
S(-2d+a +3) \\
\oplus \\
S(-2d+b +3) \\
\oplus \\
S(-2d+c +3)
\end{array} 
\to 
 S^3(-d+1) \to \Jac (f_\cA) \to 0.
\]

\end{example}


\section{ A freeness criterion} 
\label{sec: freeness criterion}
In this section, we consider only line arrangements. By combining earlier results, we establish a freeness criterion for line arrangements.  In particular, it shows that a line arrangement is free if and only adding a certain point to its general residual gives a zero-dimensional subscheme that is the intersection of two hypersurfaces.


We recall the following result of Kreuzer. 

\begin{theorem}[{\cite[Theorem 1.1]{K}}] 
    \label{thm:kreuzer} 
A $0$-dimensional subscheme $X \subset \PP^n$ is arithmetically Gorenstein if and only if there is some  integer $a_X \ge 0$ such that $h_X (j) = \deg X - h_X (a_X - j)$ for any integer $j$ and, for any subscheme $Y \subset X$ with degree $-1 + \deg X$, one has $h_Y (a_X) = \deg X  - 1$. 
\end{theorem}

The conditions on the subschemes $Y$ is the statement is exactly saying that $X$ satisfies the Cayley-Bacharach property.

By \Cref{cor:char ci}, we know that the Jacobian of a line arrangement is a complete intersection if and only if the lines are concurrent. Thus, it is harmless to rule out this case in the following definition and theorem. Recall that an arrangement  $\cA$ is said to be free with \emph{exponents} $e_1$ and $e_2$ if the syzygy module of $\Jac (f_\cA)$ is a free $S$-module with two generators of degree $d - 1 + e_1$ and $d- 1 + e_2$, respectively.

\begin{theorem}
     \label{thm:char freeness lines}
Let $\cA$ be a set of $d$ distinct lines of $\PP^2$. 
Assume the lines in $\cA$ are not concurrent. 
Let $\ell$ be a general linear form and let $r_\cA$ be the general residual with respect to $\ell$.
For each point $P \in S(\cA)$,  consider the ideal 
\[
\fa_P = (\ell_P, I_P^{t_P - 2})  \cap \bigcap_{P' \in S(\cA) - P} (\ell_{P'}, I_{P'}^{t_{P'} - 1} ). 
\]
Define the ideal 
\[
I = r_\cA \cap I_{\ell^\vee}.
\] 
The following conditions are equivalent:
\begin{itemize}

\item[(a)] The line arrangement determined by $\cA$ is free with exponents $a-1$ and $b-1$. 

\item[(b)] The ideal $I$ 
is generated by two forms of degree $a$ and $b$ with $a + b = d+1$ and $a, b \ge 2$. 

\item[(c)] The $h$-vector of $S/I$ is symmetric, i.e., for any integer $j$, the Hilbert function of $S/I$ satisfies 
\[
h_{S/I} (d-2-j)  = \deg I -  h_{S/I} (j),  
\]
and, for any point $P \in S(\cA)$, one has 
\[
h_{S/\fa_P \cap I_{\ell^\vee}}  (d-2) = \deg I - 1.  
\]
\end{itemize}

\end{theorem}

\begin{proof} 
Our assumption that the lines in $\cA$ are not concurrent gives that $d = |\cA| \ge 3$. 

We begin by showing  the equivalence of (a) and (b). First assume (a). Using \Cref{cor:char ci}, we know that $\Jac (f_\cA)$ is  minimally generated by the three partial derivatives of $f_\cA$. It is saturated. by the assumed freeness of $\cA$. Thus, it has a graded minimal free resolution of the form 
\[
0 \to S(-d+1-a) \oplus S(-d+1-b) \to S^3 (-d+1) \to \Jac (f_\cA) \to 0
\]
with positive integers $a, b$ satisfying $a + b = d-1$, and so $1 \le a, b, \le d-2$. Linking by $(f_\cA, \frac{\partial f_\cA}{\partial \ell})$ and taking into account 
 the cancellation due to the fact that $\frac{\partial f_\cA}{\partial \ell}$ is a minimal generator of both $\Jac (f_\cA)$ and the complete intersection, the mapping cone shows that the general residual has a graded minimal free resolution of the form 
 \[
0 \to S^2(-d) \to S (-d+b)  \oplus S(-d+a) \oplus  S(-d+1)  \to r_\cA \to 0.   
\]
Taking into account that $a+b=d-1$ we rewrite this as 
\[
0 \to S^2(-d) \to S (-a-1)  \oplus S(-b-1) \oplus  S(-d+1)  \to r_\cA \to 0. 
\]

By \Cref{{prop:add general point}}(a), we may assume that the minimal generator of $r_\cA$ with degree $d-1$ is $\frac{\partial f_\cA}{\partial \ell}$. Denote by $g$ and  $h$ the minimal generators of $r_\cA$ with degrees $a+1$ and $b+1$, respectively. 
Using the above resolution, a computation shows that
\[
\deg r_\cA = (a+1) (b+1) - 1. 
\]
The Sequence \eqref{eq:SES} (established in the proof of \Cref{prop:add general point}) induces the following diagram of exact sequences
\begin{align}
     \label{commutative diagram_multi factor}
\minCDarrowwidth20pt    
\begin{CD}
@. 0 @.  @. 0 @.   \\
 @. @VVV   @.    @VVV \\
 @. S(-d-1)  @.   @.   S^2(-d) @. \\
@. @VVV   @.    @VVV \\
@.  S^2(-d)  @.   @.   S(-a-1 \oplus S(-b-1) \oplus S(-d+1) @. \\
@. @VVV   @.    @VVV \\
@.  S(-d+1)  @.   @.   S @. \\
@.  @VVV   @.  @VVV    \\
0 @>>>  (S/I_{\ell^\vee})(-d+1)   @>>>  S/I  @>>>   S/r_\cA @>>>  0 \\
@. @VVV @.  @VVV    \\
@. 0 @. @. 0 @.  \\
\end{CD}
\end{align} 
Using the Horseshoe lemma, it implies that a graded minimal free resolution of $I$ has the form 
\[
0 \to S(-d+1) \oplus F \to S (-a-1)  \oplus S(-b-1) \oplus  F  \to I \to 0,  
\]
where $F = S^{\beta} (-d)$ (see also \Cref{cor:gens of residual}). 
In particular, we get $\deg I = 1 + \deg r_\cA = (a+1)(b+1)$. 
If the minimal generators $g, h$ of $I$ form a regular sequence, then comparing degrees we conclude that $I = (g, h)$,  as desired. Otherwise, $g$ and $h$ have a greatest common divisor, $p$, of positive degree. Write $g = p \cdot \tilde{g}$ and $h = p \cdot \tilde{h}$. Then the relation 
\[
\tilde{h} \cdot g - \tilde{g} \cdot h = 0
\]
shows that $I$ as well as $r_\cA$ have a syzygy of degree $a+1 + b+1 - \deg p = d+1 - \deg p$. Since $r_\cA$ has only minimal syzygies of degree $d$ and $a+b=d-1$, we deduce that $p$ must have degree one. Moreover, $g$ and $h$ have exactly one minimal syzygy, which gives $\beta = 1$. Thus, $I$ has exactly one minimal generator of degree $d$. Denote it by $q$. It follows that 
\[
I = (p \cdot \tilde{g}, p \cdot \tilde{h}, q) \subset (\tilde{g}, \tilde{h}) \cap (p, q). 
\]
The right-hand side is an unmixed ideal of degree $a b + d = (a+1) (b+1) = \deg I$, which yields the equality 
\[
I = (\tilde{g}, \tilde{h}) \cap (p, q). 
\]
It shows that the 0-dimensional scheme $Y$ defined by $I$ contains a subscheme of degree $d$ that is supported on the   line $L$ defined by $p$.  Recall that by \Cref{prop:decomposition of residual} we have
\[
r_\cA = \bigcap_{P \in S(\cA)} (\ell_P, I_P^{t_P - 1}).  
\]
By the generality of $\ell$, any line through some point $P \in S(\cA)$ and $\ell^\vee$ contains no other point of $S(\cA)$. Since the lines of $\cA$ are not concurrent, we know $t_P \le d-1$. Hence any such line supports a subscheme of $Y$ whose degree is at most $t_P - 1  + 1 <  d$. Therefore, the line $L$ must contain at least two points of $S(\cA)$.  Moreover, it does not contain 
$\ell^\vee$. Hence, $L$ supports a degree $d$ subscheme of the scheme defined by $r_\cA$. However, such a subscheme does not exist due to \Cref{prop:char residuals}. This contradiction shows that $g$ and $h$ must form a regular sequence and $I = (g,h)$, as desired. 

Assume now that Condition (b) is satisfied. We show that it implies Condition (a). Applying the mapping cone procedure to the exact sequence 
\[
0 \to  (S/I_{\ell^\vee})(-d+1)   \to   S/I  \to  S/r_\cA \to  0 
\]
we deduce that the graded minimal free resolution of $r_\cA$ has the form 
\[
0 \to S^2(-d) \to S (-a-1)  \oplus S(-b-1) \oplus  S(-d+1)  \to r_\cA \to 0. 
\]
Hence, applying \Cref{prop:Betti top from residual}, we get that the following graded minimal free resolution of $\Jac (f_\cA)^{sat}$:  
\[
0 \to S(-d+1-a) \oplus S(-d+1-a) \to S^3 (-d+1) \to \Jac (f_\cA)^{sat} \to 0,  
\]
where we used that $2d-1-a-1 = d-1+b$. In particular, $\Jac (f_\cA)^{sat}$ has a minimal generating set consisting of three forms of degree $d-1$. The same is true for $\Jac (f_\cA)$ because it is not a complete intersection. If follows that $\Jac (f_\cA)$  is equal to $\Jac (f_\cA)^{sat}$, that is, the line arrangement determined by $\cA$ is free, as claimed in (a). 
\smallskip

It remains to show that Conditions (b) and (c) are equivalent. Notice that any codimension two subscheme $X$ is arithmetically Gorenstein if and only if it is a complete intersection. Moreover, any subscheme of the scheme $Y$ defined by $I$ whose degree is equal to $-1 + \deg I$ is defined by one of the ideals $\fa_P$ for some $P \in S (\cA)$ or the general residual $r_\cA$. Hence, \Cref{thm:kreuzer} shows that Condition (b) implies (c). It also yields the reverse implication by taking into account \Cref{prop:add general point}, which gives 
\[
h_{S/I} (d-2) = h_{S/r_\cA} (d-2). 
\]
This completes the argument.  
\end{proof}

\begin{remark}
(i) Notice that the second assumption in Condition (c) of \Cref{thm:char freeness lines} is slightly weaker than assuming the Cayley-Bacharach property for the subscheme defined by $I$ as there is no assumption on the Hilbert function of $S/r_\cA$. 

(ii) All options for the integers $a, b$ in \Cref{thm:char freeness lines}(b) do occur for some line arrangement as the following example shows. 
\end{remark}

\begin{example}
    \label{exa:all options occur}
Fix any integers $a, b \ge 2$ with $a + b = d+1$. Let $\cA$ be an arrangement of two connected pencils, i.e., $\cA$ is   obtained from a union of two disconnect pencils consisting of $a -1$ and $b-1$ lines by adding the line through the centers of the pencils (see Figure~\ref{fig: two pencils, with line}).  In this case, the general residual $r_\cA$ is generated by forms of degrees $a, b$ and $a+b-2$. Hence, the ideal $I = r_\cA \cap I_{\ell^\vee}$ has minimal generators of degrees $a$ and $b$, and $\cA$ is free with exponents $a-1$ and $b-1$. 
\end{example}


\end{document}